\documentclass{article}
\usepackage{threeparttable}
\usepackage{multirow}
\usepackage[utf8]{inputenc}
\usepackage{hyperref}
\usepackage{amsmath}
\usepackage{amssymb}
\usepackage{graphicx}
\usepackage{amsthm}
\usepackage{float}
\usepackage{caption}
\usepackage{subcaption}
\usepackage{bm}
\usepackage{algorithm}  
\usepackage{algorithmic} 

\newtheorem{theorem}{Theorem}[section]
\newtheorem{lemma}[theorem]{Lemma}
\theoremstyle{definition}

\theoremstyle{remark}
\newtheorem{remark}[theorem]{Remark}
\numberwithin{equation}{section}
\title{A Locking-free and Loosely Coupled Robin-Robin Scheme for Fluid-Poroelasticity Interaction}
\author{
  Wenlong He\thanks{School of Mathematics and Statistics, Wuhan University, Wuhan 430072, China.}
  \and
  Thomas Wick\thanks{Institute of Applied Mathematics, Leibniz University Hannover, Welfengarten 1, 30167 Hannover, Germany. T.Wick and X.Yue are partially supported by DFG (German Research Foundation, No. 548064929).}
  \and
  Xiaohe Yue\thanks{Corresponding author. School of Mathematical Sciences, East China Normal University, Shanghai 200241, China; Institute of Applied Mathematics, Leibniz University Hannover, Welfengarten 1, 30167 Hannover, Germany. Email: yuexiaohe666@gmail.com.}
  \and
  Jiwei Zhang\thanks{School of Mathematics and Statistics, and Hubei Key Laboratory of Computational Science, Wuhan University, Wuhan 430072, China. J.Zhang is supported by NSFC (No. 12171376) and the Fundamental Research Funds for the Central Universities (No. 2042021kf0050) and WHU-2022-SYJS-0002.}
  \and
  Haibiao Zheng\thanks{School of Mathematical Sciences, Ministry of Education Key Laboratory of Mathematics and Engineering Applications, Shanghai Key Laboratory of PMMP, East China Normal University, Shanghai, 200241, China. H. Zheng is partially supported by NSFC (No.12471406) and Science and Technology Commission of Shanghai Municipality (No. 22DZ2229014).}
}

\date{}
\begin{document}
\maketitle
\vspace{1em}

\begin{abstract}
We study a fluid-poroelasticity interaction (FPSI) problem coupling the unsteady Stokes equations with the fully dynamic Biot system. A major challenge in such problems is to design partitioned schemes that remain robust in locking-related parameter regimes while preserving the physical interface coupling structure.
To address this issue, we introduce two auxiliary variables and reformulate the Biot system as a four-field problem consisting of a dynamic Stokes-like system coupled with a diffusion equation. Crucially, this reformulation preserves the original interface conditions. Based on Robin–Robin transmission conditions with explicitly lagged interface data, we construct a fully decoupled scheme in which the fluid and poroelastic subproblems can be solved independently and in parallel at each time step, without sub-iterations.
We prove that the resulting method is unconditionally stable and derive optimal-order error estimates in the $H^1$-norm. The analysis further shows that the scheme is robust with respect to extreme poroelastic parameters and avoids the locking effects inherent in standard formulations. Numerical experiments confirm the theoretical convergence results and demonstrate the locking-robust performance of the proposed method.
\end{abstract}

\section{Introduction}
	Fluid-poroelastic structure interaction (FPSI) problems have gained significant prominence due to their extensive applications across diverse fields, including geosciences, biomedical engineering, petroleum extraction, and the study of wave-vegetation interactions \cite{calo2008multiphysics,cuisiat1998petroleum,feenstra2009drug}. Poroelastic materials serve as a fundamental model for natural substances such as soil and biological tissues, as well as anthropogenic materials like cement. By establishing a coupled FPSI framework, one can capture the sophisticated interplay between free flow, filtrating interstitial flow, and the elastodynamics of the solid skeleton. However, the numerical solution of FPSI problems remains a formidable challenge, primarily due to the intricate interface coupling mechanisms, the structure on the governing equations, and the presence of physical parameters that span multiple orders of magnitude.

    In this work, we employ the time-dependent Stokes equations to characterize the free flow, while the poroelastic medium is described by the fully dynamic Biot system \cite{biot1941general, biot1955theory}. Within the poroelastic domain, the momentum conservation equation dictates the elastic deformation, and the Darcy equation 
    governs the seepage flow. The Biot system is coupled with the Stokes equations through physically consistent interface conditions, which facilitate the study of how free flow induces deformation and filtration within the poroelastic structure, and conversely how these processes modulate surrounding fluid field \cite{ambartsumyan2018lagrange, bukavc2015partitioning, guo2025fully,oyekole2020second, seboldt2021numerical}. From the perspective of multiphysics coupling, FPSI is particularly demanding as its interface conditions inherit the complexities of both the Stokes-Darcy problem (concerning fluid-filtration exchange)  \cite{discacciati2007robin, mu2010decoupled, vassilev2009coupling} and the classical fluid-structure interaction (FSI) problem (concerning stress and velocity continuity) \cite{guo2025fully,hou2012numerical, richter2017fluid,richter2010finite}. Designing a numerical scheme that simultaneously satisfies these diverse interface constraints while maintaining computational efficiency is a non-trivial task. Furthermore, the model is highly sensitive to parameter variations. In the Biot system, the well-known locking phenomenon--often triggered by specific ranges of physical parameters such as the incompressibility limit or low permeability--has been extensively studied in standalone scenarios \cite{oyarzua2016locking,phillips2009overcoming,yi2017study, zhao2025unconditionally}. However, the investigation of locking effects within the fully coupled FPSI  framework remains, to our knowledge, unexplored in the existing literature.

    Currently, numerical strategies for these coupled systems are broadly categorized into monolithic and partitioned methods. While monolithic solvers are robust in handling tight coupling, they are often computationally prohibitive and pose significant challenges for theoretical analysis \cite{ambartsumyan2018lagrange, bukavc2024computational, wen2020strongly}. Partitioned or decoupled methods, which solve the subproblems in a modular fashion, have become increasingly popular due to their flexibility in implementation and analysis. Among these, the Robin-Robin type domain decomposition has demonstrated remarkable robustness across a variety of multiphysics problems \cite{badia2009coupling, chen2011parallel, parrow2026stability, seboldt2021numerical}.
    Significant advancements have been made in the numerical treatment of FPSI problems.
    By introducing interface Lagrange multipliers, Li et al. \cite{li2024augmented} established the existence, uniqueness, and optimal error estimates for a fully mixed Navier-Stokes-Biot coupling that accounts for the Beavers–Joseph–Saffman condition on nonmatching grids. Quaini et al. \cite{cesmelioglu2016optimization} proposed a decoupling strategy based on a residual updating technique for the Stokes–Biot system, which solves the coupled problem through a least-squares optimization framework to ensure stable and efficient interface treatments. Buka\v c et al. \cite{seboldt2021numerical} proposed loosely coupled Robin-type schemes for moving domains, further splitting the Biot system into mechanics and Darcy subproblems. By examining the regularity of the Biot model and the algebraic structure of the three-field formulation, Yi \cite{yi2017study} identified the causes of numerical instabilities and developed new mixed finite elements that are robust against both Poisson locking and pressure oscillations. Oyarz\' ua \cite{oyarzua2016locking} introduced a new three-field formulation for the Biot system involving displacement, pressure, and pseudo total pressure, where the stability and error estimates were established independently of the Lam\'e constants using a Fredholm argument. 
    
    Despite these developments, parameter-robust decoupled algorithms that address locking within the fully coupled FPSI framework remain largely unexplored. In particular, for Robin--Robin type partitioned schemes, a central difficulty is how to introduce a locking-free reformulation of the poroelasticity equations without destroying the original interface coupling structure.
    
    Our previous work \cite{guo2025fully} developed a fully parallel Robin--Robin decoupling method for the standard Stokes--Biot system and established its unconditional stability. The present work goes substantially beyond that result. We introduce a four-variable reformulation of the fully dynamic Biot system by means of two auxiliary variables. This reformulation can be interpreted as a coupling between a dynamic Stokes-like system and a diffusion equation, and is specifically designed to improve robustness with respect to locking-related extreme parameters. More importantly, the reformulation preserves the original interface conditions without modification, which makes it possible to embed it into the FPSI coupling framework in a consistent way.
    
    Based on this observation, we construct Robin--Robin transmission conditions that yield a fully decoupled, parallel time-stepping scheme, where the fluid and poroelastic subproblems can be solved independently at each time step without sub-iterations. We further prove the equivalence between the reformulated decoupled system and the original coupled Stokes--Biot model under suitable interface data. In addition, we establish unconditional stability and derive optimal-order error estimates for the fully discrete scheme. The numerical experiments confirm the theoretical convergence results and demonstrate robust performance in parameter regimes associated with Poisson locking and early-time pressure oscillations.
    
    The main contributions of this work are threefold:
    \begin{itemize}
    \item we derive a locking-aware four-variable reformulation of the fully dynamic Biot system that preserves the original FPSI interface conditions;
    \item we design a fully decoupled Robin--Robin scheme for the coupled Stokes--Biot problem and prove its equivalence to the original coupled formulation;
    \item we establish unconditional stability and optimal error estimates, and validate the locking-robust behavior by numerical experiments.
    \end{itemize}

    The rest of this manuscript is organized as follows. In section 2, we introduce the governing equations, boundary and interface conditions for the fluid-poroelastic coupling system, alongside a discussion on the locking phenomenon inherent in the poroelastic model. Section 3 presents the derivation of the Robin-Robin interface conditions and a four-variable Biot system, which together yield the fully decoupled numerical scheme. In section 4, the unconditional stability of the proposed scheme is rigorously established. Furthermore, section 5 derives the optimal error estimate in the $H^{1}$-norm and demonstrates that the developed approach is locking-free. Finally, several numerical experiments are provided in section 6 to illustrate the convergence, robustness, and locking-free performance of our algorithm.
\section{Model}
 We study a coupled Stokes-Biot system that describes the interaction between a free fluid and a fluid-saturated, homogeneous, isotropic and linear elastic porous medium. Let $\Omega \subset \mathbb{R}^d$ ($d=2, 3$) be a bounded polygonal ($d=2$) or polyhedral ($d=3$) domain, partitioned into two non-overlapping subdomains: a fluid domain $\Omega_f$ and a poroelastic domain $\Omega_p$. The interface separating the two regions is denoted by $\Gamma = \partial\Omega_f \cap \partial\Omega_p$. The external boundary of the entire domain is decomposed as $\partial\Omega = \Gamma_f \cup \Gamma_p$, where $\Gamma_f = \partial\Omega_f \setminus \Gamma$ and $\Gamma_p = \partial\Omega_p \setminus \Gamma$. Furthermore, let $\mathbf{n}_{f}$ and $\mathbf{n}_{p}$ denote the outward unit normal vectors on $\partial\Omega_{f}$ and $\partial\Omega_{p}$, respectively. For any function $f(t)$, we denote the first and second derivatives by $\partial_{t}f$ and $\partial_{tt}^{2}f$ separately.
\subsection{Fluid equations}
 In the fluid region $\Omega_f$, the flow is governed by the unsteady Stokes equations:
	\begin{alignat}{2}
		\rho_{f}\partial_{t}\mathbf{v}_{f}-\nabla\cdot\bm{\sigma}_{f}(\mathbf{v}_{f},p_{f})&=\mathbf{f}_{f},&&\qquad \mbox{in } \Omega_{f}\times (0,T],\label{2.1}\\
		\nabla\cdot\mathbf{v}_{f}&=\phi_{f},&&\qquad \mbox{in } \Omega_{f}\times (0,T],\label{2.2}
	\end{alignat}
    where $\mathbf{v}_{f}$ and $p_{f}$ denote the fluid velocity and pressure, respectively. The fluid stress tensor $\bm{\sigma}_{f}$ and the deformation strain tensor $\varepsilon(\mathbf{v}_{f})$ are defined as:
	$$\bm{\sigma}_{f}(\mathbf{v}_{f},p_{f}) = 2\mu_{f}\varepsilon(\mathbf{v}_{f}) - p_{f}\mathbf{I} \quad \text{and} \quad \varepsilon(\mathbf{v}_{f}) = \frac{1}{2}(\nabla\mathbf{v}_{f} + \nabla^{\top}\mathbf{v}_{f}),$$
    where $\mathbf{I}$ represents the identity tensor. Here, $\rho_f > 0$ is the fluid density, $\mu_f > 0$ is the dynamic viscosity. The source terms $\mathbf{f}_f$ and $\phi_f$ represent the body force and the mass source, respectively.
    
	To close the system, we prescribe the initial and homogeneous boundary conditions for the Stokes equations as
	\begin{alignat}{2}
		\mathbf{v}_{f}&=\mathbf{0},&&\qquad \mbox{on } \Gamma_{f}^{D}\times (0,T],\label{2.3}\\
		\bm{\sigma}_{f}\mathbf{n}_{f}&=\mathbf{0},&&\qquad \mbox{on } \Gamma_{f}^{N}\times (0,T],\label{2.4}\\
		\mathbf{v}_{f}&=\mathbf{v}_{f,0}, &&\qquad \mbox{in } \Omega_{f}\times\{t=0\},\label{2.5}
	\end{alignat}
	where the external boundary $\Gamma_{f}$ is decomposed into $\Gamma_{f}^{D}$ and $\Gamma_{f}^{N}$, representing the Dirichlet (no-slip) and Neumann (traction-free) boundaries, respectively. We assume that $|\Gamma_{f}^{D}| > 0$ to ensure the well-posedness of the velocity field. Here, $\mathbf{v}_{f,0}$ denotes the given initial velocity distribution in $\Omega_f$.
\subsection{Poroelastic equations}
 	In the poroelastic region $\Omega_p$, the coupled flow and deformation are governed by the fully dynamic Biot system:
	\begin{alignat}{2}
	\rho_{p}\partial_{tt}^{2}\mathbf{u}_{p}-\nabla\cdot\bm{\sigma}_{p}(\mathbf{u}_{p},p_{p})&=\mathbf{f}_{p},&&\qquad \mbox{in } \Omega_{p}\times (0,T],\label{2.6}\\
		c_{0}\partial_{t}p_{p}+\alpha\partial_{t}\nabla\cdot\mathbf{u}_{p}-\nabla\cdot [\mu_{f}^{-1}K(\nabla p_{p}-\rho_{f}\mathbf{g})]&=\phi_{p},&&\qquad \mbox{in } \Omega_{p}\times (0,T],\label{2.7}
	\end{alignat}
where $\mathbf{u}_p$ and $p_p$ represent the solid displacement and the pore pressure, respectively. The total stress tensor $\bm{\sigma}_p$ is defined according to the effective stress principle:
$$\bm{\sigma}_{p}(\mathbf{u}_{p},p_{p}) = \bm{\sigma}_{e}(\mathbf{u}_{p}) - \alpha p_{p}\mathbf{I}, \quad \bm{\sigma}_{e}(\mathbf{u}_{p}) = 2\mu_{p}\varepsilon(\mathbf{u}_{p}) + \lambda_{p}(\nabla\cdot\mathbf{u}_{p})\mathbf{I},$$
where $\bm{\sigma}_e$ denotes the elastic effective stress tensor and $\varepsilon(\mathbf{u}_p) = \frac{1}{2}(\nabla\mathbf{u}_p + \nabla^{\top}\mathbf{u}_p)$ 
is the linearized strain tensor. The source terms are denoted by $\mathbf{f}_p$ and $\phi_p$, while $\mathbf{g}$ represents the gravitational acceleration. The permeability tensor $K(\mathbf{x})$ is assumed to be symmetric and uniformly positive definite, satisfying:
$$K_1 |\bm{\zeta}|^2 \le K(\mathbf{x})\bm{\zeta}\cdot \bm{\zeta} \le K_2 |\bm{\zeta}|^2, \quad \forall \bm{\zeta} \in \mathbb{R}^d, \text{ a.e. } \mathbf{x} \in \Omega_p,$$
for some positive constants $K_1$ and $K_2$. The physical coefficients $\rho_p$, $\alpha$, and $c_0 \ge 0$ denote the structure density, the Biot--Willis coefficient, and the constrained specific storage coefficient, respectively. Finally, the Lam\'e parameters $\lambda_p$ and $\mu_p$ are related to the Young's modulus $E$ and Poisson's ratio $\nu$ by:$$\lambda_p = \frac{E\nu}{(1+\nu)(1-2\nu)}, \qquad \mu_p = \frac{E}{2(1+\nu)}.$$

 To close the system \eqref{2.6}-\eqref{2.7}, we impose the following initial and homogeneous boundary conditions for the Biot equations:
	\begin{alignat}{2}
		\mathbf{u}_{p}&=\mathbf{0},&&\quad \mbox{on } \Gamma_{p}^{D}\times (0,T],\label{2.8}\\
		\bm{\sigma}_{p}\mathbf{n}_{p}&=\mathbf{0},&&\quad \mbox{on } \Gamma_{p}^{N}\times (0,T],\label{2.9}\\
		p_{p}&=0,&&\qquad \mbox{on } \tilde{\Gamma}_{p}^{D}\times (0,T],\label{2.10}\\
		-\mu_{f}^{-1}K(\nabla p_{p}-\rho_{f}\mathbf{g})\cdot\mathbf{n}_{p}&=0,&&\quad \mbox{on } \tilde{\Gamma}_{p}^{N}\times (0,T],\label{2.11}\\
		\partial_{t}\mathbf{u}_{p}=\mathbf{u}_{v,0},\quad\mathbf{u}_{p}=\mathbf{u}_{p,0},\quad p_{p}&=p_{p,0},&&\quad \mbox{in } \Omega_{p}\times\{t=0\},\label{2.12}
	\end{alignat}
	where the poroelastic boundary $\Gamma_{p}$ is decomposed as $\Gamma_{p} = \Gamma_{p}^{D}\cup\Gamma_{p}^{N} = \tilde{\Gamma}_{p}^{D}\cup\tilde{\Gamma}_{p}^{N}$. Here, $\Gamma_{p}^{D}$ and $\Gamma_{p}^{N}$ represent the Dirichlet and Neumann boundaries for the solid displacement, while $\tilde{\Gamma}_{p}^{D}$ and $\tilde{\Gamma}_{p}^{N}$ correspond to the drainage and no-flow boundaries for the pore pressure, respectively. For simplicity, we assume $\Gamma_{p}^{D} = \tilde{\Gamma}_{p}^{D}$ in the following. The initial state is defined by the given functions $\mathbf{u}_{p,0}$, $\mathbf{v}_{p,0}$, and $p_{p,0}$.
\subsection{Coupling conditions}
	To couple the fluid model and the Biot model, we impose a set of physically consistent interface conditions on the fluid–poroelastic interface $\Gamma$ for $t \in (0, T]$. These conditions include mass conservation, balance of normal stress, and the Beavers-Joseph-Saffman (BJS) condition to account for the tangential slip \cite{mikelic2000interface}:
	\begin{alignat}{2}
		\mathbf{v}_{f}\cdot\mathbf{n}_{f}&=-(\partial_{t}\mathbf{u}_{p}-\mu_{f}^{-1}K(\nabla p_{p}-\rho_{f}\mathbf{g}))\cdot\mathbf{n}_{p},&&\qquad \mbox{on } \Gamma\times (0,T],\label{2.13}\\
		\bm{\sigma}_{f}\mathbf{n}_{f}\cdot\mathbf{n}_{f}&=-p_{p},&&\qquad \mbox{on } \Gamma\times (0,T],\label{2.14}\\
		\bm{\sigma}_{f}\mathbf{n}_{f}+\bm{\sigma}_{p}\mathbf{n}_{p}&=\mathbf{0},&&\qquad \mbox{on } \Gamma\times (0,T],\label{2.15}\\
		-\mu_{f}\gamma\sqrt{K_{j}^{-1}}(\mathbf{v}_{f}-\partial_{t}\mathbf{u}_{p})\cdot\bm{\tau}_{f,j}&=\bm{\sigma}_{f}\mathbf{n}_{f}\cdot\bm{\tau}_{f,j}~\text{for}~ j = 1,\cdots, d - 1,&&\qquad \mbox{on } \Gamma\times (0,T].\label{2.16}
	\end{alignat}
Here, equation \eqref{2.13} represents the conservation of mass, ensuring the continuity of normal flux across the interface. Equation \eqref{2.14} enforces the balance of normal force, while equation \eqref{2.15} represents the equilibrium of total stress. The BJS condition \eqref{2.16} describes the tangential velocity jump, where ${\bm{\tau}_{f,j}}$ denotes an orthonormal set of unit tangent vectors on $\Gamma$. The parameter $K_{j}$ is the component of the permeability tensor in the $j$-th tangential direction, and $\gamma > 0$ is a dimensionless friction coefficient determined experimentally.
\begin{remark}
       It is known that the Biot system, when discretized in space using continuous Galerkin finite element methods, exhibits two types of locking phenomena induced by extreme values of physical parameters: 

       (i) Poisson locking \cite{phillips2009overcoming,yi2017study} arises in the incompressible limit $\lambda_p \to \infty$, where the displacement field satisfies $\nabla \cdot \mathbf{u} \to 0$. Low-order conforming elements, such as linear triangles or bilinear quadrilaterals, possess very few divergence-free modes. For example, in 2D bilinear elements, enforcing zero divergence and continuity reduces the effective degrees of freedom to constants. Consequently, nonconstant boundary deformations cannot be captured, causing overly stiff numerical responses and potential nonphysical oscillations in the stress field, highlighting a major challenge for standard finite element discretizations under near-incompressibility.

       (ii) Pressure oscillations represent a more severe form of locking and typically arise at early times under specific parameter regimes. As shown by Phillips and Wheeler \cite{phillips2009overcoming}, when $c_0=0$, the permeability is very small, and small time steps are used, the discrete flow equation enforces an almost divergence-free displacement, leading to nonphysical pressure oscillations. From an algebraic perspective, Yi \cite{yi2017study} further attributed this phenomenon to the incompatibility between the displacement and pressure spaces, which renders the discrete system ill-posed up to spurious pressure modes.
\end{remark}

\section{Numerical method}
In this section, we develop a fully decoupled and locking-free numerical scheme for the Stokes--Biot system. We begin with deriving a four-variable formulation for the Biot equations, which is specifically designed to possess the potential to overcome locking phenomena in nearly incompressible poroelastic media. After that, we detail the construction of the Robin--Robin type interface conditions, which serve as the foundation for decomposing the global coupled system into independent subproblems. By leveraging these interface conditions, the system is fully decoupled into fluid and poroelastic subsystems. Finally, a fully discrete decoupled algorithm is presented, ensuring both computational efficiency and physical accuracy.
\subsection{Four-variable formulation for the Biot equations}
 To reveal the multiphysics process of the Biot model and to construct an intrinsic mechanism that circumvents locking phenomena, we introduce two auxiliary variables $\mathbf{v}_{p}=\partial_{t}\mathbf{u}_{p}$ and $\beta_{p}=\alpha p_{p}-\lambda_{p}\nabla\cdot\mathbf{u}_{p}$ and reformulate the Biot model \eqref{2.6}-\eqref{2.7} into
\begin{alignat}{2}
\mathbf{v}_{p}-\partial_{t}\mathbf{u}_{p}&=\mathbf{0},&&~~ \mbox{in } \Omega_{p}\times (0,T],\label{2.17}\\
\rho_{p}\partial_{t}\mathbf{v}_{p}-2\mu_{p}\nabla\cdot\varepsilon(\mathbf{u}_{p})+\nabla\beta_{p}&=\mathbf{f}_{p},&&~~ \mbox{in } \Omega_{p}\times (0,T],\label{2.18}\\
\frac{1}{\lambda}_{p}\beta_{p}+\nabla\cdot\mathbf{u}_{p}&=\frac{\alpha}{\lambda}_{p}p_{p},&&~~ \mbox{in } \Omega_{p}\times (0,T],\label{2.19}\\	
 \big(c_{0}+\frac{\alpha^{2}}{\lambda_{p}}\big)\partial_{t}p_{p}-\frac{\alpha}{\lambda_{p}}\partial_{t}\beta_{p}-\nabla\cdot [\mu_{f}^{-1}K(\nabla p_{p}-\rho_{f}\mathbf{g})]&=\phi_{p},&&~~ \mbox{in } \Omega_{p}\times (0,T].\label{2.20}
\end{alignat}
\begin{remark}
The reformulated system provides a natural mechanism for alleviating the two locking-related difficulties discussed in Remark 2.1.

(i) By introducing the total-pressure-type variable $\beta_p$, the volumetric constraint is separated from the ill-conditioned displacement equation. In the nearly incompressible regime $\lambda_p \to \infty$, the resulting structure resembles a generalized Stokes system with an explicit incompressibility constraint, which explains why mixed finite element pairs satisfying the inf--sup condition are expected to behave more robustly with respect to Poisson locking.

(ii) For parameter regimes with $c_0=0$, small permeability, and small time step size, the reformulated pressure equation no longer enforces an approximately divergence-free displacement in the same way as the standard formulation. This observation helps explain why the present formulation is less prone to spurious early-time pressure oscillations.
\end{remark}

\subsection{Robin--Robin type interface conditions}
Let $\mathcal{M}_{*}(\mathbf{v})=\sum_{j=1}^{d-1}\mathbf{v}\cdot\bm{\tau}_{*,j}$ denote the tangential projection of the velocity $\mathbf{v}$ onto the interface $\Gamma$, where the subscript $* \in \{f, p\}$ identifies the fluid and poroelasticity medium phase, respectively. Furthermore, we define $c_{BJS} = \mu_{f}\gamma\sqrt{K^{-1}}$ as the Beavers--Joseph--Saffman (BJS) slip coefficient. To facilitate the decoupling of the Stokes-Biot system, we introduce Robin-type transmission conditions on the interface $\Gamma$. Specifically, given the artificial parameters $L_{1}, L_{2}$ and $L_{3} > 0$, we assume that the fluid subproblem satisfies the following Robin-type conditions on the interface $\Gamma$:
\begin{align}
&L_{1}\mathbf{v}_{f}\cdot\mathbf{n}_{f}+\bm{\sigma}_{f}\mathbf{n}_{f}\cdot\mathbf{n}_{f}
= R_{1},\\
 &c_{BJS}\mathcal{M}_{f}(\mathbf{v}_{f})+\mathcal{M}_{f}(\bm{\sigma}_{f}\mathbf{n}_{f})
 = -R_{2}.
 \end{align}
Simultaneously, the following Robin-type boundary conditions are constructed for the Biot subproblem: 
\begin{align}
&L_{2}\partial_{t}\mathbf{u}_{p}\cdot\mathbf{n}_{p}+\bm{\sigma}_{p}\mathbf{n}_{p}\cdot\mathbf{n}_{p}= R_{3},\\
&c_{BJS}\mathcal{M}_{p}(\partial_{t}\mathbf{u}_{p})+\mathcal{M}_{p}(\bm{\sigma}_{p}\mathbf{n}_{p})
= -R_{4},\\
&L_{3}p_{p}+\mu_{f}^{-1}K\big(\nabla p_{p}-\rho_{f}\mathbf{g})\cdot\mathbf{n}_{p}
= R_{5},\label{2.25}
\end{align}
where $R_{i} (i = 1,2,\cdots,5$) are some known functions defined on the interface 
$\Gamma$.

The standard function space notation is adopted, with precise definitions provided in \cite{brenner2008mathematical, ciarlet2002finite}. For example, the standard inner products on $L^2(\Omega)$ and $L^2(\partial \Omega)$ are denoted by $(\cdot, \cdot)$ and $\langle \cdot, \cdot \rangle$, respectively. For any Banach space $B$, we define $\mathbf{B} = [B]^d$ and denote $\mathbf{B}'$ by its dual space. Before starting the weak formulation, we define the following function spaces for the fluid and poroelastic subproblems:
\begin{align*}
&\mathbf{V}_{f}:=\{\mathbf{w}_{f}\in \mathbf{H}^{1}(\Omega_{f}):\mathbf{w}_{f}|_{\Gamma_{f}^{D}}=\mathbf{0}\},\qquad M_{f}:=L^{2}(\Omega_{f}),\\ 
&\mathbf{V}_{p}:=\{\mathbf{w}_{p}\in \mathbf{H}^{1}(\Omega_{p}):\mathbf{w}_{p}|_{\Gamma_{p}^{D}}=\mathbf{0}\},\qquad M_{p}:=\{q_{p}\in L^{2}(\Omega_{p}):q_{p}|_{\tilde{\Gamma}_{p}^{D}}=0\},\\
&W_{p}:=\{q_{p}\in H^{1}(\Omega_{p}):q_{p}|_{\Gamma_{p}^{D}}=0\}.
\end{align*}
Based on the reformulated Biot system and the introduced Robin conditions, the weak form of the Stokes problem reads: 
For any $(\mathbf{w}_{f},q_{f})\in\mathbf{V}_{f}\times M_{f}$, find $(\mathbf{v}_{f},p_{f})\in C^{1}([0,T];\mathbf{V}_{f})\times C^{0}([0,T];M_{f})$ 
such that:
\begin{align}
&\rho_{f}\big(\partial_{t}\mathbf{v}_{f},\mathbf{w}_{f}\big)_{\Omega_{f}}
+2\mu_{f}\big(\varepsilon(\mathbf{v}_{f}),\varepsilon(\mathbf{w}_{f})\big)_{\Omega_{f}}
- \big(p_{f},\nabla\cdot\mathbf{w}_{f}\big)_{\Omega_{f}} \label{3.1}\\
&\quad+L_{1}\langle\mathbf{v}_{f}\cdot\mathbf{n}_{f},\mathbf{w}_{f}\cdot\mathbf{n}_{f}\rangle_{\Gamma}+c_{BJS}\langle \mathcal{M}_{f}(\mathbf{v}_{f}),\mathcal{M}_{f}(\mathbf{w}_{f})\rangle_{\Gamma}\nonumber\\
&=\big(\mathbf{f}_{f},\mathbf{w}_{f}\big)_{\Omega_{f}}
+\langle R_{1},\mathbf{w}_{f}\cdot\mathbf{n}_{f}\rangle_{\Gamma}
-\langle R_{2},\mathcal{M}_{f}(\mathbf{w}_{f})\rangle_{\Gamma},\nonumber\\ &\big(\nabla\cdot\mathbf{v}_{f},q_{f}\big)_{\Omega_{f}}=\big(\phi_{f},q_{f}\big)_{\Omega_{f}}.\label{3.2}
\end{align}
Similarly on the Biot system: For any $(\mathbf{z}_{p},\mathbf{w}_{p},\varphi_{p},\psi_{p})\in\mathbf{V}_{p}\times\mathbf{V}_{p}\times M_{p}\times W_{p}$, find $(\mathbf{v}_{p},\mathbf{u}_{p},\beta_{p},p_{p})\in C^{1}([0,T];\mathbf{V}_{p})\times C^{1}([0,T];\mathbf{V}_{p})\times C^{1}([0,T];M_{p})\times C^{1}([0,T];W_{p})$ 
such that:
\begin{align}
&\big(\mathbf{v}_{p},\mathbf{z}_{p}\big)_{\Omega_{p}}-\big(\partial_{t}\mathbf{u}_{p},\mathbf{z}_{p}\big)_{\Omega_{p}}=0,\label{3.3}\\
&\rho_{p}\big(\partial_{t}\mathbf{v}_{p},\mathbf{w}_{p}\big)_{\Omega_{p}}+2\mu_{p}\big(\varepsilon(\mathbf{u}_{p}),\varepsilon(\mathbf{w}_{p})\big)_{\Omega_{p}}-\big(\beta_{p},\nabla\cdot\mathbf{w}_{p}\big)_{\Omega_{p}}\label{3.4}\\
&\quad+L_{2}\langle\partial_{t}\mathbf{u}_{p}\cdot\mathbf{n}_{p},\mathbf{w}_{p}\cdot\mathbf{n}_{p}\rangle_{\Gamma}+c_{BJS}\langle\mathcal{M}_{p}(\partial_{t}\mathbf{u}_{p}),\mathcal{M}_{p}(\mathbf{w}_{p})\rangle_{\Gamma}\nonumber\\
&=\big(\mathbf{f}_{p},\mathbf{w}_{p}\big)_{\Omega_{p}}
+\langle R_{3},\mathbf{w}_{p}\cdot\mathbf{n}_{p}\rangle_{\Gamma}
-\langle R_{4},\mathcal{M}_{p}(\mathbf{w}_{p})\rangle_{\Gamma},\nonumber\\
&\frac{1}{\lambda_{p}}\big(\beta_{p},\varphi_{p}\big)_{\Omega_{p}}+\big(\nabla\cdot\mathbf{u}_{p},\varphi_{p}\big)_{\Omega_{p}}=\frac{\alpha}{\lambda_{p}}\big(p_{p},\varphi_{p}\big)_{\Omega_{p}},\label{3.5}\\
&\big(c_{0}+\frac{\alpha^{2}}{\lambda_{p}}\big)\big(\partial_{t}p_{p},\psi_{p}\big)_{\Omega_{p}}-\frac{\alpha}{\lambda_{p}}\big(\partial_{t}\beta_{p},\psi_{p}\big)_{\Omega_{p}}\label{3.6}\\
&\quad +\big(\mu_{f}^{-1}K(\nabla p_{p}-\rho_{f}\mathbf{g}),\nabla\psi_{p}\big)_{\Omega_{p}}+L_{3}\langle p_{p},\psi_{p}\rangle_{\Gamma}\nonumber\\
&=\big(\phi_{p},\psi_{p}\big)_{\Omega_{p}}+\langle R_{5},\psi_{p}\rangle_{\Gamma}\nonumber.
\end{align}

By incorporating the Robin type conditions into the fluid and poroelastic problems, the coupled system is decomposed into two independent subproblems. 
The key point is that the four-variable reformulation does not modify the original FPSI interface structure; therefore the Robin–Robin coupling can be derived consistently at the continuous level.
In the following, we present a theorem to demonstrate that the solution of the decoupled system is equivalent to that of the original coupled Stokes-Biot model after a specific choice for $R_{i} (i = 1,2,\cdots,5)$.
\begin{theorem}\label{equivalence}
    Let $(\mathbf{v}_{f},p_{f},\mathbf{u}_{p}, p_{p})$ be the weak solution of the original coupled Stokes-Biot system \eqref{2.1}--\eqref{2.16} and let 
    $(\mathbf{v}_{f,r},p_{f, r},\mathbf{v}_{p, r},\mathbf{u}_{p, r},\beta_{p, r},p_{p, r})$
    be the solution of the decoupled system \eqref{3.1} -- \eqref{3.6} with deduced Robin type boundary conditions at the interface. Then 
    $$
    \mathbf{v}_{f,r} = \mathbf{v}_{f},\ p_{f, r} = p_{f},\ \mathbf{u}_{p, r} = \mathbf{u}_{p},
    \ p_{p, r} = p_{p},
    $$
    if and only if  $R_{i} (i = 1,2,\cdots,5)$ satisfy the following compatibility conditions:
    \begin{align*}
        &R_{1} = L_{1}\mathbf{v}_{f}\cdot\mathbf{n}_{f} - p_{p}
        ,\quad &&R_{2}= c_{BJS}\mathcal{M}_{p}(\partial_{t}\mathbf{u}_{p}) , \\
        &R_{3} = L_{2}\partial_{t}\mathbf{u}_{p}\cdot\mathbf{n}_{p}-p_{p}
        ,\quad &&R_{4} = c_{BJS}\mathcal{M}_{f}(\mathbf{v}_{f})
        , \quad R_{5} = L_{3}p_{p}+\mathbf{v}_{f}\cdot\mathbf{n}_{f}+\partial_{t}\mathbf{u}_{p}\cdot\mathbf{n}_{p} .
    \end{align*}
\end{theorem}
\begin{remark}

The proof of Theorem \ref{equivalence} follows a similar argument as presented in our previous work \cite{guo2025fully}. We point out that the primary distinction here is the auxiliary variable $\beta_{p}$ within the Biot subsystem. However, since $\beta_{p}$ does not alter the interface conditions, the fundamental process of the equivalence proof remains unchanged.
\end{remark}

Moreover, the well-posedness analysis of the coupled system under the Robin-type interface conditions is essential for ensuring the reliability of its numerical discretization. To rigorously characterize the existence and uniqueness of the solution, we now present the following theorem.
\begin{theorem}\label{thm3.2}
    The solution to problem \eqref{3.1}-\eqref{3.6} exists uniquely.
\end{theorem}
\begin{proof}
    Appendix \ref{app:proof}.
\end{proof}
\subsection{Fully decoupled  and locking-free numerical scheme}
 Let $\mathcal{T}_{f,h}$ and $\mathcal{T}_{p,h}$ be the 
quasi-uniform triangulation or rectangular partitions of $\Omega_{f}$ and $\Omega_{p}$ with maximum mesh size $h$, and $\bar{\Omega}_{f}=\bigcup_{\mathcal{K}_{f}\in\mathcal{T}_{f,h}}\bar{\mathcal{K}_{f}}$, $\bar{\Omega}_{p}=\bigcup_{\mathcal{K}_{p}\in\mathcal{T}_{f,h}}\bar{\mathcal{K}_{p}}$. We assume that the partitions $\mathcal{T}_{f,h}$ and $\mathcal{T}_{p,h}$ are compatible on $\Gamma$; i.e., they share the same edges (if $d=2$) or faces (if $d=3$) therein. The family of partitions induced on $\Gamma$ will be denoted by $\mathcal{T}_{h}$. The time interval $[0, T]$ is divided as  $N$ equal intervals, denoted by $[t_{n-1}, t_{n}], n=1,2,\cdots,N$,  and $\Delta t=T/N$, then $t_n=n\Delta t$. In this work, we use the backward Euler method to the Stokes-Biot system in time. Moreover, we define the following finite element spaces (cf. \cite{brezzi2012mixed}) as:
\begin{align*}
	\mathbf{V}_{f,h}&=
	\bigl\{\mathbf{w}_{f,h}\in\mathbf{V}_{f},\mathbf{w}_{f,h}\in \mathbf{C}^0(\overline{\Omega}_{f});\,
	\mathbf{w}_{f,h}|_{\mathcal{K}_{f}}\in \mathbf{P}_{2}(\mathcal{K}_{f})~~\forall \mathcal{K}_{f}\in \mathcal{T}_{f,h} \bigr\},\\
	M_{f,h} &=\bigl\{q_{f,h}\in M_{f}, q_{f,h}\in C^0(\overline{\Omega}_{f});\, q_{f,h}|_{\mathcal{K}_{f}}\in P_{1}(\mathcal{K}_{f})
	~~\forall \mathcal{K}_{f}\in \mathcal{T}_{f,h} \bigr\},\\
	\mathbf{V}_{p,h}&=
	\bigl\{\mathbf{w}_{p,h}\in\mathbf{V}_{p},\mathbf{w}_{p,h}\in \mathbf{C}^0(\overline{\Omega}_{p});\,
	\mathbf{w}_{p,h}|_{\mathcal{K}_{p}}\in \mathbf{P}_{2}(\mathcal{K}_{p})~~\forall \mathcal{K}_{p}\in \mathcal{T}_{p,h} \bigr\},\\
	M_{p,h} &=\bigl\{q_{p,h}\in W_{p}, q_{p,h}\in C^0(\overline{\Omega}_{p});\, q_{p,h}|_{\mathcal{K}_{p}}\in P_{1}(\mathcal{K}_{p})
	~~\forall \mathcal{K}_{p}\in \mathcal{T}_{p,h} \bigr\}, \\
 W_{p,h} &=\bigl\{q_{p,h}\in M_{p}, q_{p,h}\in C^0(\overline{\Omega}_{p});\, q_{p,h}|_{\mathcal{K}_{p}}\in P_{2}(\mathcal{K}_{p})
	~~\forall \mathcal{K}_{p}\in \mathcal{T}_{p,h} \bigr\}, 
\end{align*}
where $P_{k}(\mathcal{K_{*}}),~*=f,p$ is the space of polynomials of degree $k$ on $\mathcal{K}_{*}$. From \cite{brezzi2012mixed}, it is easy to check that $(\mathbf{w}_{*,h},q_{*,h})\in(\mathbf{V}_{*,h},M_{*,h})$ satisfies the inf-sup condition.

Next, we present the locking-free and loosely coupled Robin-Robin scheme for the Stokes-Biot model. Within the proposed computational framework, the Stokes equations in the fluid domain $\Omega_{f}$ are solved using the auxiliary variables $R_{1}$ and $R_{2}$ obtained from the previous time step. Specifically, at each time level, the updates of the fluid velocity and pressure rely solely on the values of $R_{1}$ and $R_{2}$ from the preceding step, without requiring information from the current-time-step unknowns in the porous medium. Similarly, in the poroelastic domain $\Omega_{p}$, the Biot equations are solved based on the values of $R_{3}$, $R_{4}$, and $R_{5}$ from the previous step. In this case as well, the updates of solid displacement and pore pressure depend exclusively on these past-step quantities. Since the solution of each system does not involve the current-step unknowns of the other, the Stokes and Biot equations can be solved independently and concurrently within the same time step. This decoupling strategy significantly enhances computational efficiency, providing a scalable approach for large-scale fluid–structure interaction problems.
\begin{algorithm}[htbp]
	\caption{Fully decoupled and locking-free scheme for the Stokes-Biot problem}
    \label{al1}
	\begin{itemize}
		\item[(i)]
		Compute $\mathbf{v}^0_{f,h},~\mathbf{v}_{p,h}^{0},~\mathbf{u}_{p,h}^{0},~p_{p,h}^{0}$ and
		 $R_{i,h}^{0}(i=1, 2\cdots 5)$ as:
		\begin{align*}
			&\mathbf{v}^0_{f,h}=\mathcal{Q}_{h}\mathbf{v}_{f,0},\qquad \mathbf{v}_{p,h}^{0}=\mathcal{Q}_{h}\mathbf{v}_{p}^{0},\qquad \mathbf{u}_{p,h}^{0}=\mathcal{Q}_{h}\mathbf{u}_{p}^{0},\qquad p_{p,h}^{0}=\mathcal{R}_{h}p_{p}^{0},\\
			&R_{1,h}^{0},= L_{1}\mathbf{v}_{f,h}^{0}\cdot\mathbf{n}_{f}-p_{p,h}^{0},\quad R_{2,h}^{0}=c_{BJS}\mathcal{M}_{p}(\mathbf{v}_{p,h}^{0}),\\
           &R_{3,h}^{0}= L_{2}\mathbf{v}_{p,h}^{0}\cdot\mathbf{n}_{p}-p_{p,h}^{0},\quad R_{4,h}^{0}=c_{BJS}\mathcal{M}_{f}(\mathbf{v}_{f,h}^{0}),\\
           &R_{5,h}^{0}=L_{3}p_{p,h}^{0}+ \mathbf{v}_{f,h}^{0}\cdot\mathbf{n}_{f}+\mathbf{v}_{p,h}^{0}\cdot\mathbf{n}_{p}.
		\end{align*}

	     \item[(ii)] \textbf{Fluid subproblem}: For any $(\mathbf{w}_{f,h}, q_{f,h})\in \mathbf{V}_{f,h} \times M_{f,h}$ and $n\geq0$, solve for $(\mathbf{v}^{n+1}_{f,h}, p_{f,h}^{n+1})\in \mathbf{V}_{f,h} \times M_{f,h}$ such that
		\begin{align}
		&\rho_{f}\big(d_{t}\mathbf{v}_{f,h}^{n+1},\mathbf{w}_{f,h}\big)_{\Omega_{f}}+2\mu_{f}\big(\varepsilon(\mathbf{v}_{f,h}^{n+1}),\varepsilon(\mathbf{w}_{f,h})\big)_{\Omega_{f}}-\big(p_{f,h}^{n+1},\nabla\cdot\mathbf{w}_{f,h}\big)_{\Omega_{f}}\label{4.2}\\
          &\quad+L_{1}\langle\mathbf{v}_{f,h}^{n+1}\cdot\mathbf{n}_{f},\mathbf{w}_{f,h}\cdot\mathbf{n}_{f}\rangle_{\Gamma} 
          + c_{BJS}\langle\mathcal{M}_{f}(\mathbf{v}_{f,h}^{n+1}),\mathcal{M}_{f}(\mathbf{w}_{f,h})\rangle_{\Gamma}\nonumber\\
          &=\big(\mathbf{f}_{f}^{n+1},\mathbf{w}_{f,h}\big)_{\Omega_{f}}+\langle R_{1,h}^{n},\mathbf{w}_{f,h}\cdot\mathbf{n}_{f}\rangle_{\Gamma}-\langle R_{2,h}^{n},\mathcal{M}_{f}(\mathbf{w}_{f,h})\rangle_{\Gamma},\nonumber\\
          &\big(\nabla\cdot\mathbf{v}_{f,h}^{n+1},q_{f,h}\big)_{\Omega_{f}}=\big(\phi_{f}^{n+1},q_{f,h}\big)_{\Omega_{f}}.\label{4.3}
		\end{align}
  
		\item[(iii)] \textbf{Poroelastic subproblem}: For any $(\mathbf{z}_{p,h},\mathbf{w}_{p,h},\varphi_{p,h},\psi_{p,h})\in\mathbf{V}_{p,h}\times\mathbf{V}_{p,h}\times M_{p,h}\times W_{p,h}$, solve for $(\mathbf{v}_{p,h}^{n+1},\mathbf{u}_{p,h}^{n+1},\beta_{p,h}^{n+1},p_{p,h}^{n+1})\in\mathbf{V}_{p,h}\times\mathbf{V}_{p,h}\times M_{p,h}\times W_{p,h}$ such that
		\begin{align}
		&\big(\mathbf{v}_{p,h}^{n+1},\mathbf{z}_{p,h}\big)_{\Omega_{p}}-\big(d_{t}\mathbf{u}_{p,h}^{n+1},\mathbf{z}_{p,h}\big)_{\Omega_{p}}=0,\label{4.4}\\		&\rho_{p}\big(d_{t}\mathbf{v}_{p,h}^{n+1},\mathbf{w}_{p,h}\big)_{\Omega_{p}}+2\mu_{p}\big(\varepsilon(\mathbf{u}_{p,h}^{n+1}),\varepsilon(\mathbf{w}_{p,h})\big)_{\Omega_{p}}-\big(\beta_{p,h}^{n+1},\nabla\cdot\mathbf{w}_{p,h}\big)_{\Omega_{p}}\label{4.5}\\
          &\quad+L_{2}\langle d_{t}\mathbf{u}_{p,h}^{n+1}\cdot\mathbf{n}_{p},\mathbf{w}_{p,h}\cdot\mathbf{n}_{p}\rangle_{\Gamma}+c_{BJS}\langle\mathcal{M}_{p}(d_{t}\mathbf{u}_{p,h}^{n+1}),\mathcal{M}_{p}(\mathbf{w}_{p,h})\rangle_{\Gamma}\nonumber\\
          &=\big(\mathbf{f}_{p}^{n+1},\mathbf{w}_{p,h}\big)_{\Omega_{p}}+\langle R_{3,h}^{n},\mathbf{w}_{p,h}\cdot\mathbf{n}_{p}\rangle_{\Gamma}-\langle R_{4,h}^{n},\mathcal{M}_{p}(\mathbf{w}_{p,h})\rangle_{\Gamma},\nonumber\\
          &\frac{1}{\lambda_{p}}\big(\beta_{p,h}^{n+1},\varphi_{p,h}\big)_{\Omega_{p}}+\big(\nabla\cdot\mathbf{u}_{p,h}^{n+1},\varphi_{p,h}\big)_{\Omega_{p}}=\frac{\alpha}{\lambda_{p}}\big(p_{p,h}^{n+1},\varphi_{p,h}\big)_{\Omega_{p}},\label{4.6}\\
          &\big(c_{0}+\frac{\alpha^{2}}{\lambda_{p}}\big)\big(d_{t}p_{p,h}^{n+1},\psi_{p,h}\big)_{\Omega_{p}}-\frac{\alpha}{\lambda_{p}}\big(d_{t}\beta_{p,h}^{n+1},\psi_{p,h}\big)_{\Omega_{p}}\label{4.7}\\
          &\quad+\big(\mu_{f}^{-1}K(\nabla p_{p,h}^{n+1}-\rho_{f}\mathbf{g}),\nabla\psi_{p,h}\big)_{\Omega_{p}}+L_{3}\langle p_{p,h}^{n+1},\psi_{p,h}\rangle_{\Gamma}\nonumber\\
          &=\big(\phi_{p}^{n+1},\psi_{p,h}\big)_{\Omega_{p}}+\langle R_{5,h}^{n},\psi_{p,h}\rangle_{\Gamma}.\nonumber
		\end{align}
		\item[(iv)] Update $R_{i,h}^{n+1}$ by
		\begin{align}
             &R_{1,h}^{n+1},= L_{1}\mathbf{v}_{f,h}^{n+1}\cdot\mathbf{n}_{f}-p_{p,h}^{n+1},\label{4.8}\\
             &R_{2,h}^{n+1}=c_{BJS}\mathcal{M}_{p}(d_{t}\mathbf{u}_{p,h}^{n+1}),\label{4.9}\\
           &R_{3,h}^{n+1}= L_{2}d_{t}\mathbf{u}_{p,h}^{n+1}\cdot\mathbf{n}_{p}-p_{p,h}^{n+1},\label{4.10}\\
           &R_{4,h}^{n+1}=c_{BJS}\mathcal{M}_{f}(\mathbf{v}_{f,h}^{n+1}),\label{4.11}\\
           &R_{5,h}^{n+1}=L_{3}p_{p,h}^{n+1}+ \mathbf{v}_{f,h}^{n+1}\cdot\mathbf{n}_{f}+d_{t}\mathbf{u}_{p,h}^{n+1}\cdot\mathbf{n}_{p}.\label{4.12}
           \end{align}
	\end{itemize}
\end{algorithm}
\begin{remark}
For the discretization of the Biot system, we primarily employ the $\mathbf{P}_2$--$P_1$ finite element pairs to solve equations \eqref{4.5}-\eqref{4.6}, as this choice satisfies the inf--sup stability condition in a natural and robust manner. The proposed formulation is not restricted to this particular pair and can be readily extended to other stable combinations, such as the $\mathbf{P}_3$--$P_2$ elements or the MINI element. These alternatives provide additional flexibility in balancing accuracy and computational cost. It is worth noting that, under certain conditions, equations \eqref{4.5}-\eqref{4.6} can also be approximated using $\mathbf{P}_{1}-P_{1}$ elements without introducing additional stabilization; further details on this approach can be found in \cite{di2025optimal}.

Moreover, by lagging the pressure variable $p_p$ in time, namely by replacing the superscript $n+1$ with $n$, the fully coupled Biot system can be decomposed into two subproblems: a generalized Stokes problem for the displacement and total pressure, and a diffusion problem for the pore pressure. This splitting strategy preserves the essential coupling mechanisms while significantly reducing computational complexity. In particular, it enables the two subproblems to be solved independently, thereby facilitating parallel computations in the poroelastic domain and improving overall computational efficiency.
\end{remark}
\section{Stability analysis}
To establish the stability of the proposed scheme, we first recall the useful identity:
 \begin{align}
 2a(a-b)=a^{2}-b^{2}+(a-b)^{2}.\label{5.1}
\end{align}
For the sake of brevity, we assume the absence of external forcing terms in the system, i.e.,
$\mathbf{f}_{f}=\mathbf{f}_{p}=\mathbf{g}=\mathbf{0}$ and $\phi_{f}=\phi_{p}=0$. We remark that the subsequent results can be extended to the non-homogeneous case in a straightforward manner by employing standard techniques, such as the Cauchy-Schwarz and Young inequalities.
For conciseness, these details are omitted here. Moreover, we introduce $C,~\widehat{C}$ and $\check{C}$ as the positive constants related to the model parameters and common inequalities, such as Poincar\'e inequality, Korn inequality. 

By denoting the accumulated energy as:
\begin{align*}
\mathcal{E}_{h}^{n+1}=&\frac{\rho_{f}}{2}\|\mathbf{v}_{f,h}^{n+1}\|_{L^{2}(\Omega_{f})}^{2}+\frac{\rho_{p}}{2}\|\mathbf{v}_{p,h}^{n+1}\|_{L^{2}(\Omega_{p})}^{2}+\mu_{p}\|\varepsilon(\mathbf{u}_{p,h}^{n+1})\|_{L^{2}(\Omega_{p})}^{2}+\frac{c_{0}}{2}\|p_{p,h}^{n+1}\|_{L^{2}(\Omega_{p})}^{2}\\
&+\frac{1}{2\lambda_{p}}\|\alpha p_{p,h}^{n+1}-\beta_{p,h}^{n+1}\|_{L^{2}(\Omega_{p})}^{2},\\
\mathcal{N}_{h}^{n+1}=&\frac{L_{1}}{2}\|\mathbf{v}_{f,h}^{n+1}\cdot\mathbf{n}_{f}\|_{L^{2}(\Gamma)}^{2}+\frac{L_{2}}{2}\|d_{t}\mathbf{u}_{p,h}^{n+1}\cdot\mathbf{n}_{p}\|_{L^{2}(\Gamma)}^{2}+\frac{L_{3}}{2}\|p_{p,h}^{n+1}\|_{L^{2}(\Gamma)}^{2},\\
\mathcal{J}_{h}^{n+1}
=&2\mu_{f}\|\varepsilon(\mathbf{v}_{f,h}^{n+1})\|_{L^{2}(\Omega_{f})}^{2}
+\mu_{f}^{-1}\|K^{\frac{1}{2}}\nabla p_{p,h}^{n+1}\|_{L^{2}(\Omega_{p})}^{2}+\frac{\rho_{f}\Delta t}{2}\|d_{t}\mathbf{v}_{f,h}^{n+1}\|_{L^{2}(\Omega_{f})}^{2}\\
&+\frac{\rho_{p}\Delta t}{2}\|d_{t}\mathbf{v}_{p,h}^{n+1}\|_{L^{2}(\Omega_{p})}^{2}+\mu_{p}\Delta t\|d_{t}\varepsilon(\mathbf{u}_{p,h}^{n+1})\|_{L^{2}(\Omega_{p})}^{2}+\frac{c_{0}\Delta t}{2}\|d_{t}p_{p,h}^{n+1}\|_{L^{2}(\Omega_{p})}^{2}\\
&+\frac{\Delta t}{2\lambda_{p}}\|d_{t}(\alpha p_{p,h}^{n+1}-\beta_{p,h}^{n+1})\|_{L^{2}(\Omega_{p})}^{2},\\
\widetilde{\mathcal{J}}_{h}^{n+1}=&\mathcal{J}_{h}^{n+1}-\mu_{f}\|\varepsilon(\mathbf{v}_{f,h}^{n+1})\|_{L^{2}(\Omega_{f})}^{2}-\frac{1}{2\mu_{f}}\|K^{\frac{1}{2}}\nabla p_{p,h}^{n+1}\|_{L^{2}(\Omega_{p})}^{2},\\
\mathcal{M}_{h}^{n+1}=&\frac{1}{2}\sum_{j=1}^{d-1}c_{BJS}\big(\|\mathbf{v}_{f,h}^{n+1}\cdot\bm{\tau}_{f,j}\|_{L^{2}(\Gamma)}^{2}+\|d_{t}\mathbf{u}_{p,h}^{n+1}\cdot\bm{\tau}_{p,j}\|_{L^{2}(\Gamma)}^{2}\big),\\
\mathcal{L}_{h}^{n+1}=&\frac{1}{2}\sum_{j=1}^{d-1}c_{BJS}\big(\|(\mathbf{v}_{f,h}^{n+1}-d_{t}\mathbf{u}_{p,h}^{n})\cdot\bm{\tau}_{f,j}\|_{L^{2}(\Gamma)}^{2}+\|(d_{t}\mathbf{u}_{p,h}^{n+1}-\mathbf{v}_{f,h}^{n})\cdot\bm{\tau}_{p,j}\|_{L^{2}(\Gamma)}^{2}\big),
\end{align*}
we present the energy estimate for Algorithm \ref{al1} in the following theorem:
\begin{theorem}\label{lem5.1}
 Let $(\mathbf{v}_{f,h}^{n+1},p_{f,h}^{n+1},\mathbf{v}_{p,h}^{n+1},\mathbf{u}_{p,h}^{n+1},\beta_{p,h}^{n+1},p_{p,h}^{n+1})$ be the numerical solution generated by Algorithm \ref{al1} at time level $t_{n+1}$. Then the following energy inequalities hold:
    \begin{align}
    &\mathcal{E}_{h}^{l+1}+\Delta t(\mathcal{N}_{h}^{l+1}+\mathcal{M}_{h}^{l+1})+\Delta t\sum_{n=0}^{l}(\widetilde{\mathcal{J}}_{h}^{n+1}+\mathcal{L}_{h}^{n+1})\label{5.2}\\
    &\leq\Big[\mathcal{E}_{h}^{0}+\Delta t(\mathcal{N}_{h}^{0}+\mathcal{M}_{h}^{0})\Big]e^{\widetilde{C}T},\nonumber\\
    &R_{1,h}^{l+1}\leq\big(1+\frac{1}{L_{3}}\big)\Big[\frac{2}{\sqrt{\Delta t}}\mathcal{E}_{h}^{0}+2(\mathcal{N}_{h}^{0}+\mathcal{M}_{h}^{0})\Big]e^{\widetilde{C}T},\label{5.3}
    \end{align}
    \begin{align}
    &R_{2,h}^{l+1}\leq\Big[\frac{2}{\sqrt{\Delta t}}\mathcal{E}_{h}^{0}+2(\mathcal{N}_{h}^{0}+\mathcal{M}_{h}^{0})\Big]e^{\widetilde{C}T},\label{5.4}\\
    &R_{3,h}^{l+1}\leq\big(1+\frac{1}{L_{3}}\big)\Big[\frac{2}{\sqrt{\Delta t}}\mathcal{E}_{h}^{0}+2(\mathcal{N}_{h}^{0}+\mathcal{M}_{h}^{0})\Big]e^{\widetilde{C}T},\label{5.5}\\
    &R_{4,h}^{l+1}\leq\Big[\frac{2}{\sqrt{\Delta t}}\mathcal{E}_{h}^{0}+2(\mathcal{N}_{h}^{0}+\mathcal{M}_{h}^{0})\Big]e^{\widetilde{C}T},\label{5.6}\\
    &R_{5,h}^{l+1}\leq\big(1+\frac{1}{L_{1}}+\frac{1}{L_{2}}\big)\Big[\frac{2}{\sqrt{\Delta t}}\mathcal{E}_{h}^{0}+2(\mathcal{N}_{h}^{0}+\mathcal{M}_{h}^{0})\Big]e^{\widetilde{C}T},\label{5.7}
    \end{align}
    where the positive constant $\widetilde{C}$ depends on the Robin parameters $L_{1}, L_{2}, L_{3}$.
\end{theorem}
\begin{proof}
We first apply the discrete temporal difference operator $d_{t}$ to both sides of the equation \eqref{4.6}. Then, we test the system \eqref{4.2}--\eqref{4.7} by choosing test functions as $(\mathbf{w}_{f,h},q_{f,h})=(\mathbf{v}_{f,h}^{n+1},p_{f,h}^{n+1})$ for the Stokes subproblem, and $(\mathbf{z}_{p,h},\mathbf{w}_{p,h},\varphi_{p,h},\psi_{p,h})=(\rho_{p}d_{t}\mathbf{v}_{p,h}^{n+1},d_{t}\mathbf{u}_{p,h}^{n+1},\beta_{p,h}^{n+1},p_{p,h}^{n+1})$ for the Biot subproblem. Summing the resulting identities, we obtain:
\begin{align}
    &\frac{\rho_{f}\Delta t}{2}\|d_{t}\mathbf{v}_{f,h}^{n+1}\|_{L^{2}(\Omega_{f})}^{2}
    +\frac{\rho_{f}}{2}d_{t}\|\mathbf{v}_{f,h}^{n+1}\|_{L^{2}(\Omega_{f})}^{2}
    +2\mu_{f}\|\varepsilon(\mathbf{v}_{f,h}^{n+1})\|_{L^{2}(\Omega_{f})}^{2}
\label{5.8}\\
    &\quad+L_{1}\langle\mathbf{v}_{f,h}^{n+1}\cdot\mathbf{n}_{f},\mathbf{v}_{f,h}^{n+1}\cdot\mathbf{n}_{f}\rangle_{\Gamma}+c_{BJS}\langle\mathcal{M}_{f}(\mathbf{v}_{f,h}^{n+1}),\mathcal{M}_{f}(\mathbf{v}_{f,h}^{n+1})\rangle_{\Gamma}\nonumber\\
    &=\langle R_{1,h}^{n},\mathbf{v}_{f,h}^{n+1}\cdot\mathbf{n}_{f}\rangle_{\Gamma}-\langle R_{2,h}^{n},\mathcal{M}_{f}(\mathbf{v}_{f,h}^{n+1})\rangle_{\Gamma},\nonumber\\
    &\frac{\rho_{p}\Delta t}{2}\|d_{t}\mathbf{v}_{p,h}^{n+1}\|_{L^{2}(\Omega_{p})}^{2}
    +\frac{\rho_{p}}{2}d_{t}\|\mathbf{v}_{p,h}^{n+1}\|_{L^{2}(\Omega_{p})}^{2}
    +\mu_{p}\Delta t\|\varepsilon(d_{t}\mathbf{u}_{p,h}^{n+1})\|_{L^{2}(\Omega_{p})}^{2}
\label{5.9}\\
    &\quad+\mu_{p}d_{t}\|\varepsilon(\mathbf{u}_{p,h}^{n+1})\|_{L^{2}(\Omega_{p})}^{2}+\frac{c_{0}}{2}d_{t}\|p_{p,h}^{n+1}\|_{L^{2}(\Omega_{p})}^{2}
    +\frac{c_{0}\Delta t}{2}\|d_{t}p_{p,h}^{n+1}\|_{L^{2}(\Omega_{p})}^{2}
    \nonumber\\
    &\quad+\frac{1}{2\lambda_{p}}d_{t}\|\alpha p_{p,h}^{n+1}
    -\beta_{p,h}^{n+1}\|_{L^{2}(\Omega_{p})}^{2}+\frac{\Delta t}{2\lambda_{p}}\|d_{t}(\alpha p_{p,h}^{n+1}
    -\beta_{p,h}^{n+1})\|_{L^{2}(\Omega_{p})}^{2}\nonumber\\
    &+\mu_{f}^{-1}\|K^{\frac{1}{2}}\nabla p_{p,h}^{n+1})\|_{L^{2}(\Omega_{p})}^{2}+L_{2}\langle d_{t}\mathbf{u}_{p,h}^{n+1}\cdot\mathbf{n}_{p},d_{t}\mathbf{u}_{p,h}^{n+1}\cdot\mathbf{n}_{p}\rangle_{\Gamma}\nonumber\\
    &\quad +c_{BJS}\langle\mathcal{M}_{p}(d_{t}\mathbf{u}_{p,h}^{n+1}),\mathcal{M}_{p}(d_{t}\mathbf{u}_{p,h}^{n+1})\rangle_{\Gamma}+L_{3}\langle p_{p,h}^{n+1},p_{p,h}^{n+1}\rangle_{\Gamma}\nonumber\\
    &=\langle R_{3,h}^{n},d_{t}\mathbf{u}_{p,h}^{n+1}\cdot\mathbf{n}_{p}\rangle_{\Gamma}-\langle R_{4,h}^{n},\mathcal{M}_{p}(d_{t}\mathbf{u}_{p,h}^{n+1})\rangle_{\Gamma}+\langle R_{5,h}^{n},p_{p,h}^{n+1}\rangle_{\Gamma}.\nonumber
    \end{align}
Adding \eqref{5.8} to \eqref{5.9}, then utilizing the relations in \eqref{4.8}-\eqref{4.12}, we arrive at
\begin{align}
&d_{t}\mathcal{E}_{h}^{n+1}+\mathcal{J}_{h}^{n+1}+L_{1}\langle(\mathbf{v}_{f,h}^{n+1}-\mathbf{v}_{f,h}^{n})\cdot\mathbf{n}_{f},\mathbf{v}_{f,h}^{n+1}\cdot\mathbf{n}_{f}\rangle_{\Gamma}
    \label{5.11}\\
    &\quad+c_{BJS}\langle\mathcal{M}_{f}(\mathbf{v}_{f,h}^{n+1}-d_{t}\mathbf{u}_{p,h}^{n}),\mathcal{M}_{f}(\mathbf{v}_{f,h}^{n+1})\rangle_{\Gamma}\nonumber\\
    &\quad+L_{2}\langle (d_{t}\mathbf{u}_{p,h}^{n+1}-d_{t}\mathbf{u}_{p,h}^{n})\cdot\mathbf{n}_{p},d_{t}\mathbf{u}_{p,h}^{n+1}\cdot\mathbf{n}_{p}\rangle_{\Gamma}
    \nonumber\\
    &\quad+c_{BJS}\langle\mathcal{M}_{p}(d_{t}\mathbf{u}_{p,h}^{n+1}-\mathbf{v}_{f,h}^{n}),\mathcal{M}_{p}(d_{t}\mathbf{u}_{p,h}^{n+1})\rangle_{\Gamma}+L_{3}\langle p_{p,h}^{n+1}-p_{p,h}^{n},p_{p,h}^{n+1}\rangle_{\Gamma}\nonumber\\
    &=-\langle p_{p,h}^{n},\mathbf{v}_{f,h}^{n+1}\cdot\mathbf{n}_{f}\rangle_{\Gamma}-\langle p_{p,h}^{n},d_{t}\mathbf{u}_{p,h}^{n+1}\cdot\mathbf{n}_{p}\rangle_{\Gamma}+ \langle\mathbf{v}_{f,h}^{n}\cdot\mathbf{n}_{f},p_{p,h}^{n+1}\rangle_{\Gamma}\nonumber\\
    &\quad+\langle d_{t}\mathbf{u}_{p,h}^{n}\cdot\mathbf{n}_{p},p_{p,h}^{n+1}\rangle_{\Gamma}.\nonumber
    \end{align}
    Applying the identity \eqref{5.1} to the interface terms, we obtain:
\begin{align}
    &L_{1}\langle(\mathbf{v}_{f,h}^{n+1}-\mathbf{v}_{f,h}^{n})\cdot\mathbf{n}_{f},\mathbf{v}_{f,h}^{n+1}\cdot\mathbf{n}_{f}\rangle_{\Gamma}\label{5.12}\\
    &=\frac{L_{1}}{2}\|\mathbf{v}_{f,h}^{n+1}\cdot\mathbf{n}_{f}\|_{L^{2}(\Gamma)}^{2}-\frac{L_{1}}{2}\|\mathbf{v}_{f,h}^{n}\cdot\mathbf{n}_{f}\|_{L^{2}(\Gamma)}^{2}\nonumber\\
    &\quad+\frac{L_{1}}{2}\|(\mathbf{v}_{f,h}^{n+1}-\mathbf{v}_{f,h}^{n})\cdot\mathbf{n}_{f}\|_{L^{2}(\Gamma)}^{2},\nonumber\\
    &L_{2}\langle(d_{t}\mathbf{u}_{p,h}^{n+1}-d_{t}\mathbf{u}_{p,h}^{n})\cdot\mathbf{n}_{p},d_{t}\mathbf{u}_{p,h}^{n+1}\cdot\mathbf{n}_{p}\rangle_{\Gamma}\label{5.13}\\
    &=\frac{L_{2}}{2}\|d_{t}\mathbf{u}_{p,h}^{n+1}\cdot\mathbf{n}_{p}\|_{L^{2}(\Gamma)}^{2}-\frac{L_{2}}{2}\|d_{t}\mathbf{u}_{p,h}^{n}\cdot\mathbf{n}_{p}\|_{L^{2}(\Gamma)}^{2}\nonumber\\
    &\quad+\frac{L_{2}}{2}\|(d_{t}\mathbf{u}_{p,h}^{n+1}-d_{t}\mathbf{u}_{p,h}^{n})\cdot\mathbf{n}_{p}\|_{L^{2}(\Gamma)}^{2},\nonumber\\
    &c_{BJS}\langle\mathcal{M}_{f}(\mathbf{v}_{f,h}^{n+1})
    -\mathcal{M}_{f}(d_{t}\mathbf{u}_{p,h}^{n}),\mathcal{M}_{f}(\mathbf{v}_{f,h}^{n+1})\rangle_{\Gamma}
    \label{5.14}\\
    &\quad+c_{BJS}\langle\mathcal{M}_{p}(d_{t}\mathbf{u}_{p,h}^{n+1})-\mathcal{M}_{p}(\mathbf{v}_{f,h}^{n}),\mathcal{M}_{p}(d_{t}\mathbf{u}_{p,h}^{n+1})\rangle_{\Gamma}\nonumber\\
	 &=\frac{1}{2}\sum_{j=1}^{d-1} c_{BJS} \big(\|\mathbf{v}_{f,h}^{n+1}\cdot\bm{\tau}_{f,j}\|_{L^{2}(\Gamma)}^{2}-\|d_{t}\mathbf{u}_{p,h}^{n}\cdot\bm{\tau}_{f,j}\|_{L^{2}(\Gamma)}^{2}\nonumber\\
	 	&\quad+\|(\mathbf{v}_{f,h}^{n+1}-d_{t}\mathbf{u}_{p,h}^{n})\cdot\bm{\tau}_{f,j}\|_{L^{2}(\Gamma)}^{2}+\|d_{t}\mathbf{u}_{p,h}^{n+1}\cdot\bm{\tau}_{p,j}\|_{L^{2}(\Gamma)}^{2}\nonumber\\
        &\quad-\|\mathbf{v}_{f,h}^{n}\cdot\bm{\tau}_{p,j}\|_{L^{2}(\Gamma)}^{2}+\|(d_{t}\mathbf{u}_{p,h}^{n+1}-\mathbf{v}_{f,h}^{n})\cdot\bm{\tau}_{p,j}\|_{L^{2}(\Gamma)}^{2}\big),\nonumber\\
   &L_{3}\langle p_{p,h}^{n+1}-p_{p,h}^{n},p_{p,h}^{n+1}\rangle_{\Gamma}\label{5.15}\\
   &=\frac{L_{3}}{2}\|p_{p,h}^{n+1}\|_{L^{2}(\Gamma)}^{2}-\frac{L_{3}}{2}\|p_{p,h}^{n}\|_{L^{2}(\Gamma)}^{2}+\frac{L_{3}}{2}\|p_{p,h}^{n+1}-p_{p,h}^{n})\|_{L^{2}(\Gamma)}^{2}.\nonumber
    \end{align}
After applying Cauchy-Schwarz and Young inequalities for the first and third terms on right-hand side of \eqref{5.11}, we get:
\begin{align}    &\langle\mathbf{v}_{f,h}^{n}\cdot\mathbf{n}_{f},p_{p,h}^{n+1}\rangle_{\Gamma}-\langle p_{p,h}^{n},\mathbf{v}_{f,h}^{n+1}\cdot\mathbf{n}_{f}\rangle_{\Gamma}\label{5.16}\\
&=\langle\mathbf{v}_{f,h}^{n}\cdot\mathbf{n}_{f},p_{p,h}^{n+1}-p_{p,h}^{n}\rangle_{\Gamma}-\langle p_{p,h}^{n},(\mathbf{v}_{f,h}^{n+1}-\mathbf{v}_{f,h}^{n})\cdot\mathbf{n}_{f}\rangle_{\Gamma}\nonumber\\
    &\leq \frac{1}{L_{3}}\|\mathbf{v}_{f,h}^{n}\cdot\mathbf{n}_{f}\|_{L^{2}(\Gamma)}^{2}
    +\frac{L_{3}}{4}\|p_{p,h}^{n+1}-p_{p,h}^{n}\|_{L^{2}(\Gamma)}^{2}
    +\frac{1}{2L_{1}}\|p_{p,h}^{n}\|_{L^{2}(\Gamma)}^{2}\nonumber\\
    &\quad+\frac{L_{1}}{2}\|(\mathbf{v}_{f,h}^{n+1}-\mathbf{v}_{f,h}^{n})\cdot\mathbf{n}_{f}\|_{L^{2}(\Gamma)}^{2}.\nonumber
    \end{align}
    Similarly, by adding and subtracting the cross term, it follows that:
\begin{align}
    &\langle d_{t}\mathbf{u}_{p,h}^{n}\cdot\mathbf{n}_{p},p_{p,h}^{n+1}\rangle_{\Gamma}-\langle p_{p,h}^{n},d_{t}\mathbf{u}_{p,h}^{n+1}\cdot\mathbf{n}_{p}\rangle_{\Gamma}\label{5.17}
    \\
    &\leq \frac{1}{L_{3}}\|d_{t}\mathbf{u}_{p,h}^{n}\cdot\mathbf{n}_{p}\|_{L^{2}(\Gamma)}^{2}+\frac{L_{3}}{4}\|p_{p,h}^{n+1}-p_{p,h}^{n}\|_{L^{2}(\Gamma)}^{2}+\frac{1}{2L_{2}}\|p_{p,h}^{n}\|_{L^{2}(\Gamma)}^{2}\nonumber\\
    &\quad+\frac{L_{2}}{2}\|(d_{t}\mathbf{u}_{p,h}^{n+1}-d_{t}\mathbf{u}_{p,h}^{n})\cdot\mathbf{n}_{p}\|_{L^{2}(\Gamma)}^{2}.\nonumber
    \end{align}
Substituting \eqref{5.12}-\eqref{5.17} into \eqref{5.11}, then summing the resulting inequality over the time steps $n=0, 1, \cdots, l$ and multiplying by $\Delta t$, we obtain:
\begin{align}
&\mathcal{E}_{h}^{l+1}+\Delta t(\mathcal{N}_{h}^{l+1}+\mathcal{M}_{h}^{l+1})+\Delta t\sum_{n=0}^{l}(\mathcal{J}_{h}^{n+1}+\mathcal{L}_{h}^{n+1})\label{5.19}\\
&\leq\mathcal{E}_{h}^{0}+\Delta t(\mathcal{N}_{h}^{0}+\mathcal{M}_{h}^{0})+\Delta t\sum_{n=0}^{l}\widetilde{\mathcal{N}}_{h}^{n},\nonumber
\end{align}
    where
\begin{align*}
    \widetilde{\mathcal{N}}_{h}^{n}=\frac{1}{L_{3}}\|\mathbf{v}_{f,h}^{n}\cdot\mathbf{n}_{f}\|_{L^{2}(\Gamma)}^{2}+\big(\frac{1}{2L_{1}}+\frac{1}{2L_{2}}\big)\|p_{p,h}^{n}\|_{L^{2}(\Gamma)}^{2}+\frac{1}{L_{3}}\|d_{t}\mathbf{u}_{p,h}^{n}\cdot\mathbf{n}_{p}\|_{L^{2}(\Gamma)}^{2}.
    \end{align*}
    For the third term on the right side of inequality \eqref{5.19}, using the trace, Korn and Young inequalities, we obtain:
\begin{align}
     \Delta t\sum_{n=0}^{l}\widetilde{\mathcal{N}}_{h}^{n}\leq&\Delta t\sum_{n=0}^{l}\big(\mu_{f}\|\varepsilon(\mathbf{v}_{f,h}^{n})\|_{L^{2}(\Omega_{f})}^{2}
    +\frac{1}{2\mu_{f}}\|K^{\frac{1}{2}}\nabla p_{p,h}^{n}\|_{L^{2}(\Omega_{p})}^{2}
    \label{5.20}\\
    &+ \frac{\widehat{C}}{L_{3}}\| d_{t}\varepsilon(\mathbf{u}_{p, h}^{n})\|_{L^{2}(\Omega_{p})}^{2}\big)+C\Delta t\sum_{n=0}^{l}\big(
    \frac{1}{4L_{3}^{2}\mu_{f}}\|\mathbf{v}_{f,h}^{n}\|_{L^{2}(\Omega_{f})}^{2}
    \nonumber\\
    &+\frac{\mu_{f}}{8K_{2}}(\frac{1}{L_{1}^{2}}
    +\frac{1}{L_{2}^{2}})\|p_{p,h}^{n}\|_{L^{2}(\Omega_{p})}^{2}\big).\nonumber
    \end{align}
Note that we use the fact of $\frac{\widehat{C}}{L_3} \leq \mu_{p}\Delta t$ which can be satisfied by an appropriate choice of $L_{3}$. Substituting \eqref{5.20} into \eqref{5.19} and applying Gr\"onwall inequality, we conclude that \eqref{5.2} holds. \eqref{5.3}-\eqref{5.7} immediately follow from the definition of $L^{2}$-norm, Cauchy-Schwarz inequality and \eqref{5.2}. The proof is complete.  
\end{proof}
\begin{remark}\label{rem5.1}
Based on \eqref{2.25} and \eqref{4.12}, we have
\begin{equation}
R_{5,h}^{n}
=
L_3 p_{p,h}^{n+1} +
\mu_f^{-1}K\big(\nabla p_{p,h}^{n+1}-\rho_f\mathbf{g}\big)\cdot \mathbf{n}_p =
L_3 p_{p,h}^{n} + \mathbf{v}_{f,h}^{n}\cdot \mathbf{n}_f +
d_t\mathbf{u}_{p,h}^{n}\cdot \mathbf{n}_p .
\label{eq:remark42}
\end{equation}
When the coefficient tensor $\mu_f^{-1}K$ is sufficiently small, the Darcy flux term
\[
\mu_f^{-1}K\big(\nabla p_{p,h}^{n+1}-\rho_f\mathbf{g}\big)\cdot \mathbf{n}_p
\]
becomes small accordingly. In view of the interface condition \eqref{2.13}, this implies
\[
\mathbf{v}_{f,h}^{n}\cdot \mathbf{n}_f + d_t\mathbf{u}_{p,h}^{n}\cdot \mathbf{n}_p \approx 0 .
\]
Therefore, if $|\mu_f^{-1}K| \ll L_3$, then \eqref{eq:remark42} reduces approximately to
\[
L_3 p_{p,h}^{n+1} \approx L_3 p_{p,h}^{n},
\]
which means that the interface condition has only a very weak influence on the discrete solution. To avoid this degeneracy and to retain the physical coupling effect, the parameter $L_3$ should be chosen consistently with the permeability scale. In the present formulation, it is therefore natural to take
\[
L_3 = O(\mu_f^{-1}K).
\]
By contrast, no sharp theoretical criterion is currently available for the choice of $L_1$ and $L_2$. These two parameters mainly act as stabilizing parameters in the Robin--Robin coupling. Choosing them too small may weaken the stabilizing effect, whereas choosing them excessively large may cause the Robin terms to dominate the original interface coupling conditions, which in turn introduces additional numerical dissipation. This observation is consistent with the numerical experiments in Section~6.
\end{remark}
\section{Error estimate}
The primary objective of this section is to establish the optimal error estimate for the  fully decoupled scheme proposed in Algorithm \ref{al1}. To facilitate the analysis, we first introduce the Stokes projection operator
$\mathcal{Q}_{h}:(\mathbf{v},p)\in\mathbf{V}_{*}\times M_{*}\rightarrow(\mathcal{Q}_{h}\mathbf{v},\mathcal{Q}_{h}p)\in\mathbf{V}_{*,h}\times M_{*,h}$ by:
\begin{align}
&2\mu_{*}\big(\varepsilon(\mathcal{Q}_{h}\mathbf{v}),\varepsilon(\bm{\omega}_{h})\big)_{\Omega_{*}}-\big(\mathcal{Q}_{h}p,\nabla\cdot\bm{\omega}_{h}\big)_{\Omega_{*}}\label{6.1}\\
&=2\mu_{*}\big(\varepsilon(\mathbf{v}),\varepsilon(\bm{\omega}_{h})\big)_{\Omega_{*}}-\big(p,\nabla\cdot\bm{\omega}_{h}\big)_{\Omega_{*}},~\forall~\bm{\omega}_{h}\in\mathbf{V}_{*,h},\nonumber\\
&\big(\nabla\cdot\mathcal{Q}_{h}\mathbf{v},q_{h}\big)_{\Omega_{*}}=\big(\nabla\cdot\mathbf{v},q_{h}\big)_{\Omega_{*}},~\forall~q_{h}\in M_{*,h},\label{6.2}
\end{align}
where the subscript $*=f,p$. Under standard regularity assumptions for $(\mathbf{v}, p)$, the following approximation property holds: 
for any $\mathbf{v} \in \mathbf{H}^{3}(\Omega_{*})$ and $p \in H^{2}(\Omega_{*})$,
\begin{align}
&\|\mathcal{Q}_{h}\mathbf{v}-\mathbf{v}\|_{H^{1}(\Omega_{*})}+\|\mathcal{Q}_{h}p-p\|_{L^{2}(\Omega_{*})}\leq Ch^{2}\big(\|\mathbf{v}\|_{H^{3}(\Omega_{*})}+\|p\|_{H^{2}(\Omega_{*})}\big).\label{6.3}
\end{align}

 Next, for any $\varphi \in H^{1}(\Omega_{p})$, we define the Ritz projection operator $\mathcal{R}_{h}: M_{p}\rightarrow M_{p,h}$ by:
\begin{align}
	&(\mu_{f}^{-1}K\nabla\mathcal{R}_{h}\varphi,\nabla\psi_{h})_{\Omega_{p}}=(\mu_{f}^{-1}K\nabla\varphi,\nabla\psi_{h})_{\Omega_{p}},~~~~~\forall \psi_{h}\in M_{p,h}.\label{6.4}
\end{align}

As established in \cite{brenner2008mathematical}, the operator $\mathcal{R}_{h}$ satisfies the following approximation estimate:
\begin{align}
&\|\mathcal{R}_{h}\varphi-\varphi\|_{L^{2}(\Omega_{p})}+h\|\nabla(\mathcal{R}_{h}\varphi-\varphi)\|_{L^{2}(\Omega_{p})}\leq Ch^{3}\|\varphi\|_{H^{3}(\Omega_{p})},~\forall \varphi\in H^{3}(\Omega_{p}).\label{6.5}
\end{align}
Let 
$
e_{\mathbf{v}_{f}}^{n}=\mathbf{v}_{f}(t_{n})-\mathbf{v}_{f,h}^{n},
e_{p_{f}}^{n}=p_{f}(t_{n})-p_{f,h}^{n}, 
e_{\mathbf{v}_{p}}^{n}=\mathbf{v}_{p}(t_{n})-\mathbf{v}_{p,h}^{n},
e_{\mathbf{u}_{p}}^{n}=\mathbf{u}_{p}(t_{n})-\mathbf{u}_{p,h}^{n},
$
$
e_{\beta_{p}}^{n}=\beta_{p}(t_{n})-\beta_{p,h}^{n},
e_{p_{p}}^{n}=p_{p}(t_{n})-p_{p,h}^{n}
$ denote the errors of the primary variables.
And let:
\begin{align*}
    e_{\mathcal{A}_{*}}^{n} &= \mathcal{A}_{*}(t_{n}) - \mathcal{Q}_{h}\mathcal{A}_{*}(t_{n}) + \mathcal{Q}_{h}\mathcal{A}_{*}(t_{n}) - \mathcal{A}_{*,h}^{n} = \Lambda_{\mathcal{A}_{*}}^{n} + \Theta_{\mathcal{A}_{*}}^{n}, \\
    e_{\mathcal{B}_{p}}^{n} &= \mathcal{B}_{p}(t_{n}) - \mathcal{R}_{h}\mathcal{B}_{p}(t_{n}) + \mathcal{R}_{h}\mathcal{B}_{p}(t_{n}) - \mathcal{B}_{p,h}^{n} = \widetilde{\Lambda}_{\mathcal{B}_{p}}^{n} + \widetilde{\Theta}_{\mathcal{B}_{p}}^{n},
\end{align*}
where $
\mathcal{A}_{*}=\mathbf{v}_{f},p_{f},\mathbf{v}_{p},\mathbf{u}_{p},\beta_{p},p_{p},
\ \text{and}\ \mathcal{B}_{p} = \beta_{p},p_{p}.
$

To simsplify our error analysis, we introduce the following notation:
\begin{equation*}
\begin{aligned}
\mathcal{E}_{\Theta}^{n+1}=&\frac{\rho_{f}}{2}\|\Theta_{\mathbf{v}_{f}}^{n+1}\|_{L^{2}(\Omega_{f})}^{2}+\frac{\rho_{p}}{2}\|\Theta_{\mathbf{v}_{p}}^{n+1}\|_{L^{2}(\Omega_{p})}^{2}+\mu_{p}\|\varepsilon(\Theta_{\mathbf{u}_{p}}^{n+1})\|_{L^{2}(\Omega_{p})}^{2}+\frac{c_{0}}{2}\|\widetilde{\Theta}_{p_{p}}^{n+1}\|_{L^{2}(\Omega_{p})}^{2}\\
  &+\frac{1}{2\lambda_{p}}\|\alpha\widetilde{\Theta}_{p_{p}}^{n+1}-\Theta_{\beta_{p}}^{n+1}\|_{L^{2}(\Omega_{p})}^{2},
 \\
\mathcal{J}_{\Theta}^{n+1}
=&2\mu_{f}\|\varepsilon(\Theta_{\mathbf{v}_{f}}^{n+1})\|_{L^{2}(\Omega_{f})}^{2}
+\mu_{f}^{-1}\|K^{\frac{1}{2}}\nabla\widetilde{\Theta}_{p_{p}}^{n+1}\|_{L^{2}(\Omega_{p})}^{2}
+\frac{\rho_{f}\Delta t}{2}\|d_{t}\Theta_{\mathbf{v}_{f}}^{n+1}\|_{L^{2}(\Omega_{f})}^{2}\\
&+\frac{\rho_{p}\Delta t}{2}\|d_{t}\Theta_{\mathbf{v}_{p}}^{n+1}\|_{L^{2}(\Omega_{p})}^{2}
+\mu_{p}\Delta t\|d_{t}\varepsilon(\Theta_{\mathbf{u}_{p}}^{n+1})\|_{L^{2}(\Omega_{p})}^{2}
+\frac{c_{0}\Delta t}{2}\|d_{t}\widetilde{\Theta}_{p_{p}}^{n+1}\|_{L^{2}(\Omega_{p})}^{2}\\
&+\frac{\Delta t}{2\lambda_{p}}\|d_{t}(\alpha\widetilde{\Theta}_{p_{p}}^{n+1}-\Theta_{\beta_{p}}^{n+1})\|_{L^{2}(\Omega_{p})}^{2},
\\
\mathcal{N}_{\Theta}^{n+1}
=&\frac{L_{1}}{2}\|\Theta_{\mathbf{v}_{f}}^{n+1}\cdot\mathbf{n}_{f}\|_{L^{2}(\Gamma)}^{2}
+\frac{L_{2}}{2}\|d_{t}\Theta_{\mathbf{u}_{p}}^{n+1}\cdot\mathbf{n}_{p}\|_{L^{2}(\Gamma)}^{2}
+\frac{L_{3}}{2}\|\widetilde{\Theta}_{p_{p}}^{n+1}\|_{L^{2}(\Gamma)}^{2},
\\
\mathcal{D}_{\Theta}^{n+1}
=&\frac{L_{1}}{2}\|(\Theta_{\mathbf{v}_{f}}^{n+1}
-\Theta_{\mathbf{v}_{f}}^{n})\cdot\mathbf{n}_{f}\|_{L^{2}(\Gamma)}^{2}
+\frac{L_{2}}{2}\|(d_{t}\Theta_{\mathbf{u}_{p}}^{n+1}
-d_{t}\Theta_{\mathbf{u}_{p}}^{n})\cdot\mathbf{n}_{p}\|_{L^{2}(\Gamma)}^{2}\\
&+\frac{L_{3}}{2}\|\widetilde{\Theta}_{p_{p}}^{n+1}
-\widetilde{\Theta}_{p_{p}}^{n}\|_{L^{2}(\Gamma)}^{2},
\\
\mathcal{M}_{\Theta}^{n+1}
=&\frac{1}{2}\sum_{j=1}^{d-1}c_{BJS}\big(\|\Theta_{\mathbf{v}_{f}}^{n+1}\cdot\bm{\tau}_{f,j}\|_{L^{2}(\Gamma)}^{2}
+\|d_{t}\Theta_{\mathbf{u}_{p}}^{n+1}\cdot\bm{\tau}_{p,j}\|_{L^{2}(\Gamma)}^{2}\big),
\\
\mathcal{L}_{\Theta}^{n+1}
=&\frac{1}{2}\sum_{j=1}^{d-1}c_{BJS}\big(\|(\Theta_{\mathbf{v}_{f}}^{n+1}
-d_{t}\Theta_{\mathbf{u}_{p}}^{n})\cdot\bm{\tau}_{f,j}\|_{L^{2}(\Gamma)}^{2}
+\|(d_{t}\Theta_{\mathbf{u}_{p}}^{n+1}
-\Theta_{\mathbf{v}_{f}}^{n})\cdot\bm{\tau}_{p,j}\|_{L^{2}(\Gamma)}^{2}\big).
\end{aligned}
\end{equation*}
Before proceeding with the error estimates, we present the following error equation.
\begin{lemma}
Let $(\mathbf{v}_{f}, p_{f}, \mathbf{v}_{p}, \mathbf{u}_{p}, \beta_{p}, p_{p})$ be the exact solution of \eqref{3.1}-\eqref{3.6} at time $t_{n+1}$, and let $(\mathbf{v}_{f,h}^{n+1}, p_{f,h}^{n+1}, \mathbf{v}_{p,h}^{n+1}, \mathbf{u}_{p,h}^{n+1}, \beta_{p,h}^{n+1}, p_{p,h}^{n+1})$ be the numerical solution generated by Algorithm \ref{al1}. Then, the following error equation hold:
\begin{align}
    &\mathcal{E}_{\Theta}^{l+1}+\Delta t(\mathcal{N}_{\Theta}^{l+1}+\mathcal{M}_{\Theta}^{l+1})+\Delta t\sum_{n=0}^{l}(\mathcal{J}_{\Theta}^{n+1}+\mathcal{D}_{\Theta}^{n+1}+\mathcal{L}_{\Theta}^{n+1})\label{6.6}\\
    &=\mathcal{E}_{\Theta}^{0}+\Delta t(\mathcal{N}_{\Theta}^{0}+\mathcal{M}_{\Theta}^{0})+\Delta t\sum_{i=1}^{7}\sum_{n=0}^{l}\Phi_{i},\nonumber
\end{align}
    where:
\begin{align*}
\Phi_{1}=&-\rho_{f}\big(d_{t}\Lambda_{\mathbf{v}_{f}}^{n+1},\Theta_{\mathbf{v}_{f}}^{n+1}\big)_{\Omega_{f}}
+\frac{\alpha}{\lambda_{p}}\big(d_{t}\widetilde{\Lambda}_{p_{p}}^{n+1},\Theta_{\beta_{p}}^{n+1}\big)_{\Omega_{p}}
-\frac{1}{\lambda_{p}}\big(d_{t}\Lambda_{\beta_{p}}^{n+1},\Theta_{\beta_{p}}^{n+1}\big)_{\Omega_{p}}\\
&-\big(c_{0}+\frac{\alpha^{2}}{\lambda_{p}}\big)\big(d_{t}\widetilde{\Lambda}_{p_{p}}^{n+1},\widetilde{\Theta}_{p_{p}}^{n+1}\big)_{\Omega_{p}}
+\frac{\alpha}{\lambda_{p}}\big(\Lambda_{\beta_{p}}^{n+1},\widetilde{\Theta}_{p_{p}}^{n+1}\big)_{\Omega_{p}},\\
\Phi_{2}
=&L_{1}\langle (\Lambda_{\mathbf{v}_{f}}^{n}
-\Lambda_{\mathbf{v}_{f}}^{n+1})\cdot\mathbf{n}_{f},
\Theta_{\mathbf{v}_{f}}^{n+1}\cdot\mathbf{n}_{f}\rangle_{\Gamma}
-\langle \widetilde{\Lambda}_{p_{p}}^{n}
,\Theta_{\mathbf{v}_{f}}^{n+1}\cdot\mathbf{n}_{f}\rangle_{\Gamma}
\\
&-c_{BJS}\langle \mathcal{M}_{p}(d_{t}\Lambda_{\mathbf{u}_{p}}^{n}),\mathcal{M}_{f}(\Theta_{\mathbf{v}_{f}}^{n+1})\rangle_{\Gamma}-c_{BJS}\langle\mathcal{M}_{f}(\Lambda_{\mathbf{v}_{f}}^{n+1}),\mathcal{M}_{f}(\Theta_{\mathbf{v}_{f}}^{n+1})\rangle_{\Gamma}\\
&+L_{3}\langle\widetilde{\Lambda}_{p_{p}}^{n}
-\widetilde{\Lambda}_{p_{p}}^{n+1},\widetilde{\Theta}_{p_{p}}^{n+1}\rangle_{\Gamma}
+\langle\Lambda_{\mathbf{v}_{f}}^{n}\cdot\mathbf{n}_{f},\widetilde{\Theta}_{p_{p}}^{n+1}\rangle_{\Gamma}+\langle d_{t}\Lambda_{\mathbf{u}_{p}}^{n}\cdot\mathbf{n}_{p}
,\widetilde{\Theta}_{p_{p}}^{n+1}\rangle_{\Gamma},
\end{align*}
\begin{align*}
\Phi_{3}
=&-\langle \widetilde{\Theta}_{p_{p}}^{n},
\Theta_{\mathbf{v}_{f}}^{n+1}\cdot\mathbf{n}_{f}\rangle_{\Gamma}
-\langle\widetilde{\Theta}_{p_{p}}^{n}
,d_{t}\Theta_{\mathbf{u}_{p}}^{n+1}\cdot\mathbf{n}_{p}\rangle_{\Gamma}
+\langle\Theta_{\mathbf{v}_{f}}^{n}\cdot\mathbf{n}_{f}
,\widetilde{\Theta}_{p_{p}}^{n+1}\rangle_{\Gamma}\\
&+\langle d_{t}\Theta_{\mathbf{u}_{p}}^{n}\cdot\mathbf{n}_{p}
,\widetilde{\Theta}_{p_{p}}^{n+1}\rangle_{\Gamma},\\
\Phi_{4}
=&\rho_{f}\big(d_{t}\mathbf{v}_{f}(t_{n+1})
-\partial_{t}\mathbf{v}_{f}(t_{n+1}),\Theta_{\mathbf{v}_{f}}^{n+1}\big)_{\Omega_{f}}
+\rho_{p}\big(\partial_{t}\mathbf{u}_{p}(t_{n+1})
-d_{t}\mathbf{u}_{p}(t_{n+1}),d_{t}\Theta_{\mathbf{v}_{p}}^{n+1}\big)_{\Omega_{p}}\\
&+\rho_{p}\big(d_{t}\mathbf{v}_{p}(t_{n+1})
-\partial_{t}\mathbf{v}_{p}(t_{n+1}),d_{t}\Theta_{\mathbf{u}_{p}}^{n+1}\big)_{\Omega_{p}}+\frac{\alpha}{\lambda_{p}}\big(\partial_{t}\beta_{p}(t_{n+1})
-d_{t}\beta_{p}(t_{n+1}),\widetilde{\Theta}_{p_{p}}^{n+1}\big)_{\Omega_{p}}\\
&+\big(c_{0}+\frac{\alpha^{2}}{\lambda_{p}}\big)\big(d_{t}p_{p}(t_{n+1})
-\partial_{t}p_{p}(t_{n+1}),\widetilde{\Theta}_{p_{p}}^{n+1}\big)_{\Omega_{p}},\\
\Phi_{5}
=&c_{BJS}\langle\mathcal{M}_{p}(d_{t}\mathbf{u}_{p}(t_{n})
-\partial_{t}\mathbf{u}_{p}(t_{n})),\mathcal{M}_{f}(\Theta_{\mathbf{v}_{f}}^{n+1})\rangle_{\Gamma}+\langle (\partial_{t}\mathbf{u}_{p}(t_{n})
-d_{t}\mathbf{u}_{p}(t_{n}))\cdot\mathbf{n}_{p}
,\widetilde{\Theta}_{p_{p}}^{n+1}\rangle_{\Gamma}\\
&+L_{2}\langle (d_{t}\mathbf{u}_{p}(t_{n+1})
-\partial_{t}\mathbf{u}_{p}(t_{n+1}))\cdot\mathbf{n}_{p},d_{t}\Theta_{\mathbf{u}_{p}}^{n+1}\cdot\mathbf{n}_{p}\rangle_{\Gamma}\\
&+L_{2}\langle (\partial_{t}\mathbf{u}_{p}(t_{n})
-d_{t}\mathbf{u}_{p}(t_{n}))\cdot\mathbf{n}_{p}
,d_{t}\Theta_{\mathbf{u}_{p}}^{n+1}\cdot\mathbf{n}_{p}\rangle_{\Gamma}\\
&+c_{BJS}\langle\mathcal{M}_{p}(d_{t}\mathbf{u}_{p}(t_{n+1})
-\partial_{t}\mathbf{u}_{p}(t_{n+1}))
,\mathcal{M}_{p}(d_{t}\Theta_{\mathbf{u}_{p}}^{n+1})\rangle_{\Gamma},\\
\Phi_{6}=
&-\rho_{p}\big(\Lambda_{\mathbf{v}_{p}}^{n+1}
,d_{t}\Theta_{\mathbf{v}_{p}}^{n+1}\big)_{\Omega_{p}}
+\rho_{p}\big(d_{t}\Lambda_{\mathbf{u}_{p}}^{n+1},d_{t}\Theta_{\mathbf{v}_{p}}^{n+1}\big)_{\Omega_{p}}
-\rho_{p}\big(d_{t}\Lambda_{\mathbf{v}_{p}}^{n+1}
, d_{t}\Theta_{\mathbf{u}_{p}}^{n+1}\big)_{\Omega_{p}},\\
\Phi_{7}=
&L_{2}\langle (d_{t}\Lambda_{\mathbf{u}_{p}}^{n}
-d_{t}\Lambda_{\mathbf{u}_{p}}^{n+1})\cdot\mathbf{n}_{p},d_{t}\Theta_{\mathbf{u}_{p}}^{n+1}\cdot\mathbf{n}_{p}\rangle_{\Gamma}
-\langle \widetilde{\Lambda}_{p_{p}}^{n},d_{t}\Theta_{\mathbf{u}_{p}}^{n+1}\cdot\mathbf{n}_{p}\rangle_{\Gamma}\\
&-c_{BJS}\langle\mathcal{M}_{f}(\Lambda_{\mathbf{v}_{f}}^{n})
,\mathcal{M}_{p}(d_{t}\Theta_{\mathbf{u}_{p}}^{n+1})\rangle_{\Gamma}\nonumber
-c_{BJS}\langle\mathcal{M}_{p}(d_{t}\Lambda_{\mathbf{u}_{p}}^{n+1})
,\mathcal{M}_{p}(d_{t}\Theta_{\mathbf{u}_{p}}^{n+1})\rangle_{\Gamma}.
\end{align*}
\end{lemma}
\begin{proof}
Firstly, it is straightforward to see that the following error equations associated with the Robin type interface conditions hold:
\begin{align*}
    &e_{R_{1}}^{n}
    =L_{1}(\mathbf{v}_{f}(t_{n})-\mathbf{v}_{f,h}^{n})\cdot\mathbf{n}_{f}
    -(p_{p}(t_{n})-p_{p,h}^{n}),\\
    &e_{R_{2}}^{n}=c_{BJS}\mathcal{M}_{p}(\partial_{t}\mathbf{u}_{p}(t_{n}))
    -c_{BJS}\mathcal{M}_{p}(d_{t}\mathbf{u}_{p,h}^{n}),\\
    &e_{R_{3}}^{n}
    =L_{2}(\partial_{t}\mathbf{u}_{p}(t_{n})
    -d_{t}\mathbf{u}_{p,h}^{n})\cdot\mathbf{n}_{p}-(p_{p}(t_{n})-p_{p,h}^{n}),\\
    &e_{R_{4}}^{n}=c_{BJS}\mathcal{M}_{p}(\mathbf{v}_{f}(t_{n}))-c_{BJS}\mathcal{M}_{p}(\mathbf{v}_{f,h}^{n}),\\
    &e_{R_{5}}^{n}=L_{3}(p_{p}(t_{n})-p_{p,h}^{n})
    +(\mathbf{v}_{f}(t_{n})-\mathbf{v}_{f,h}^{n})\cdot\mathbf{n}_{f}
    +(\partial_{t}\mathbf{u}_{p}(t_{n})-d_{t}\mathbf{u}_{p,h}^{n})\cdot\mathbf{n}_{p}, 
    \end{align*}
 By evaluating the weak form \eqref{3.1}-\eqref{3.2} at time $t = t_{n+1}$, subtracting the fully discrete scheme \eqref{4.2}-\eqref{4.3} from it and utilizing the definition of the operator $\mathcal{Q}_{h}$, together with the above properties, we have:
\begin{align}
&\rho_{f}\big(d_{t}\Theta_{\mathbf{v}_{f}}^{n+1},
\mathbf{w}_{f,h}\big)_{\Omega_{f}}
+2\mu_{f}\big(\varepsilon(\Theta_{\mathbf{v}_{f}}^{n+1})
,\varepsilon(\mathbf{w}_{f,h})\big)_{\Omega_{f}}
-\big(\Theta_{p_{f}}^{n+1},\nabla\cdot\mathbf{w}_{f,h}\big)_{\Omega_{f}}\label{6.13}\\
&+L_{1}\langle (\Theta_{\mathbf{v}_{f}}^{n+1}
-\Theta_{\mathbf{v}_{f}}^{n})\cdot\mathbf{n}_{f}
,\mathbf{w}_{f,h}\cdot\mathbf{n}_{f}\rangle_{\Gamma}
+c_{BJS}\langle\mathcal{M}_{f}(\Theta_{\mathbf{v}_{f}}^{n+1})
,\mathcal{M}_{f}(\mathbf{w}_{f,h})\rangle_{\Gamma}\nonumber\\
& = L_{1}\langle (\Lambda_{\mathbf{v}_{f}}^{n}
-\Lambda_{\mathbf{v}_{f}}^{n+1})\cdot\mathbf{n}_{f}
,\mathbf{w}_{f,h}\cdot\mathbf{n}_{f}\rangle_{\Gamma}
-\langle \widetilde{\Theta}_{p_{p}}^{n}
,\mathbf{w}_{f,h}\cdot\mathbf{n}_{f}\rangle_{\Gamma}
-\langle \widetilde{\Lambda}_{p_{p}}^{n}
,\mathbf{w}_{f,h}\cdot\mathbf{n}_{f}\rangle_{\Gamma}\nonumber\\
&-c_{BJS}\langle \mathcal{M}_{p}(d_{t}\Theta_{\mathbf{u}_{p}}^{n})
,\mathcal{M}_{f}(\mathbf{w}_{f,h})\rangle_{\Gamma}
-c_{BJS}\langle \mathcal{M}_{p}(d_{t}\Lambda_{\mathbf{u}_{p}}^{n})
,\mathcal{M}_{f}(\mathbf{w}_{f,h})\rangle_{\Gamma}\nonumber\\
 &+c_{BJS}\langle\mathcal{M}_{p}(d_{t}\mathbf{u}_{p}(t_{n})
 -\partial_{t}\mathbf{u}_{p}(t_{n}))
 ,\mathcal{M}_{f}(\mathbf{w}_{f,h})\rangle_{\Gamma}
 -\rho_{f}\big(d_{t}\Lambda_{\mathbf{v}_{f}}^{n+1}
 ,\mathbf{w}_{f,h}\big)_{\Omega_{f}}\nonumber\\
&+\rho_{f}\big(d_{t}\mathbf{v}_{f}(t_{n+1})
-\partial_{t}\mathbf{v}_{f}(t_{n+1})
,\mathbf{w}_{f,h}\big)_{\Omega_{f}}
-c_{BJS}\langle\mathcal{M}_{f}(\Lambda_{\mathbf{v}_{f}}^{n+1})
,\mathcal{M}_{f}(\mathbf{w}_{f,h})\rangle_{\Gamma},\nonumber\\
 &\big(\nabla\cdot \Theta_{\mathbf{v}_{f}}^{n+1},q_{f,h}\big)_{\Omega_{f}}=0,\label{6.14}
 \end{align}
 Similarly, for the equations \eqref{3.3}-\eqref{3.6} at $t_{n+1}$ and equations \eqref{4.4}-\eqref{4.7}, we find that the following equations hold:
 \begin{align}
&\big(\Theta_{\mathbf{v}_{p}}^{n+1}
,\mathbf{z}_{p,h}\big)_{\Omega_{p}}
-\big(d_{t}\Theta_{\mathbf{u}_{p}}^{n+1}
,\mathbf{z}_{p,h}\big)_{\Omega_{p}}\label{6.15}\\
&=\big(\partial_{t}\mathbf{u}_{p}(t_{n+1})
-d_{t}\mathbf{u}_{p}(t_{n+1})
,\mathbf{z}_{p,h}\big)_{\Omega_{p}}
-\big(\Lambda_{\mathbf{v}_{p}}^{n+1}
,\mathbf{z}_{p,h}\big)_{\Omega_{p}}\nonumber\\
&\quad+\big(d_{t}\Lambda_{\mathbf{u}_{p}}^{n+1}
,\mathbf{z}_{p,h}\big)_{\Omega_{p}},\nonumber\\
&\rho_{p}\big(d_{t}\Theta_{\mathbf{v}_{p}}^{n+1}
,\mathbf{w}_{p,h}\big)_{\Omega_{p}}
+2\mu_{p}\big(\varepsilon(\Theta_{\mathbf{u}_{p}}^{n+1})
,\varepsilon(\mathbf{w}_{p,h})\big)_{\Omega_{p}}
-\big(\Theta_{\beta_{p}}^{n+1}
,\nabla\cdot\mathbf{w}_{p}\big)_{\Omega_{p}}\label{6.16}\\
&\quad+L_{2}\langle (d_{t}\Theta_{\mathbf{u}_{p}}^{n+1}
-d_{t}\Theta_{\mathbf{u}_{p}}^{n})\cdot\mathbf{n}_{p}
,\mathbf{w}_{p,h}\cdot\mathbf{n}_{p}\rangle_{\Gamma}\nonumber\\
&\quad+c_{BJS}\langle\mathcal{M}_{p}(d_{t}\Theta_{\mathbf{u}_{p}}^{n+1})
,\mathcal{M}_{p}(\mathbf{w}_{p,h})\rangle_{\Gamma}\nonumber\\
&=L_{2}\langle (d_{t}\Lambda_{\mathbf{u}_{p}}^{n}
-d_{t}\Lambda_{\mathbf{u}_{p}}^{n+1})\cdot\mathbf{n}_{p}
,\mathbf{w}_{p,h}\cdot\mathbf{n}_{p}\rangle_{\Gamma}-\langle \widetilde{\Lambda}_{p_{p}}^{n}
,\mathbf{w}_{p,h}\cdot\mathbf{n}_{p}\rangle_{\Gamma}\nonumber\\
&\quad+L_{2}\langle(d_{t}\mathbf{u}_{p}(t_{n+1})
-\partial_{t}\mathbf{u}_{p}(t_{n+1}))\cdot\mathbf{n}_{p}
,\mathbf{w}_{p,h}\cdot\mathbf{n}_{p}\rangle_{\Gamma}\nonumber\\
&\quad+L_{2}\langle(\partial_{t}\mathbf{u}_{p}(t_{n})
-d_{t}\mathbf{u}_{p}(t_{n}))\cdot\mathbf{n}_{p}
,\mathbf{w}_{p,h}\cdot\mathbf{n}_{p}\rangle_{\Gamma}
-\langle\widetilde{\Theta}_{p_{p}}^{n}
,\mathbf{w}_{p,h}\cdot\mathbf{n}_{p}\rangle_{\Gamma}\nonumber\\
&\quad-c_{BIS}\langle\mathcal{M}_{f}(\Theta_{\mathbf{v}_{f}}^{n})
,\mathcal{M}_{p}(\mathbf{w}_{p,h})\rangle_{\Gamma}-c_{BJS}\langle\mathcal{M}_{f}(\Lambda_{\mathbf{v}_{f}}^{n})
,\mathcal{M}_{p}(\mathbf{w}_{p,h})\rangle_{\Gamma}\nonumber\\
&\quad+\rho_{p}\big(d_{t}\mathbf{v}_{p}(t_{n+1})
-\partial_{t}\mathbf{v}_{p}(t_{n+1}),\mathbf{w}_{p,h}\big)_{\Omega_{p}}
-\rho_{p}\big(d_{t}\Lambda_{\mathbf{v}_{p}}^{n+1}
,\mathbf{w}_{p,h}\big)_{\Omega_{p}}\nonumber\\
&\quad+c_{BJS}\langle\mathcal{M}_{p}(d_{t}\mathbf{u}_{p}(t_{n+1})
-\partial_{t}\mathbf{u}_{p}(t_{n+1}))
,\mathcal{M}_{p}(\mathbf{w}_{p,h})\rangle_{\Gamma}\nonumber\\
&\quad-c_{BJS}\langle\mathcal{M}_{p}(d_{t}\Lambda_{\mathbf{u}_{p}}^{n+1})
,\mathcal{M}_{p}(\mathbf{w}_{p,h})\rangle_{\Gamma},\nonumber\\
 &\frac{1}{\lambda_{p}}\big(\Theta_{\beta_{p}}^{n+1}
 ,\varphi_{p,h}\big)_{\Omega_{p}}
 +\big(\nabla\cdot\Theta_{\mathbf{u}_{p}}^{n+1}
 ,\varphi_{p,h}\big)_{\Omega_{p}}\label{6.17}\\
 &=\frac{\alpha}{\lambda_{p}}\big(\widetilde{\Theta}_{p_{p}}^{n+1}
+\widetilde{\Lambda}_{p_{p}}^{n+1}
,\varphi_{p,h}\big)_{\Omega_{p}}
-\frac{1}{\lambda_{p}}\big(\Lambda_{\beta_{p}}^{n+1}
,\varphi_{p,h}\big)_{\Omega_{p}},\nonumber\\
&\big(c_{0}+\frac{\alpha^{2}}{\lambda_{p}}\big)\big(d_{t}\widetilde{\Theta}_{p_{p}}^{n+1}
,\psi_{p,h}\big)_{\Omega_{p}}
-\frac{\alpha}{\lambda_{p}}\big(d_{t}\Theta_{\beta_{p}}^{n+1}
,\psi_{p,h}\big)_{\Omega_{p}}
\label{6.18}\\
&\quad+\mu_{f}\big(K^{-1}\nabla\widetilde{\Theta}_{p_{p}}^{n+1}
,\nabla\psi_{p,h}\big)_{\Omega_{p}}+L_{3}\langle\widetilde{\Theta}_{p_{p}}^{n+1}
-\widetilde{\Theta}_{p_{p}}^{n}
,\psi_{p,h}\rangle_{\Gamma}\nonumber\\
&=L_{3}\langle\widetilde{\Lambda}_{p_{p}}^{n}
-\widetilde{\Lambda}_{p_{p}}^{n+1}
,\psi_{p,h}\rangle_{\Gamma}
+\langle\Theta_{\mathbf{v}_{f}}^{n}\cdot\mathbf{n}_{f}
,\psi_{p,h}\rangle_{\Gamma}
+\langle\Lambda_{\mathbf{v}_{f}}^{n}\cdot\mathbf{n}_{f}
,\psi_{p,h}\rangle_{\Gamma}\nonumber\\
&\quad+\langle d_{t}\Lambda_{\mathbf{u}_{p}}^{n}\cdot\mathbf{n}_{p}
,\psi_{p,h}\rangle_{\Gamma}
+\langle(\partial_{t}\mathbf{u}_{p}(t_{n})
-d_{t}\mathbf{u}_{p}(t_{n}))\cdot\mathbf{n}_{p}
,\psi_{p,h}\rangle_{\Gamma}\nonumber\\
&\quad+\langle d_{t}\Theta_{\mathbf{u}_{p}}^{n}\cdot\mathbf{n}_{p}
,\psi_{p,h}\rangle_{\Gamma}+\big(c_{0}+\frac{\alpha^{2}}{\lambda_{p}}\big)\big(d_{t}p_{p}(t_{n+1})
-\partial_{t}p_{p}(t_{n+1}),\psi_{p,h}\big)_{\Omega_{p}}
\nonumber\\
&\quad+\frac{\alpha}{\lambda_{p}}\big(\partial_{t}\beta_{p}(t_{n+1})
-d_{t}\beta_{p}(t_{n+1}),\psi_{p,h}\big)_{\Omega_{p}}\nonumber\\
&\quad-\big(c_{0}+\frac{\alpha^{2}}{\lambda_{p}}\big)\big(d_{t}\widetilde{\Lambda}_{p_{p}}^{n+1}
,\psi_{p,h}\big)_{\Omega_{p}}
+\frac{\alpha}{\lambda_{p}}\big(\Lambda_{\beta_{p}}^{n+1}
,\psi_{p,h}\big)_{\Omega_{p}}.\nonumber
\end{align}
Taking $(\mathbf{w}_{f,h},q_{f,h}) = (\Theta_{\mathbf{v}_{f}}^{n+1},\Theta_{p_{f}}^{n+1})$ as the test functions in \eqref{6.13}-\eqref{6.14}, and $(\mathbf{z}_{p,h},\mathbf{w}_{p,h},\varphi_{p,h},\psi_{p,h})=(\rho_{p}d_{t}\Theta_{\mathbf{v}_{p}}^{n+1},d_{t}\Theta_{\mathbf{u}_{p}}^{n+1},\Theta_{\beta_{p}}^{n+1},\widetilde{\Theta}_{p_{p}}^{n+1})$ as the test functions in \eqref{6.15}-\eqref{6.18} after applying the discrete time difference operator $d_t$ to \eqref{6.17}, then adding the resulting equalities with \eqref{5.1}, and subsequently applying the summation operator $\Delta t \sum_{n=0}^{l}$, we deduce the desired equation \eqref{6.6}. This completes the proof. 
\end{proof}

To complete the proof of the main theorem, it is convenient to define the following  notation:
\begin{align*}
C_{1}(T)=&\|\partial_{t}\mathbf{v}_{f}\|_{L^{2}((0,T);H^{3}(\Omega_{f}))}^{2}+\|\partial_{t}p_{p}\|_{L^{2}((0,T);H^{3}(\Omega_{p}))}^{2}+\|\partial_{t}\beta_{p}\|_{L^{2}((0,T);H^{2}(\Omega_{p}))}^{2}\\
       &+\|\beta_{p}\|_{L^{2}((0,T);H^{2}(\Omega_{p}))}^{2}+\|\mathbf{v}_{f}\|_{L^{\infty}((0,T);H^{3}(\Omega_{f}))}^{2}+\|p_{p}\|_{L^{\infty}((0,T);H^{3}(\Omega_{p}))}^{2}\\
       &+\|\partial_{t}\mathbf{u}_{p}\|_{L^{\infty}((0,T);H^{3}(\Omega_{p}))}^{2}+\|\mathbf{v}_{p}\|_{L^{\infty}((0,T);H^{3}(\Omega_{p}))}^{2}+\|\partial_{t}\mathbf{v}_{p}\|_{L^{\infty}((0,T);H^{3}(\Omega_{p}))}^{2}\\
       &+\|\partial_{tt}^{2}\mathbf{u}_{p}\|_{H^{3}((0,T);H^{3}(\Omega_{p}))}^{2}+\|\partial_{tt}^{2}\mathbf{v}_{p}\|_{H^{3}((0,T);H^{3}(\Omega_{p}))}^{2},\\
       C_{2}(T)=&\|\partial_{tt}^{2}\mathbf{v}_{f}\|_{L^{2}((0,T);L^{2}(\Omega_{f}))}^{2}+\|\partial_{tt}^{2}\mathbf{u}_{p}\|_{L^{\infty}((0,T);L^{2}(\Omega_{p}))}^{2}+\|\partial_{tt}^{2}\mathbf{v}_{p}\|_{L^{\infty}((0,T);L^{2}(\Omega_{p}))}^{2}\\
&+\|\partial_{tt}^{2}p_{p}\|_{L^{2}((0,T);L^{2}(\Omega_{p}))}^{2}+\|\partial_{tt}^{2}\beta_{p}\|_{L^{2}((0,T);L^{2}(\Omega_{p}))}^{2}+\|\partial_{tt}^{2}\mathbf{u}_{p}\|_{L^{2}((0,T);L^{2}(\Gamma))}^{2}\\
&+\|\partial_{tt}^{2}\mathbf{u}_{p}\|_{L^{\infty}((0,T);L^{2}(\Gamma))}^{2}.
\end{align*}
We now establish the optimal error estimates in the following theorem.
\begin{theorem}\label{thm:final_error}
Let $(\mathbf{v}_{f}, p_{f}, \mathbf{v}_{p}, \mathbf{u}_{p}, \beta_{p}, p_{p})$ be the exact solution of \eqref{3.1}--\eqref{3.6} and $(\mathbf{v}_{f,h}^{n}, p_{f,h}^{n}, \mathbf{v}_{p,h}^{n}, \mathbf{u}_{p,h}^{n},\\ \beta_{p,h}^{n}, p_{p,h}^{n})$ be the numerical solution obtained from Algorithm \ref{al1}. Under sufficient regularity assumptions on the exact solution, the following error estimate holds for any $1 \le n \le N$:
\begin{align}
&\max_{1\leq n\leq N}\Big[\mu_{p}\|\varepsilon(\mathbf{u}_{p}(t_{n})-\mathbf{u}_{p,h}^{n})\|_{L^{2}(\Omega_{p})}
+\frac{1}{2\lambda}\|\beta_{p}(t_{n})-\beta_{p,h}^{n}\|_{L^{2}(\Omega_{p})}\Big]
\label{6.22}\\
&\quad+\Big[\mu_{f}\Delta t\sum_{n=0}^{N}\|\varepsilon(\mathbf{v}_{f}(t_{n})-\mathbf{v}_{f,h}^{n})\|_{L^{2}(\Omega_{f})}^{2}\Big]^{1/2}\nonumber\\
&\quad+\Big[\mu_{f}\Delta t\sum_{n=0}^{N}\|K^{-1/2}\nabla(p_{p}(t_{n}) - p_{p,h}^{n})\|_{L^{2}(\Omega_{p})}^{2}\Big]^{1/2}\nonumber\\
&\lesssim e^{\overline{C}T}\Big[(L_{1}+L_{2}+L_{3}\Delta t+C)\sqrt{C_{1}(T)}h^{2}+C_{2}(L_{2}+C)\sqrt{C_{2}(T)}\Delta t\Big],\nonumber
\end{align}
where the positive constant $\overline{C}$ depends on the Robin parameters $L_{1}, L_{2}, L_{3}$.
\end{theorem}
\begin{proof}
Note that the initial errors vanish, i.e., $\mathcal{E}_{\Theta}^{0} = \mathcal{N}_{\Theta}^{0} = \mathcal{M}_{\Theta}^{0} = 0$, due to the choice of the initial projections. We then proceed to estimate the terms $\Phi_{i}$ individually. For $\Phi_{1}$, by applying the Cauchy–Schwarz, Poincar\'e, Korn, and Young inequalities, we obtain:
\begin{align}
\Phi_{1}\leq
& C\big(\|d_{t}\Lambda_{\mathbf{v}_{f}}^{n+1}\|_{L^{2}(\Omega_{f})}^{2}
+\|d_{t}\widetilde{\Lambda}_{p_{p}}^{n+1}\|_{L^{2}(\Omega_{p})}^{2}
+\|d_{t}\Lambda_{\beta_{p}}^{n+1}\|_{L^{2}(\Omega_{p})}^{2}
\label{6.23}\\
&+\|\Lambda_{\beta_{p}}^{n+1}\|_{L^{2}(\Omega_{p})}^{2}\big)+\frac{1}{4\lambda_{p}}\|\alpha\widetilde{\Theta}_{p_{p}}^{n+1}-\Theta_{\beta_{p}}^{n+1}\|_{L^{2}(\Omega_{p})}^{2}\nonumber\\
&+\frac{\mu_{f}}{4}\|\varepsilon(\Theta_{\mathbf{v}_{f}}^{n+1})\|_{L^{2}(\Omega_{f})}^{2}
+\frac{c_{0}}{4}\|\widetilde{\Theta}_{p_{p}}^{n+1}\|_{L^{2}(\Omega_{p})}^{2}.\nonumber
\end{align}
By virtue of the Cauchy–Schwarz, trace, Poincar\'e, Korn, and Young inequalities, the term $\Phi_{2}$ can be bounded as follows:
\begin{align}
\Phi_{2}\leq
&C\big((L_{1}^{2}+1)\|\Lambda_{\mathbf{v}_{f}}^{n}\|_{H^{1}(\Omega_{f})}^{2}
+(L_{1}^{2}+1)\|\Lambda_{\mathbf{v}_{f}}^{n+1}\|_{H^{1}(\Omega_{f})}^{2}
+\|\widetilde{\Lambda}_{p_{p}}^{n}\|_{H^{1}(\Omega_{p})}^{2}\label{6.24}\\
&+L_{3}^{2}\Delta t^{2}\|d_{t}\widetilde{\Lambda}_{p_{p}}^{n+1}\|_{H^{1}(\Omega_{p})}^{2}
+\|d_{t}\Lambda_{\mathbf{u}_{p}}^{n}\|_{H^{1}(\Omega_{p})}^{2}\big)
+\frac{\mu_{f}}{4}\|\varepsilon(\Theta_{\mathbf{v}_{f}}^{n+1})\|_{L^{2}(\Omega_{f})}^{2}
\nonumber\\
&+\frac{1}{4\mu_{f}}\|K^{\frac{1}{2}}\nabla\widetilde{\Theta}_{p_{p}}^{n+1}\|_{L^{2}(\Omega_{p})}^{2}.\nonumber
\end{align}
For the term:
\begin{align}
\Phi_{3}=&-\langle \widetilde{\Theta}_{p_{p}}^{n},(\Theta_{\mathbf{v}_{f}}^{n+1}-\Theta_{\mathbf{v}_{f}}^{n})\cdot\mathbf{n}_{f}\rangle_{\Gamma}-\langle\widetilde{\Theta}_{p_{p}}^{n},(d_{t}\Theta_{\mathbf{u}_{p}}^{n+1}-d_{t}\Theta_{\mathbf{u}_{p}}^{n})\cdot\mathbf{n}_{p}\rangle_{\Gamma}\nonumber\\
&+\langle\Theta_{\mathbf{v}_{f}}^{n}\cdot\mathbf{n}_{f},\widetilde{\Theta}_{p_{p}}^{n+1}-\widetilde{\Theta}_{p_{p}}^{n}\rangle_{\Gamma}+\langle d_{t}\Theta_{\mathbf{u}_{p}}^{n}\cdot\mathbf{n}_{p},\widetilde{\Theta}_{p_{p}}^{n+1}-\widetilde{\Theta}_{p_{p}}^{n}\rangle_{\Gamma},\nonumber
\end{align}
a combination of the Cauchy-Schwarz, trace, Korn, and Young inequalities leads to the estimate:
\begin{align}
\Phi_{3}
\leq&\frac{\mu_{f}}{4}\|\varepsilon(\Theta_{\mathbf{v}_{f}}^{n})\|_{L^{2}(\Omega_{f})}^{2}
+\frac{1}{4\mu_{f}}\|K^{\frac{1}{2}}\nabla \widetilde{\Theta}_{p_{p}}^{n}\|_{L^{2}(\Omega_{p})}^{2}
+\frac{\check{C}}{L_{3}}\|d_{t}\varepsilon(\Theta_{\mathbf{u}_{p}}^{n})\|_{L^{2}(\Omega_{p})}^{2}\label{6.25}\\
&+C\Big[\frac{1}{L_{3}^{2}\mu_{f}}\|\Theta_{\mathbf{v}_{f}}^{n}\|_{L^{2}(\Omega_{f})}^{2}+\frac{\mu_{f}}{4K_{2}}(\frac{1}{L_{1}^{2}}+\frac{1}{L_{2}^{2}})\|\widetilde{\Theta}_{p_{p}}^{n}\|_{L^{2}(\Omega_{p})}^{2}\Big]\nonumber\\
&+\frac{L_{1}}{2}\|(\Theta_{\mathbf{v}_{f}}^{n+1}
-\Theta_{\mathbf{v}_{f}}^{n})\cdot\mathbf{n}_{f}\|_{L^{2}(\Gamma)}^{2}+\frac{L_{2}}{2}\|(d_{t}\Theta_{\mathbf{u}_{p}}^{n+1}
-d_{t}\Theta_{\mathbf{u}_{p}}^{n})\cdot\mathbf{n}_{p}\|_{L^{2}(\Gamma)}^{2}\nonumber\\
&+\frac{L_{3}}{2}\|\widetilde{\Theta}_{p_{p}}^{n+1}
-\widetilde{\Theta}_{p_{p}}^{n}\|_{L^{2}(\Gamma)}^{2}.\nonumber
\end{align}
For $\Phi_{4}$, by employing the Cauchy–Schwarz inequality, the following temporal approximation properties:
\begin{align}
&\|d_{t}\mathcal{C}_{*}(t_{n+1})-\partial_{t}\mathcal{C}_{*}(t_{n+1})\|_{L^{2}(\Omega_{*})}^{2}\leq\frac{\Delta t}{3}\|\partial_{tt}^{2}\mathcal{C}_{*}\|_{L^{2}((t_{n},t_{n+1});L^{2}(\Omega_{*}))}^{2},\nonumber\\
&\|d_{t}\mathcal{C}_{*}(t_{n+1})-\partial_{t}\mathcal{C}_{*}(t_{n+1})\|_{L^{2}(\Omega_{*})}\leq\frac{\Delta t}{2}\|\partial_{tt}^{2}\mathcal{C}_{*}\|_{L^{\infty}((t_{n},t_{n+1});L^{2}(\Omega_{*}))},\nonumber
\end{align}
where $\mathcal{C}_{*}=\mathbf{v}_{f},~\mathbf{u}_{p},~\mathbf{v}_{p},~p_{p},~\beta_{p}$, together with the Poincar\'e, Korn, and Young inequalities, we arrive at:
\begin{align}
\Phi_{4}
&\leq C\Delta t\big(\|\partial_{tt}^{2}\mathbf{v}_{f}\|_{L^{2}((t_{n},t_{n+1});L^{2}(\Omega_{f}))}^{2}
+\|\partial_{tt}^{2}\mathbf{u}_{p}\|_{L^{\infty}((t_{n},t_{n+1});L^{2}(\Omega_{p}))}^{2}\label{6.26}\\
&+\|\partial_{tt}^{2}\mathbf{v}_{p}\|_{L^{\infty}((t_{n},t_{n+1});L^{2}(\Omega_{p}))}^{2}
+\|\partial_{tt}^{2}p_{p}\|_{L^{2}((t_{n},t_{n+1});L^{2}(\Omega_{p}))}^{2}\nonumber\\
&
+\|\partial_{tt}^{2}\beta_{p}\|_{L^{2}((t_{n},t_{n+1});L^{2}(\Omega_{p}))}^{2}\big)+\frac{\mu_{f}}{8}\|K^{-\frac{1}{2}}\nabla\widetilde{\Theta}_{p_{p}}^{n+1}\|_{L^{2}(\Omega_{p})}^{2}\nonumber\\
&+\frac{\rho_{p}\Delta t}{2}\|d_{t}\Theta_{\mathbf{v}_{p}}^{n+1}\|_{L^{2}(\Omega_{p})}^{2}+\frac{\mu_{f}}{4}\|\varepsilon(\Theta_{\mathbf{v}_{f}}^{n+1})\|_{L^{2}(\Omega_{f})}^{2}
+\frac{\mu_{p}\Delta t}{4}\|d_{t}\varepsilon(\Theta_{\mathbf{u}_{p}}^{n+1})\|_{L^{2}(\Omega_{f})}^{2}.\nonumber
\end{align}
Similarly, we bound $\Phi_{5}$ as:
\begin{align}
\sum_{n=0}^{l}\Phi_{5}\leq& C\Delta t \big(\|\partial_{tt}^{2}\mathbf{u}_{p}\|_{L^{2}((0,T);L^{2}(\Gamma))}^{2}+(L_{2}^{2}+1)\|\partial_{tt}^{2}\mathbf{u}_{p}\|_{L^{\infty}((0,T);L^{2}(\Gamma))}^{2}\big)\label{6.27}\\
&+\sum_{n=0}^{l}\big(\frac{\mu_{f}}{4}\|\varepsilon(\Theta_{\mathbf{v}_{f}}^{n+1})\|_{L^{2}(\Omega_{f})}^{2}+\frac{\mu_{p}\Delta t}{4}\|d_{t}\varepsilon(\Theta_{\mathbf{u}_{p}}^{n+1})\|_{L^{2}(\Omega_{f})}^{2}\nonumber\\
&+\frac{\mu_{f}}{8}\|K^{-\frac{1}{2}}\nabla\widetilde{\Theta}_{p_{p}}^{n+1}\|_{L^{2}(\Omega_{p})}^{2}\big).\nonumber
\end{align}
For $\Phi_{6}$, by applying the summation by parts formula and utilizing the fact that the initial errors vanish, i.e., $\Theta_{\mathbf{v}_{p}}^{0} = \Theta_{\mathbf{u}_{p}}^{0} = \mathbf{0}$, we obtain:
\begin{align}
\sum_{n=0}^{l}\Phi_{6}=&\frac{1}{\Delta t}\big(-\rho_{p}\big(\Lambda_{\mathbf{v}_{p}}^{l+1},\Theta_{\mathbf{v}_{p}}^{l+1}\big)_{\Omega_{p}}+\rho_{p}\big(d_{t}\Lambda_{\mathbf{u}_{p}}^{l+1},\Theta_{\mathbf{v}_{p}}^{l+1}\big)_{\Omega_{p}}\label{6.28}\\
&-\rho_{p}\big(d_{t}\Lambda_{\mathbf{v}_{p}}^{l+1},\Theta_{\mathbf{u}_{p}}^{l+1}\big)_{\Omega_{p}}\big)+\sum_{n=1}^{l}\big(\rho_{p}\big(d_{t}\Lambda_{\mathbf{v}_{p}}^{n+1},\Theta_{\mathbf{v}_{p}}^{n}\big)_{\Omega_{p}}\nonumber\\
&-\rho_{p}\big(d_{t}^{2}\Lambda_{\mathbf{u}_{p}}^{n+1},\Theta_{\mathbf{v}_{p}}^{n}\big)_{\Omega_{p}}+\rho_{p}\big(d_{t}^{2}\Lambda_{\mathbf{v}_{p}}^{n+1},\Theta_{\mathbf{u}_{p}}^{n}\big)_{\Omega_{p}}\big).\nonumber
\end{align}
Applying the Cauchy-Schwarz,  Poincar\'e, Korn and Young inequalities for \eqref{6.28}, we arrive at:
\begin{align}
\sum_{n=0}^{l}\Phi_{6}\leq& \frac{C}{\Delta t}\Big[\|\Lambda_{\mathbf{v}_{p}}^{l+1}\|_{L^{2}(\Omega_{p})}^{2}+\|d_{t}\Lambda_{\mathbf{u}_{p}}^{l+1}\|_{L^{2}(\Omega_{p})}^{2}+\|d_{t}\Lambda_{\mathbf{v}_{p}}^{l+1}\|_{L^{2}(\Omega_{p})}^{2}\label{6.29}\\
&+\Delta t\sum_{n=1}^{l}\big(\|d_{t}\Lambda_{\mathbf{v}_{p}}^{n+1}\|_{L^{2}(\Omega_{p})}^{2}+\|d_{t}^{2}\Lambda_{\mathbf{u}_{p}}^{n+1}\|_{L^{2}(\Omega_{p})}^{2}+\|d_{t}^{2}\Lambda_{\mathbf{v}_{p}}^{n+1}\|_{L^{2}(\Omega_{p})}^{2}\big)\Big]\nonumber\\
&+\frac{1}{\Delta t}\Big[\frac{\rho_{p}}{4}\|\Theta_{\mathbf{v}_{p}}^{l+1}\|_{L^{2}(\Omega_{p})}^{2}+\frac{\mu_{p}}{4}\|\varepsilon(\Theta_{\mathbf{u}_{p}}^{l+1})\|_{L^{2}(\Omega_{p})}^{2}\nonumber\\
&+\Delta t\sum_{n=1}^{l}\big(\frac{\rho_{p}}{4}\|\Theta_{\mathbf{v}_{p}}^{n+1}\|_{L^{2}(\Omega_{p})}^{2}+\frac{\mu_{p}}{4}\|\varepsilon(\Theta_{\mathbf{u}_{p}}^{n+1})\|_{L^{2}(\Omega_{p})}^{2}\big)\Big].\nonumber
\end{align}
Similar to $\Phi_{6}$, $\Phi_{7}$ can be bounded: 
\begin{align}
   \sum_{n=0}^{l}\Phi_{7}\leq&\frac{C}{\Delta t}\Big[L_{2}^{2}\|d_{t}\Lambda_{\mathbf{u}_{p}}^{l}\|_{H^{1}(\Omega_{p})}^{2}+(L_{2}^{2}+1)\|d_{t}\Lambda_{\mathbf{u}_{p}}^{l+1}\|_{H^{1}(\Omega_{p})}^{2}+\|\widetilde{\Lambda}_{p_{p}}^{l}\|_{H^{1}(\Omega_{p})}^{2}\label{6.30}\\
   &+\|\Lambda_{\mathbf{v}_{f}}^{l}\|_{H^{1}(\Omega_{f})}^{2}+\Delta t\sum_{n=1}^{l}\big((L_{2}^{2}+1)\|d_{t}^{2}\Lambda_{\mathbf{u}_{p}}^{n+1}\|_{H^{1}(\Omega_{p})}^{2}\nonumber\\
   &+\|d_{t}\widetilde{\Lambda}_{p_{p}}^{n}\|_{H^{1}(\Omega_{p})}^{2}+\|d_{t}\Lambda_{\mathbf{v}_{f}}^{n}\|_{H^{1}(\Omega_{f})}^{2}\big)\Big]+\frac{1}{\Delta t}\big(\frac{\mu_{p}}{4}\|\varepsilon(\Theta_{\mathbf{u}_{p}}^{l+1})\|_{L^{2}(\Omega_{p})}^{2}\nonumber\\
   &+\frac{\mu_{p}\Delta t}{4}\sum_{n=1}^{l}\|\varepsilon(\Theta_{\mathbf{u}_{p}}^{n+1})\|_{L^{2}(\Omega_{p})}^{2}\big).\nonumber
\end{align}
Substituting \eqref{6.23}-\eqref{6.27} and \eqref{6.29}-\eqref{6.30} into \eqref{6.6}, and applying the Poincar\'e inequality along with the approximation properties \eqref{6.3} and \eqref{6.5}, we finally arrive at:
\begin{align}
       &\frac{1}{2}\mathcal{E}_{\Theta}^{l+1}+\Delta t(\mathcal{N}_{\Theta}^{l+1}+\mathcal{M}_{\Theta}^{l+1})+\Delta t\sum_{n=0}^{l}\Big[\frac{\mu_{f}}{2}\big(\|\varepsilon(\Theta_{\mathbf{v}_{f}}^{n+1})\|_{L^{2}(\Omega_{f})}^{2}\label{6.31}\\
       &\quad+\|K^{-\frac{1}{2}}\nabla\widetilde{\Theta}_{p_{p}}^{n+1}\|_{L^{2}(\Omega_{p})}^{2}\big)+\mathcal{L}_{\Theta}^{n+1}\Big]\nonumber\\
       &\leq \frac{\overline{C}\Delta t}{2}\sum_{n=0}^{l}\mathcal{E}_{\Theta}^{n+1}+(L_{1}^{2}+L_{2}^{2}+L_{3}^{2}\Delta t^{2}+C)C_{1}(T)h^{4}\nonumber\\
       &\quad+(L_{2}^{2}+C)C_{2}(T)\Delta t^{2}.\nonumber
    \end{align}
    Note that by choosing the Robin parameter $L_3$ such that $\frac{\check{C}}{L_3} \leq \mu_{p}\Delta t$, and then applying the discrete Grönwall inequality to \eqref{6.31}, we obtain the following estimate after rearranging the terms:
    \begin{align}
    &\mathcal{E}_{\Theta}^{l+1}+\mu_{f}\Delta t\sum_{n=0}^{l}\big(\|\varepsilon(\Theta_{\mathbf{v}_{f}}^{n+1})\|_{L^{2}(\Omega_{f})}^{2}+\|K^{-\frac{1}{2}}\nabla\widetilde{\Theta}_{p_{p}}^{n+1}\|_{L^{2}(\Omega_{p})}^{2}\big)\label{6.32}\\
    &\leq e^{\overline{C}T}\Big[(L_{1}^{2}+L_{2}^{2}+L_{3}^{2}\Delta t^{2}+C)C_{1}(T)h^{4}+(L_{2}^{2}+C)C_{2}(T)\Delta t^{2}\Big].\nonumber
    \end{align}
Finally, applying the triangle inequality to the split errors $e_{\mathcal{A}_{*}}^{n} = \Lambda_{\mathcal{A}_{*}}^{n} + \Theta_{\mathcal{A}_{*}}^{n}$ and $e_{\mathcal{B}_{p}}^{n} = \widetilde{\Lambda}_{\mathcal{B}_{p}}^{n} + \widetilde{\Theta}_{\mathcal{B}_{p}}^{n}$, and combining \eqref{6.3}, \eqref{6.5} with the estimate \eqref{6.32}, we conclude that \eqref{6.22} holds. The proof is thus complete.
\end{proof}
\begin{remark}
The error estimate in Theorem \ref{thm:final_error} further demonstrates the parameters robustness of the proposed decoupled scheme. Specifically, the convergence of the displacement $\mathbf{u}_{p}$ is independent of the Lam\'e parameter $\lambda_{p}$, and similarly, the convergence of the pore pressure $p_{p}$ is not affected by the specific value of storage coefficient $c_{0}$. This robustness provides a theoretical explanation for the Algorithm \ref{al1}'s ability to overcome the locking phenomenon.
\end{remark}
\section{Numerical tests}
This section presents several numerical examples to demonstrate the capability of the proposed method in overcoming the locking phenomenon and achieving optimal convergence rates. Furthermore, the robustness and accuracy of the proposed scheme are demonstrated through a classical FPSI benchmark problem. Specifically, we investigate the impact of various Robin parameters on the simulation results and perform a detailed comparison with the results obtained from the monolithic method. Spatial discretization is carried out via the finite element method, employing $\mathbf{P}_{2}-P_{1}$ elements for the fluid variables and $\mathbf{P}_{2}-\mathbf{P}_{2}-P_{1}-P_{2}$ elements for the poroelastic variables. All simulations are implemented using the FEniCS platform \cite{alnaes2015fenics,alnaes2014unified}.

\subsection{Robustness against Poisson-Locking}
This example is designed to verify the locking-free property of the proposed scheme in the nearly incompressible limit, where the Lam\'{e} parameter $\lambda_{p}$ takes significantly large values \cite{yi2017study}.  We consider a couple system \eqref{2.1}-\eqref{2.16} where the exact solutions are constructed such that the interface conditions are satisfied regardless of the magnitude  of $\lambda_{p}$. The computational domain $\Omega$ is partitioned into a fluid domain $\Omega_{f} = [0, 1]\times [0, 1]$ and a poroelastic domain $\Omega_{p} = [0, 1] \times [-1, 0]$, separated by a common interface $\Gamma = \{(x, y)| 0\leq x \leq 1, y = 0 \}$. For simplicity, all model parameters except for $\lambda_{p}$ are set to unity: $\rho_{f} = \rho_{p} = \mu_{f} = \mu_{p} = \alpha = c_{0},~K = 1.0\mathbf{I}$.
The forcing term $\mathbf{f}_{f},~\phi_{f},~ \mathbf{f}_{p}$ and $\phi_{p}$ are derived directly from the exact solutions. Notably, $\phi_{f} \neq 0$ is maintained as the prescribed fluid velocity $\mathbf{v}_{f}$ is not divergence-free. The exact solutions are given by:
\begin{align*}
    &\mathbf{v}_{f} = \left[ 
    \begin{aligned}	
    e^{t}\cos(2\pi x)\sin(2\pi y) \\
    e^{t}\big{(}\frac{\cos(y)}{\lambda + 1} - \cos(2\pi y)\sin(2\pi x)\big{)}
    \end{aligned} \right], \quad p_{f} = e^{t}\cos(\pi x)\sin(\pi y) + \frac{\lambda e^{t}\sin(y)}{\lambda + 1} \\
    &\mathbf{u}_{p} = \left[ 
    \begin{aligned}
    e^{t}\cos(2\pi x)\sin(2\pi y) \\
    e^{t}\big{(}\frac{\cos(y)}{\lambda + 1} - \cos(2\pi y)\sin(2\pi x)\big{)}
    \end{aligned} \right], \quad p_{p} = e^{t}\cos(\pi x)\sin(\pi y).
\end{align*}
Dirichlet boundary conditions are imposed on all external boundaries, and the simulation is integrated until the final time $t = 1.0 s$. To evaluate the convergence rate, we define the following error norms for the primary variables:
\begin{align*}
&\mathbf{e}_{v_{f}, H^{1}} = ||\varepsilon(\mathbf{v}_{f} - \mathbf{v}_{f, h})||_{\Omega_{f}}, ~
e_{p_{f}, L^{2}} = ||p_{f} - p_{f, h}||_{\Omega_{f}},\\
&\mathbf{e}_{eng} = 2\mu_{p}||\varepsilon(\mathbf{u}_{p} - \mathbf{u}_{p,h})||_{\Omega_{p}} + \lambda_{p}||\nabla\cdot(\mathbf{u}_{p}  - \mathbf{u}_{p,h}) ||_{\Omega_{p}}, \\
&\mathbf{e}_{u_{p}, H^{1}} = ||\varepsilon(\mathbf{u}_{p} - \mathbf{u}_{p, h})||_{\Omega_{p}},~ 
e_{\beta_{p}, L^{2}} = ||\beta_{p} - \beta_{p, h}||_{\Omega_{p}},~
e_{p_{p}, H^{1}} = ||\nabla(p_{p} - p_{p, h})||_{\Omega_{p}}.
\end{align*}
To maintain the balance between spatial and temporal discretization errors, the time step and mesh size are refined simultaneously to satisfy $\Delta =O(h^{2})$.
The corresponding convergence rate is reported in Figure \ref{fig:ex2}.
To assess the locking-free performance, we complete simulations with $\lambda_{p} = 1.0$ and a large value of $\lambda_{p} = 1.0\times 10^{10}$.
\begin{figure}[ht]
    \centering
    \begin{subfigure}[b]{0.45\textwidth}
        \centering
        \includegraphics[width=\textwidth]{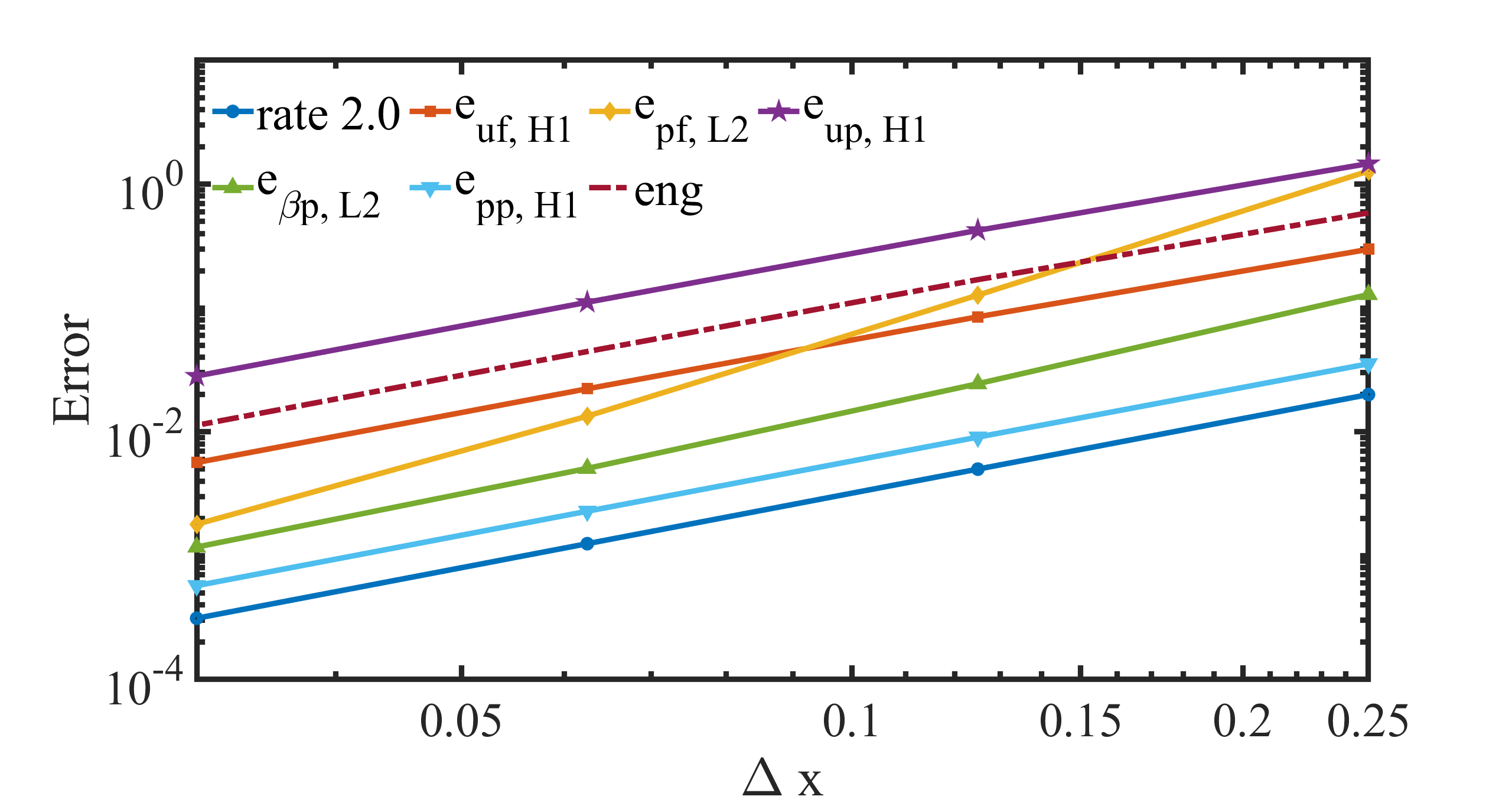}
        \caption{Convergence rates for $\lambda_{p} = 1.0$}
        \label{ex2_lam1}
    \end{subfigure}
    \hfill
    \begin{subfigure}[b]{0.45\textwidth}
        \centering
        \includegraphics[width=\textwidth]{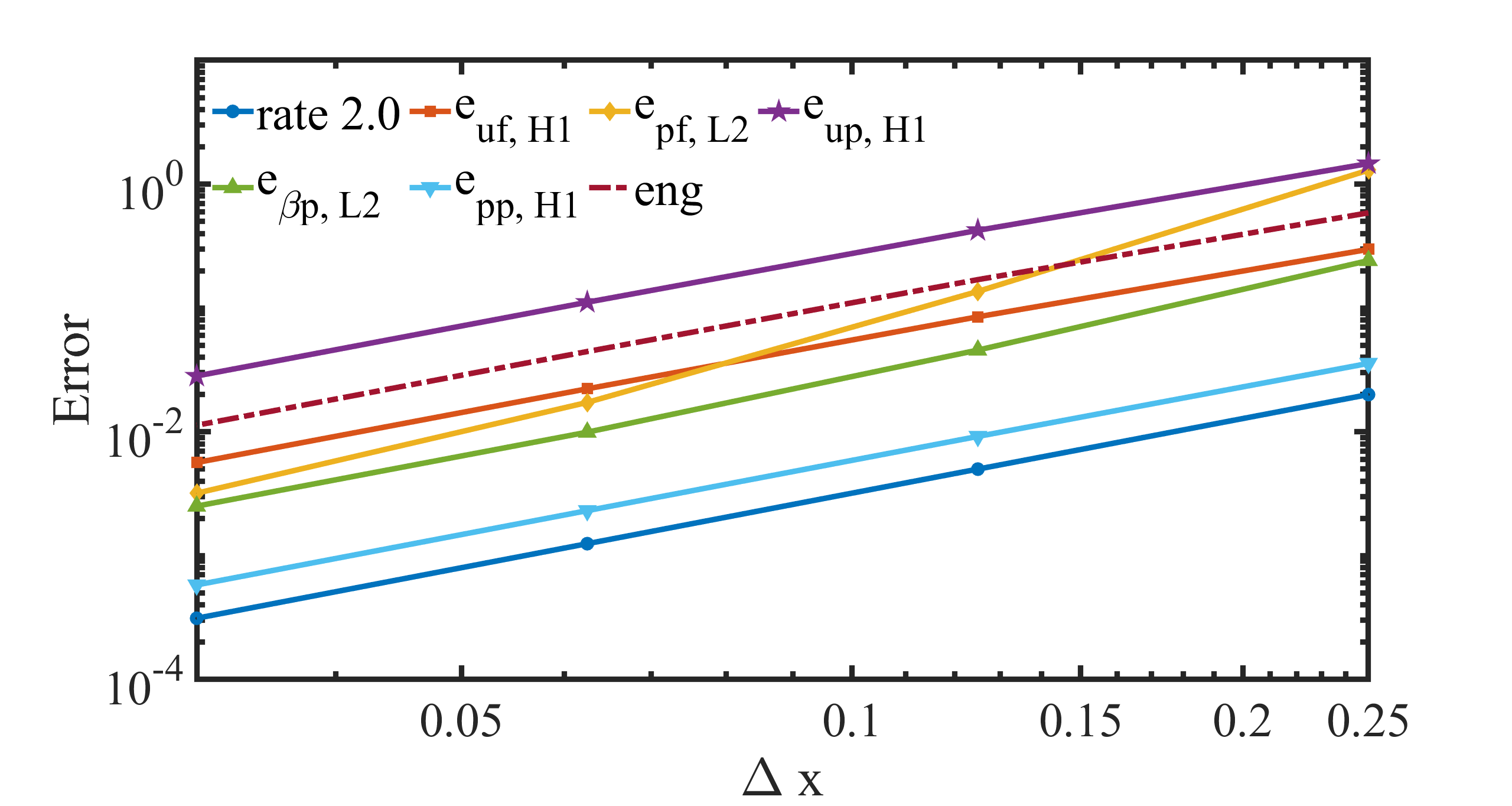}
        \caption{Convergence rates for $\lambda_{p} = 10^{10}$}
        \label{ex2_lam1e10}
    \end{subfigure}
    \caption{Convergence performance of subsection 6.1 with varying Lam\'{e} parameters. The results demonstrate that the optimal convergence rates are preserved regardless of the magnitude of $\lambda_p$, confirming the locking-free property of the proposed algorithm.}
    \label{fig:ex2}
\end{figure}
\subsection{Cantilever Bracket problem coupled with Stokes flow}
In this section, we investigate a cantilever bracket problem coupled with a Stokes flow to demonstrate that the proposed scheme is immune to spurious pressure oscillations in the poroelastic domain. The geometric configuration consists of a fluid domain $\Omega_{f} = [0, 1] \times [0, 1]$ and a poroelastic domain $\Omega_{p} = [0, 1] \times [-1, 0]$, separated by a horizontal interface  $\Gamma = \{(x, y) \mid 0 \le x \le 1, y = 0\}$.

For the fluid boundary conditions, no-slip condition is prescribed for the fluid velocity $\mathbf{v}_{f}$ on $\Gamma_{f}^{D} = \partial\Omega_{f}\setminus(\Gamma\cup\Gamma_{f}^{N})$, while a traction-free condition is imposed on $\Gamma_{f}^{N} = \{(x, y)|0\leq x \leq 1, y = 1\}$. Regarding the poroelastic region, The bottom boundary $\Gamma_{p}^{D}$ is clamped ($\mathbf{u}_{p} = 0$), while a constant horizontal traction $\mathbf{\sigma}_{p}\cdot\mathbf{n}_{p} = (-1, 0)^{T}$ is applied to the left boundary $\Gamma_{p}^{N*} = \{(x, y)|x = 0, -1\leq y \leq 0\}$. Traction-free conditions are prescribed on the remaining external boundaries.

To capture the transient behavior accurately, we set the time step $\Delta t = 10^{-5}$ and mesh size $h = 0.05$. The model parameters are chosen following \cite{phillips2009overcoming} as: 
$$\rho_{f} = 1.0, ~\mu_{f} = 10^{-2}, ~\rho_{p} = 10^{-10}, ~E = 10^{5}, ~\nu = 0.4, ~c_{0} = 0.0, ~\alpha = 0.93, ~K = 10^{-7}\mathbf{I}.$$
These settings are specifically selected to investigate the solver's performance for pressure oscillation.
\begin{figure}[htb]
    \centering 
    \includegraphics[width=0.8\textwidth]{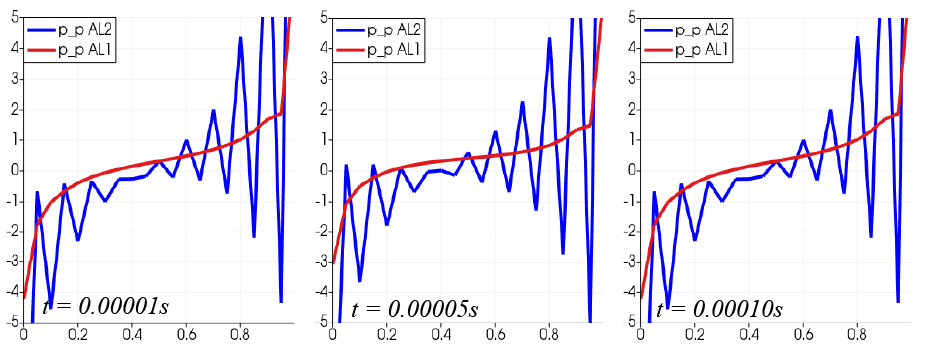} 
    \caption{The comparison of numerical results obtained by different algorithms during the early simulation stages. The profiles show the poroelastic pressure $p_p$ along the bottom boundary $\Gamma_{p}^{D}$, where the proposed scheme (AL 1) successfully suppresses the spurious oscillations observed in the standard method (AL 2).}
    \label{fig:pp_2d}
\end{figure}
It is well-documented \cite{phillips2009overcoming} that when the storage coefficient $c_{0}$ is small and the time step $\Delta t$ is significantly restricted, standard mixed finite element formulation often suffers from volumetric locking, manifesting as severe numerical oscillation in the pressure field $p_{p}$ near the bottom boundary $\Gamma_{p}^{D}$. As shown in the comparisons within Figure \ref{fig:pp_2d}, while algorithm AL 2 from \cite{guo2025fully} exhibits significant instabilities at early time steps ($t = 10^{-5}$ to $10^{-4}s$), our proposed scheme (AL 1) consistently yields a smooth and stable pressure distribution. This robustness against oscillations ensures the reliability of numerical solution for the entire coupled system.
\subsection{Blood flow example}
This example addresses a benchmark problem of blood flow in an idealized artery\cite{dalal2025robin,parrow2026stability}. The system consists of the Stokes equations governing the blood flow in the lumen and the Biot equations modeling the arterial wall, coupled through the interfaces at $y = \pm R$. The geometry is defined by a lumen of radius $R$ and length $L$, with the fluid domain $\Omega_f = (0, L) \times (-R, R)$ surrounded by a poroelastic wall of thickness $r_p$. The inlet and outlet boundaries of the fluid domain are defined as $\Gamma_f^{\text{in}} = \{(0,y) \mid -R < y < R\}$ and $\Gamma_f^{\text{out}} = \{(L,y) \mid -R < y < R\}$, respectively. Similarly, the corresponding boundaries for the poroelastic structure are given by $\Gamma_p^{\text{in}} = \{(0,y) \mid -R-r_p < y < -R \text{ or } R < y < R+r_p\}$ and $\Gamma_p^{\text{out}} = \{(L,y) \mid -R-r_p < y < -R \text{ or } R < y < R+r_p\}$. And the external structure boundary is defined as $\Gamma_p^{\text{ext}} = \{(x, y) \mid 0 < x < L,\ y = -R - r_p\ \text{or}\ y = R + r_p\}$.

To faithfully represent the 3D cylindrical tube from which this 2D problem is derived, we augment equation \eqref{2.6} with an additional term:
\begin{align}
	\rho_{p}\partial_{tt}\mathbf{u}_{p}-\nabla\cdot\bm{\sigma}_{p}(\mathbf{u}_{p},p_{p})+\beta\mathbf{u}_{p}&=\mathbf{f}_{p},&&\quad \mbox{in } \Omega_{p}\times (0,T],\label{7.1}
\end{align}
where $\beta$ is a spring coefficient. The additional term or the last term on the left-hand side of equation \eqref{7.1}-originates from the axially symmetric 2D formulation, which captures the recoil due to circumferential strain \cite{badia2008fluid}. The body force terms $\mathbf{f}_f$ and $\mathbf{f}_p$, as well as the external sources $\phi_{f}$ and $\phi_p$, are set to zero. Moreover, we impose the following boundary conditions: 
	\begin{align*}
		&\mathbf{u}_p = \mathbf{0}, &&\quad \text{on } \Gamma_p^{\mathrm{in}} \cup \Gamma_p^{\mathrm{out}} \times (0, T],\\
		&\mathbf{v}_p\cdot\mathbf{n}_{p} = 0, &&\quad \text{on } \Gamma_p^{\mathrm{in}} \cup \Gamma_p^{\mathrm{out}}\cup\Gamma_p^{\mathrm{ext}} \times (0, T],\\
		&(\pmb{\sigma}_p\mathbf{n}_{p})\cdot\mathbf{n}_{p} = 0, &&\quad \text{on } \Gamma_p^{\mathrm{ext}} \times (0, T],\\
		&\sigma_f \mathbf{n}_f = -p_{\text{in}}(t) \mathbf{n}_f, &&\quad\text{ on } \Gamma^{\text{in}}_f \times (0,T],\\
		&(\sigma_f \mathbf{n}_f)\cdot\mathbf{n}_{f} = 0, &&\quad\text{ on } \Gamma^{\text{out}}_f \times (0,T],
	\end{align*} 
	where $\mathbf{n}^{\text{in}}_f$ and $\mathbf{n}^{\text{out}}_f$ are the outward unit normals separately and
	\begin{align*}
		p_{\text{in}}(t) = 
		\left\lbrace\begin{aligned}
			&0, && \text{if } t > T_{\text{max}}, \\
			&\frac{P_{\text{max}}}{2} \Big[ 1 - \cos\Big( \frac{2\pi t}{T_{\text{max}}} \Big) \Big], && \text{if } t \leq T_{\text{max}},
		\end{aligned}\right.
	\end{align*}
	with $P_{\text{max}} = 13,334~\rm{dyn/cm^2}$ and $T_{\text{max}} = 0.003~\text{s}$.The amplitude of this wave is comparable to the pressure difference between the systolic and diastolic phases of the heartbeat 
 
 Under the assumption of axial symmetry, we have the domain along the horizontal symmetry axis, denoted with $\Gamma^{\text{sym}}_f$, and  impose the following symmetry conditions therein:
 \begin{align*}
 \mathbf{v}_{f}\cdot\bm{n}_{f} = 0,\quad\text{ on } \Gamma^{\text{sym}}_f \times (0,T].
 \end{align*}
 Although the flow distribution and pressure field are often unknown, they are frequently employed in blood flow models. In a system at rest, this inlet boundary condition generates a pressure pulse that propagates through both the fluid and the poroelastic structure. To prevent the pressure pulse from reaching the outlet, the end of the time interval of interest is set to $T = 0.014~\text{s}$. 
	
	The physical parameters for this test are listed in Tables \ref{tab7} and \ref{tab8}. They are set to values that lie within the physiological range for arterial blood flow, thus guaranteeing the relevance of our model. The time step is set to $\Delta t=5\times 10^{-5} \text{s}$, and the mesh size is set to $h=5\times 10^{-2}$. To verify the accuracy of Algorithm~\ref{al1}, the numerical results are compared with those obtained using a strongly coupled method previously published in \cite{cesmelioglu2017analysis}. The solutions obtained with this method are used as reference data. The elastic displacement, fluid pressure, axial fluid velocity at the interface, and the axial fluid velocity at the bottom boundary of the fluid domain are shown in Figure~\ref{fig:2000}. For Algorithm~\ref{al1}, the curves of all variables obtained with $L_1 = L_2 = 10^3$ (AL 1.1) and $L_1 = L_2 = 10^2$ (AL 1.2) are nearly identical and show excellent agreement with the reference data, whereas some discrepancies are observed when $L_1 = L_2 = 10^4$ (AL 1.3) is used. Specifically, the wave propagation in AL 1.3 is slower over time compared to AL 1.1 and AL 1.2. This is because numerical dissipation is present in the algorithm when $L_1 = L_2 = 10^{4}$, whereas the cases with $L_1 = L_2 = 10^{3}$ (or $L_1 = L_2 = 10^{2}$) can be regarded as stabilized. Furthermore, comparison with the numerical results in Example~ 2 of \cite{parrow2026stability} also demonstrates the accuracy of Algorithm~\ref{al1}.
    
	\begin{table}[htbp]
		\centering
		\caption{Physical parameters 1.}\label{tab7}
		\begin{tabular}{l l l l}
			\hline
			Parameter & Symbol & Units & Reference Value \\
			\hline
             Radius &R &cm &0.5\\
             Length &L &cm &6\\
			Poroelastic wall density & $\rho_p$ & g/cm$^3$ & 1.1 \\
			Fluid density & $\rho_f$ & g/cm$^3$ & 1.0 \\
			Dynamic viscosity & $\mu_f$ & g/(cm$\cdot$s) & 0.035 \\
			Spring coefficient & $\beta$ & dyn/cm$^4$ & $4 \times 10^6$ \\
			Storage coefficient & $c_0$ & cm$^2$/dyn & $10^{-3}$ \\
			Permeability & $K$ & cm$^2$ & $ 10^{-6}\mathbf{I}$ \\
			Lam\'e coefficient & $\mu_p$ & dyn/cm$^2$ & $5.575 \times 10^5$ \\
			Lam\'e coefficient & $\lambda_p$ & dyn/cm$^2$ & $1.7 \times 10^6$ \\
			BJS coefficient & $\gamma$ & g/cm$^2\cdot$ s & $10^{3}$ \\
			Biot-Willis constant & $\alpha$ & -- & 1 \\
            Combination constant & $L_{3}$ & -- & $10^{-6}$ \\
			\hline
		\end{tabular}
	\end{table}
\begin{table}[htbp]
		\centering
		\caption{Physical parameters 2.}\label{tab8}
		\begin{tabular}{l l l ll}
			\hline
			Parameter & Symbol & Case 1& Case 2& Case 3 \\
			\hline
			Combination constant & $L_{1}$ & $10^{3}$ &$10^{2}$  &$10^{4}$\\
			Combination constant & $L_{2}$ & $10^{3}$ &$10^{2}$ &$10^{4}$\\
			\hline
		\end{tabular}
	\end{table}
 
\begin{figure}[htbp]
    \centering 
    \begin{subfigure}[b]{0.32\textwidth}
        \centering
        \includegraphics[width=\linewidth]{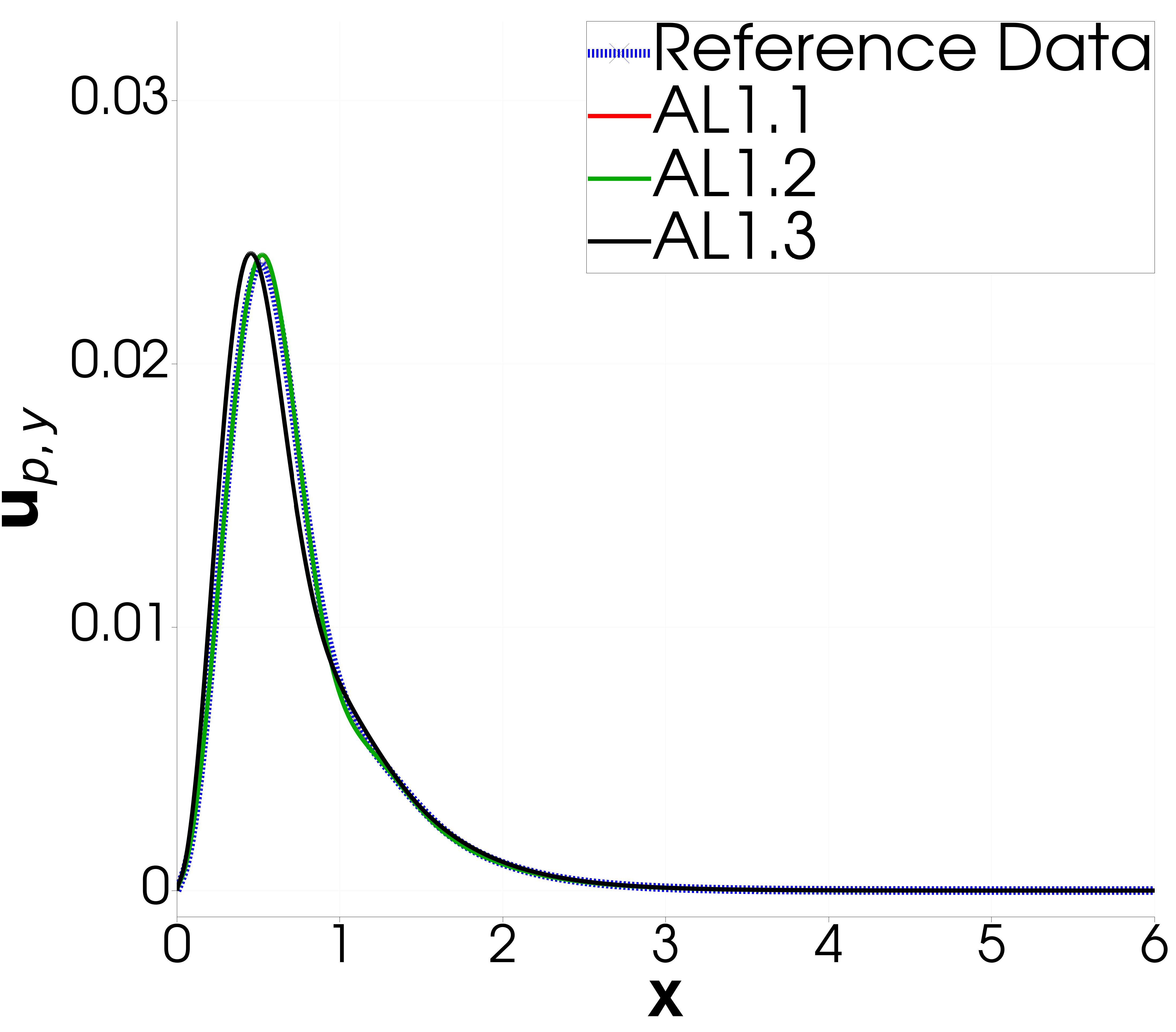} 
        \caption{t = 0.0035s}
    \end{subfigure}
    \hfill 
    \begin{subfigure}[b]{0.32\textwidth}
        \centering
        \includegraphics[width=\linewidth]{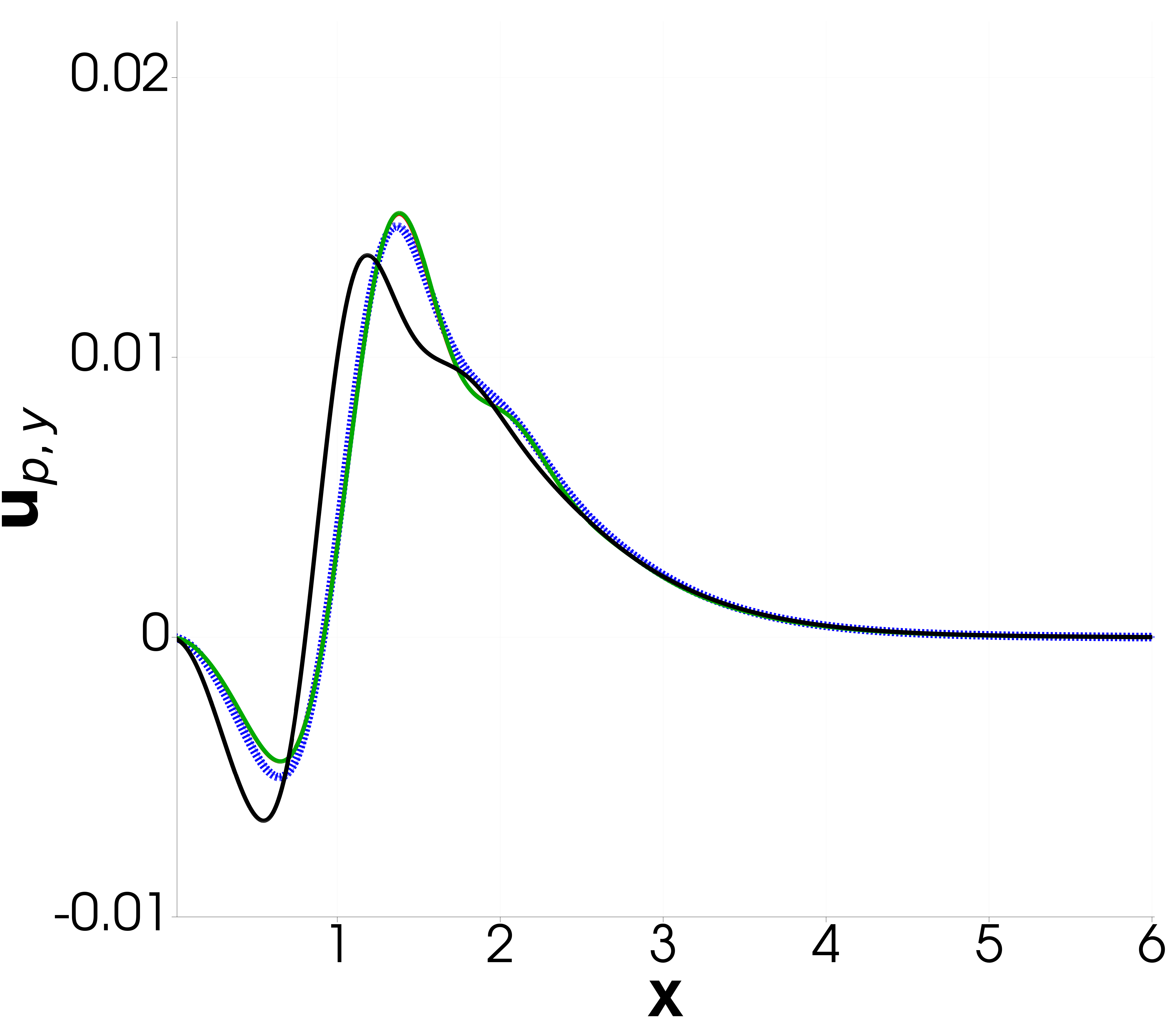}
        \caption{t = 0.0070s}
    \end{subfigure}
    \hfill
    \begin{subfigure}[b]{0.32\textwidth}
        \centering
        \includegraphics[width=\linewidth]{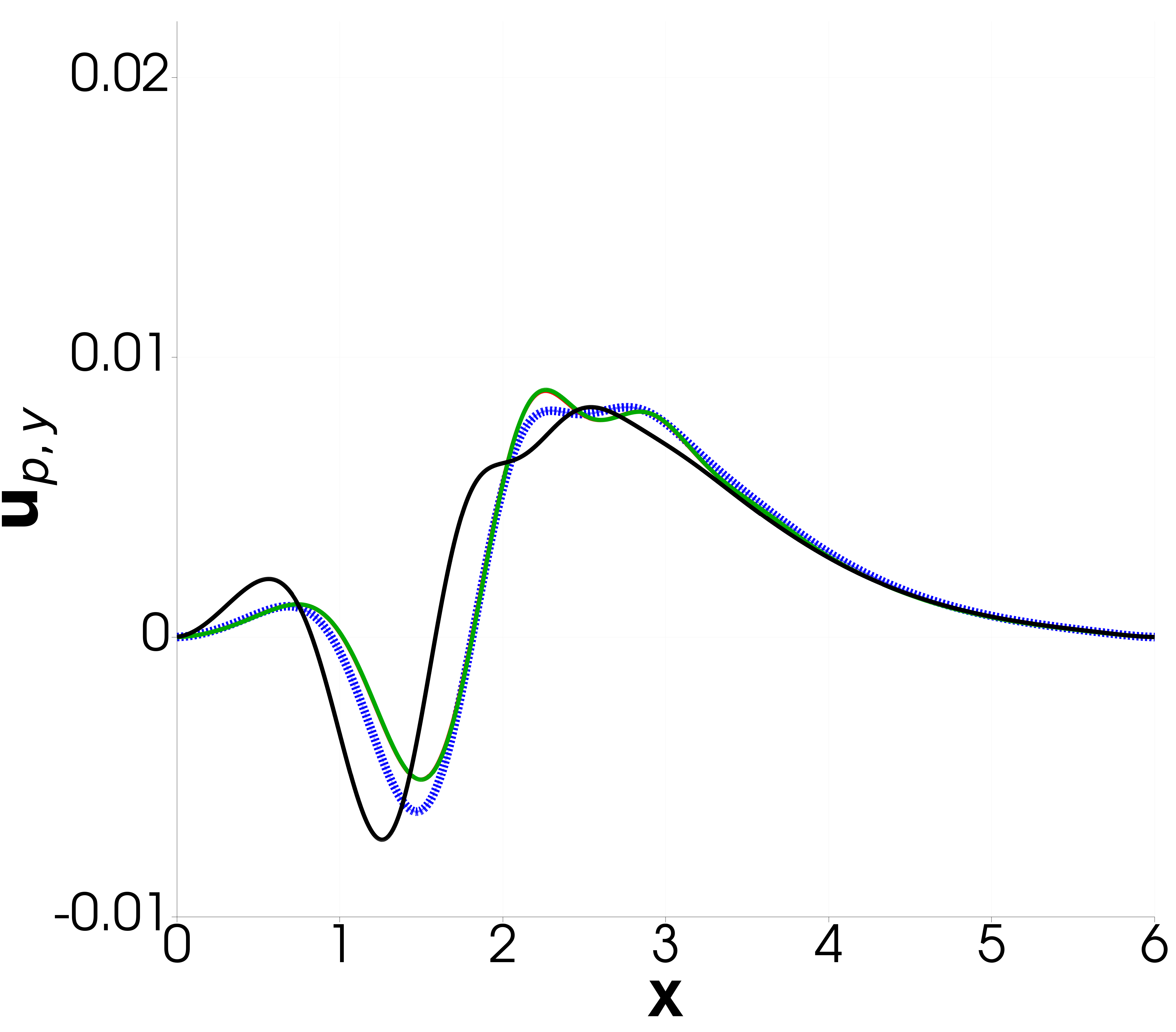}
        \caption{t = 0.0105s}
    \end{subfigure}
    \begin{subfigure}[b]{0.32\textwidth}
        \centering
        \includegraphics[width=\linewidth]{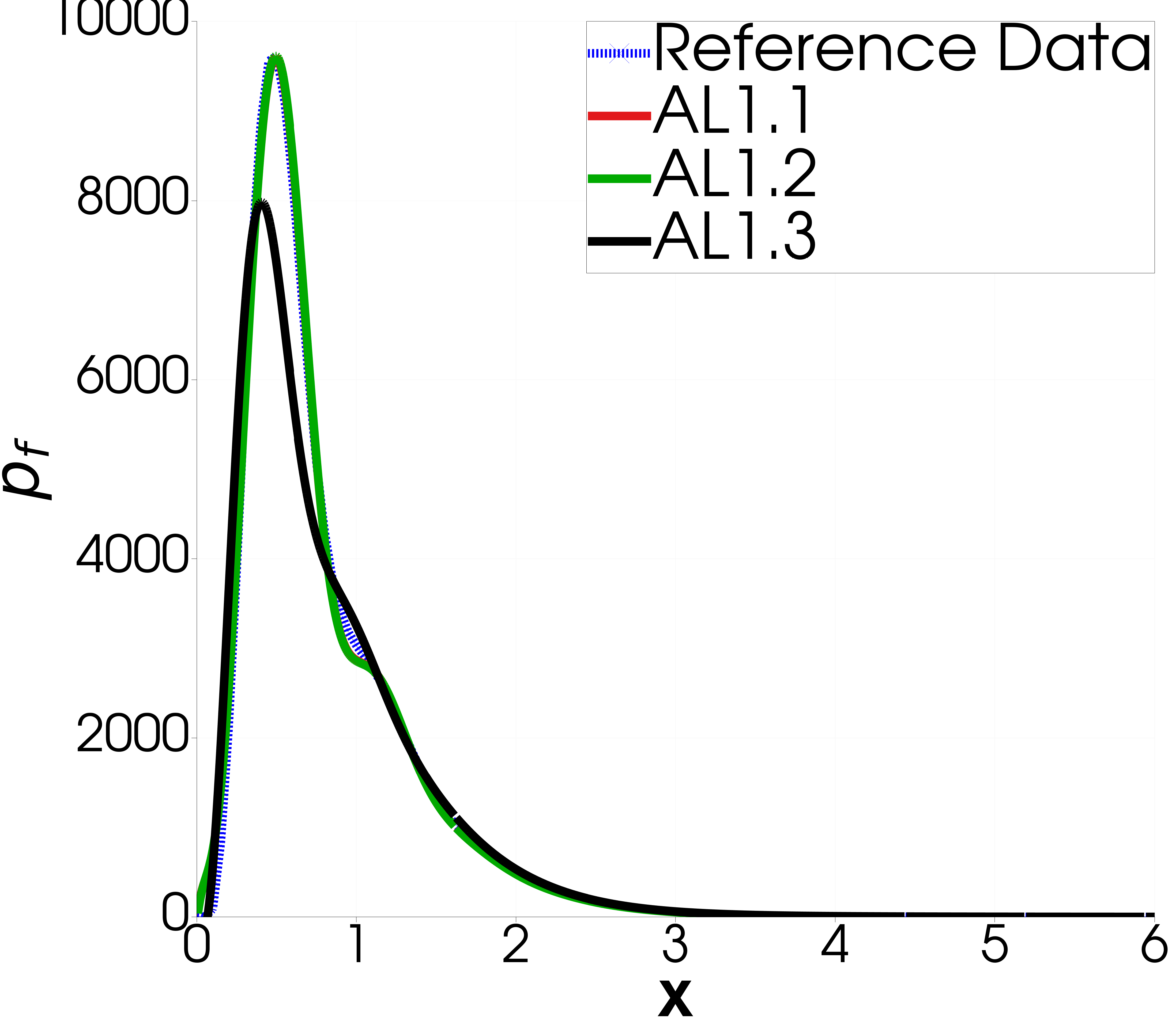} 
        \caption{t = 0.0035s}
    \end{subfigure}
    \hfill 
    \begin{subfigure}[b]{0.32\textwidth}
        \centering
        \includegraphics[width=\linewidth]{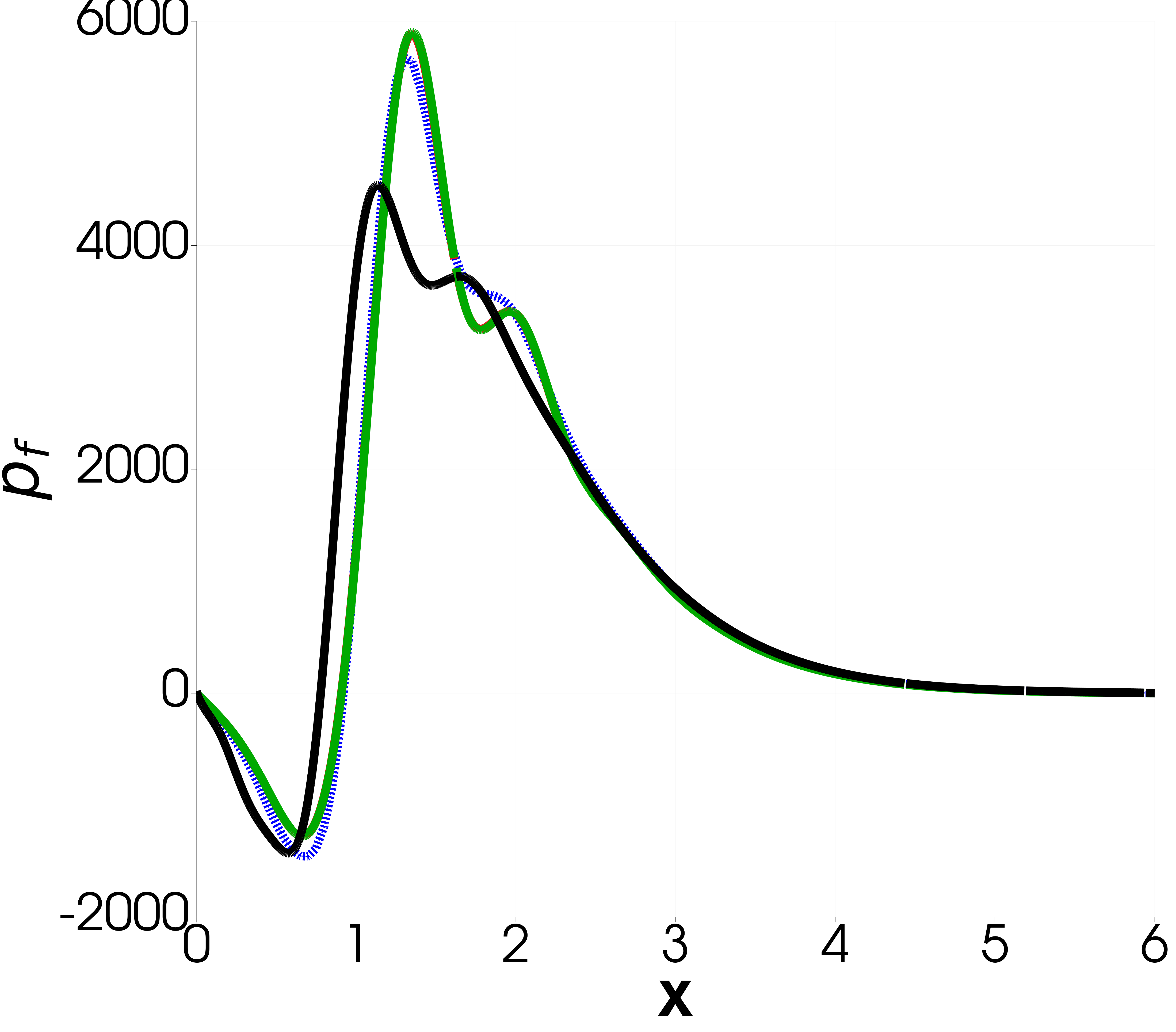}
        \caption{t = 0.0070s}
    \end{subfigure}
    \hfill
    \begin{subfigure}[b]{0.32\textwidth}
        \centering
        \includegraphics[width=\linewidth]{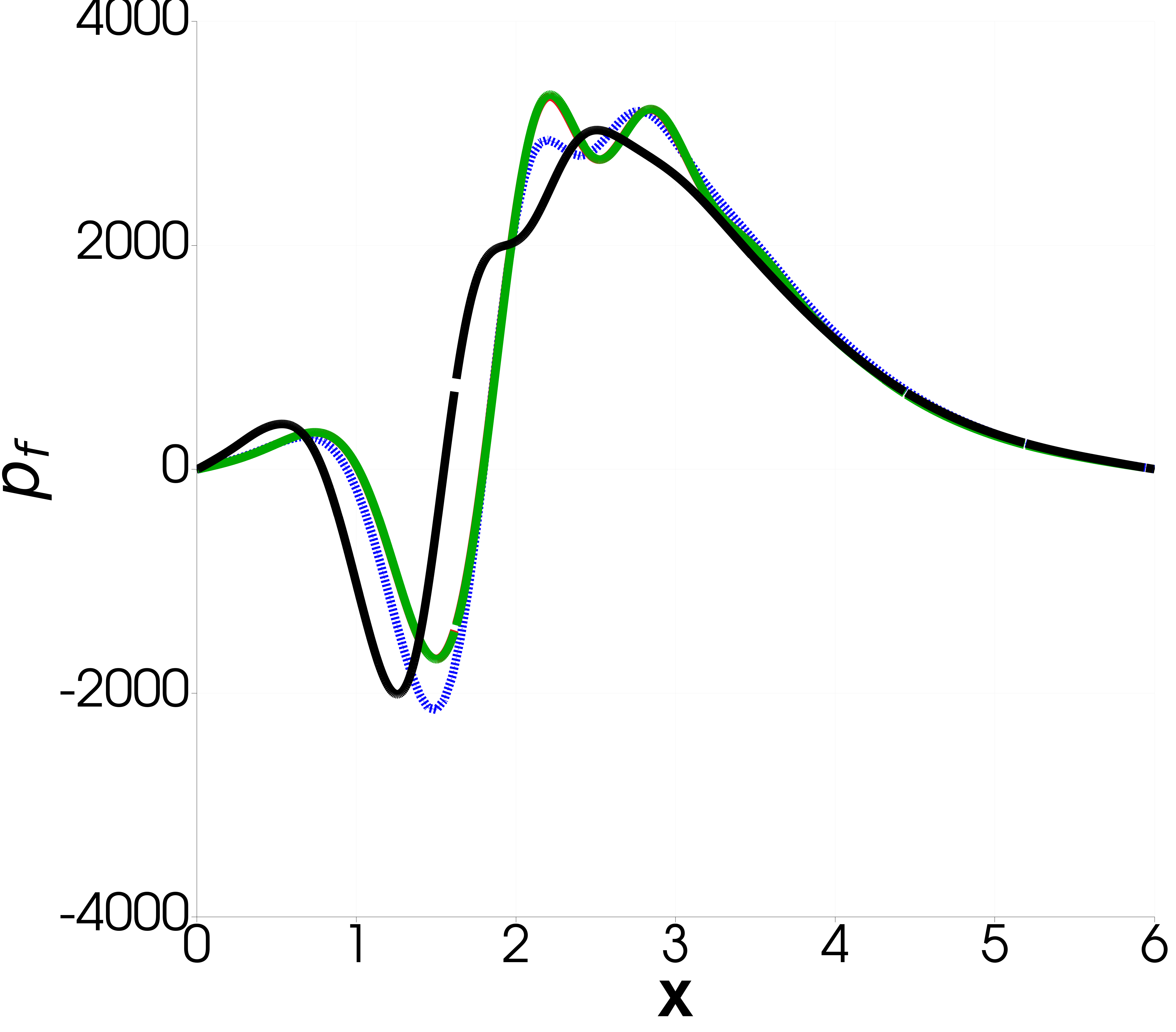}
        \caption{t = 0.0105s}
        \label{fig:eta_105}
    \end{subfigure}
        \begin{subfigure}[b]{0.32\textwidth}
        \centering
        \includegraphics[width=\linewidth]{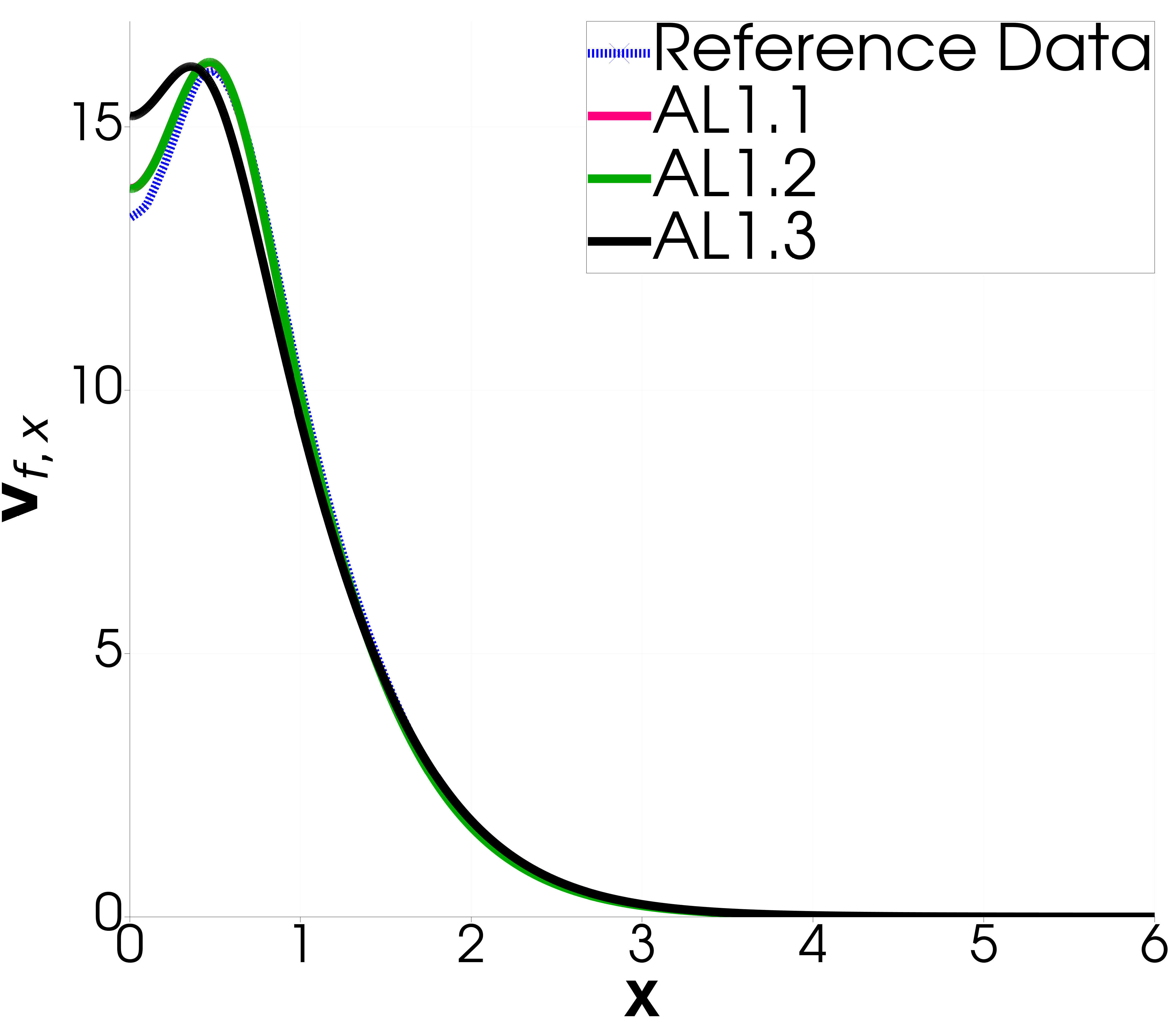} 
        \caption{t = 0.0035s}
    \end{subfigure}
    \hfill 
    \begin{subfigure}[b]{0.32\textwidth}
        \centering
        \includegraphics[width=\linewidth]{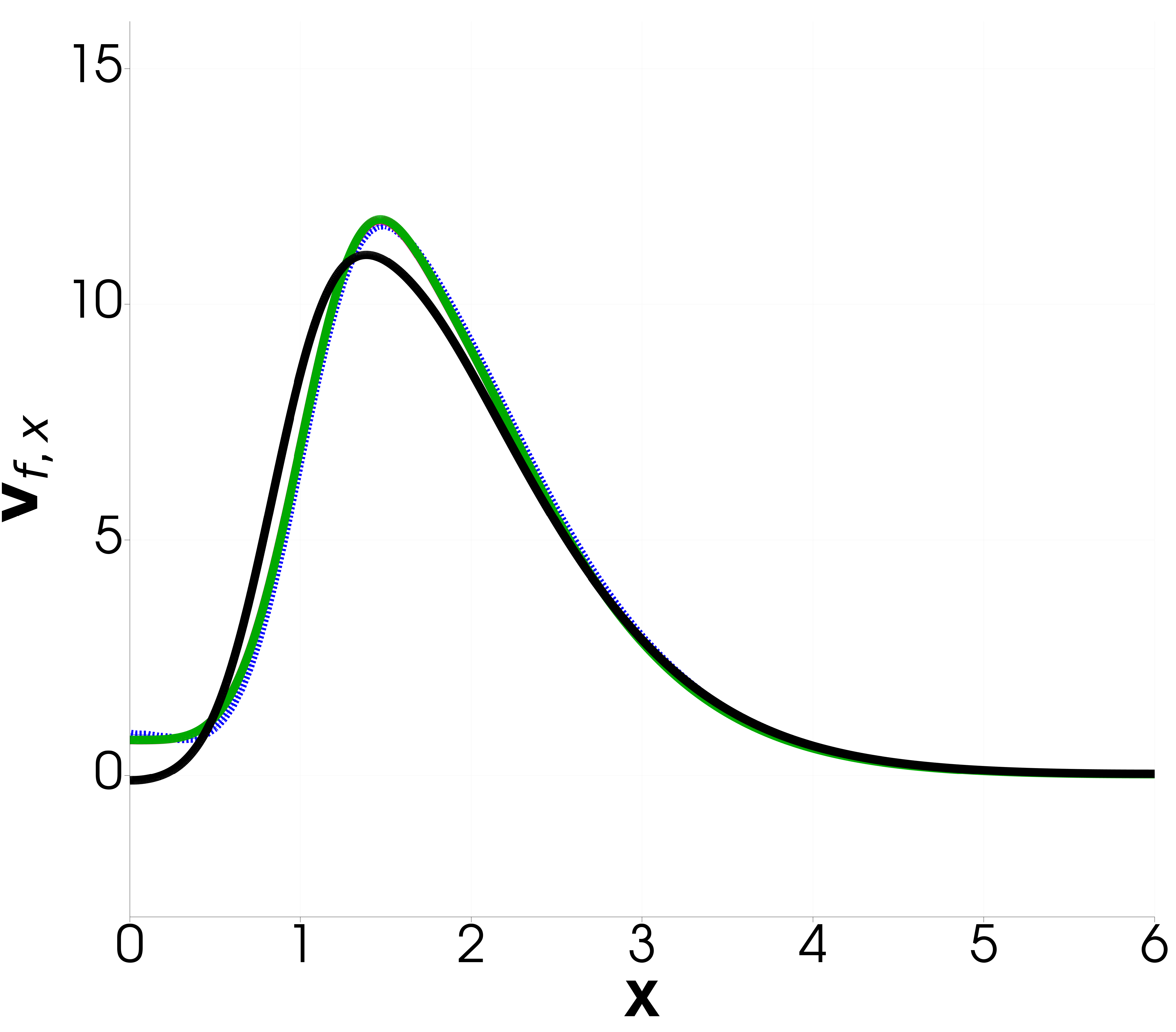}
        \caption{t = 0.0070s}
    \end{subfigure}
    \hfill
    \begin{subfigure}[b]{0.32\textwidth}
        \centering
        \includegraphics[width=\linewidth]{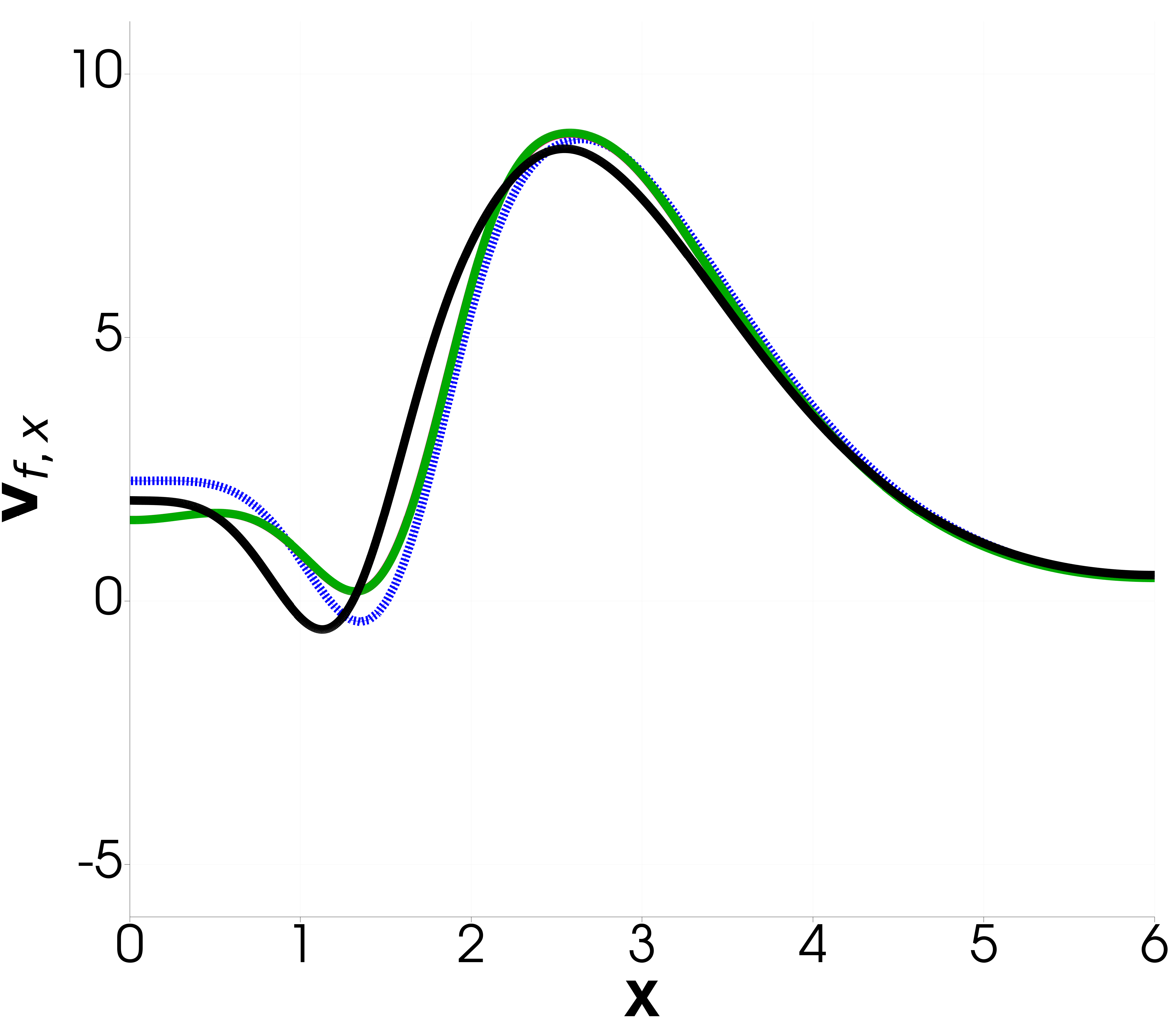}
        \caption{t = 0.0105s}
    \end{subfigure}
        \begin{subfigure}[b]{0.32\textwidth}
        \centering
        \includegraphics[width=\linewidth]{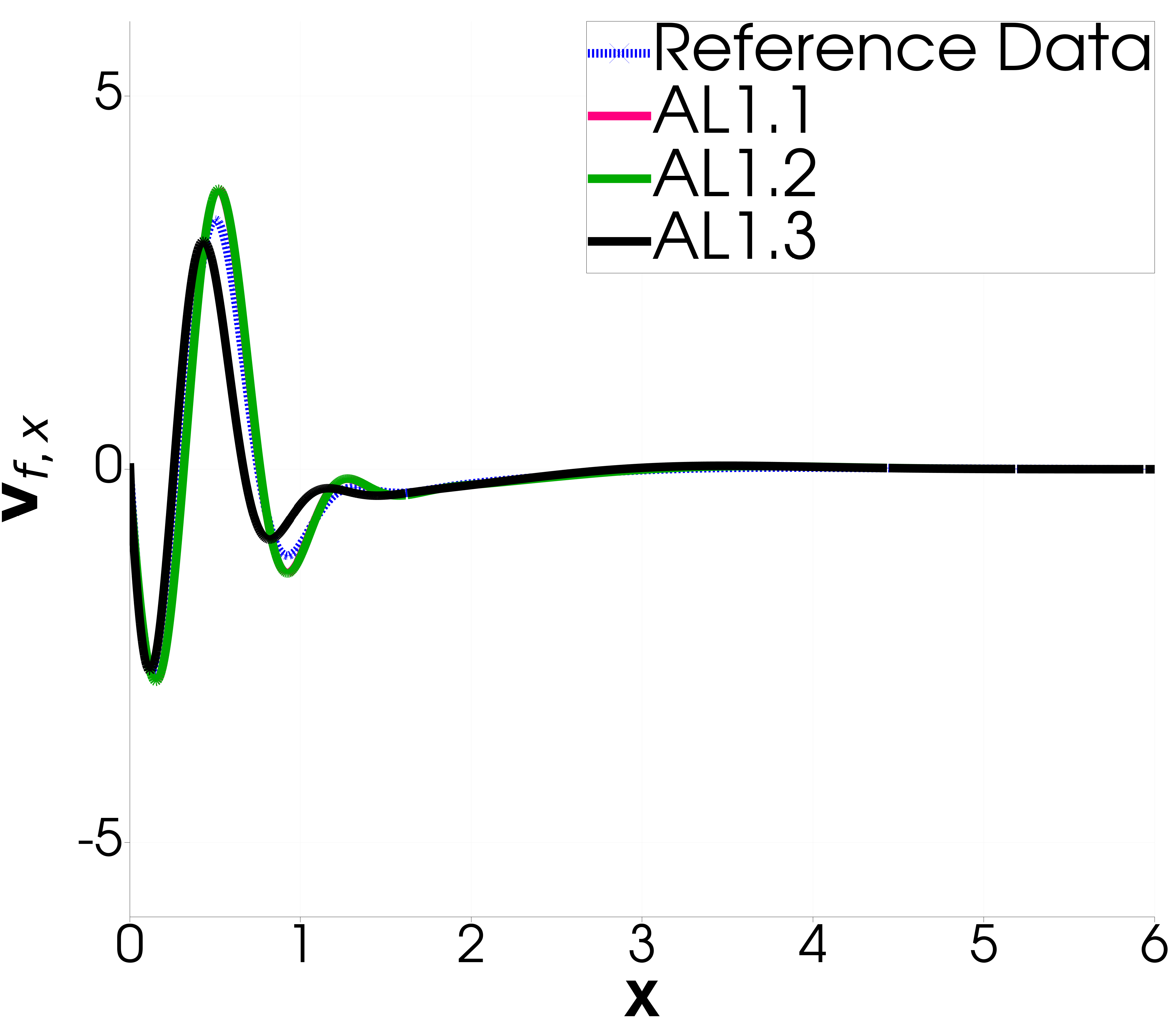} 
        \caption{t = 0.0035s}
    \end{subfigure}
    \hfill 
    \begin{subfigure}[b]{0.32\textwidth}
        \centering
        \includegraphics[width=\linewidth]{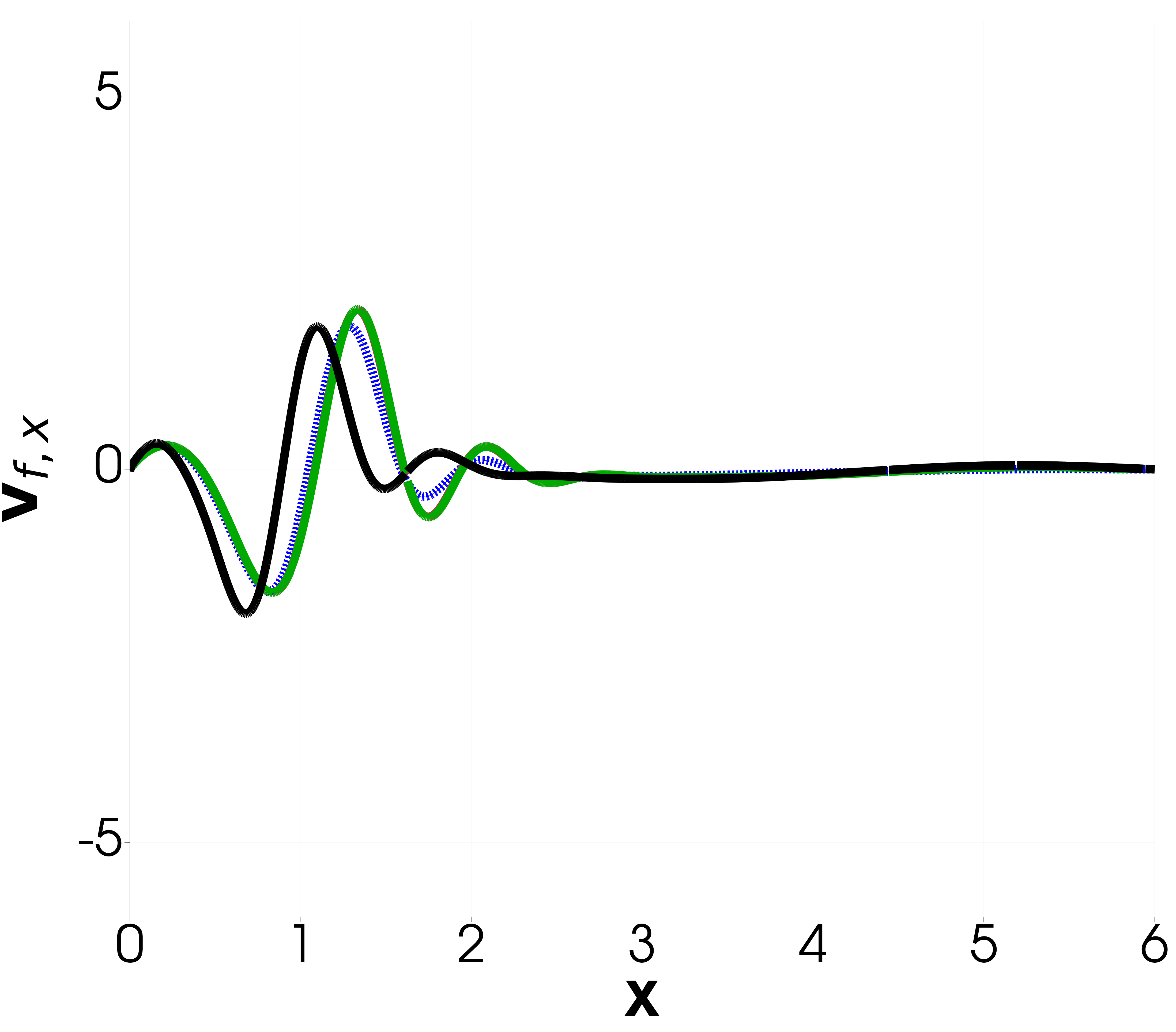}
        \caption{t = 0.0070s}
    \end{subfigure}
    \hfill
    \begin{subfigure}[b]{0.32\textwidth}
        \centering
        \includegraphics[width=\linewidth]{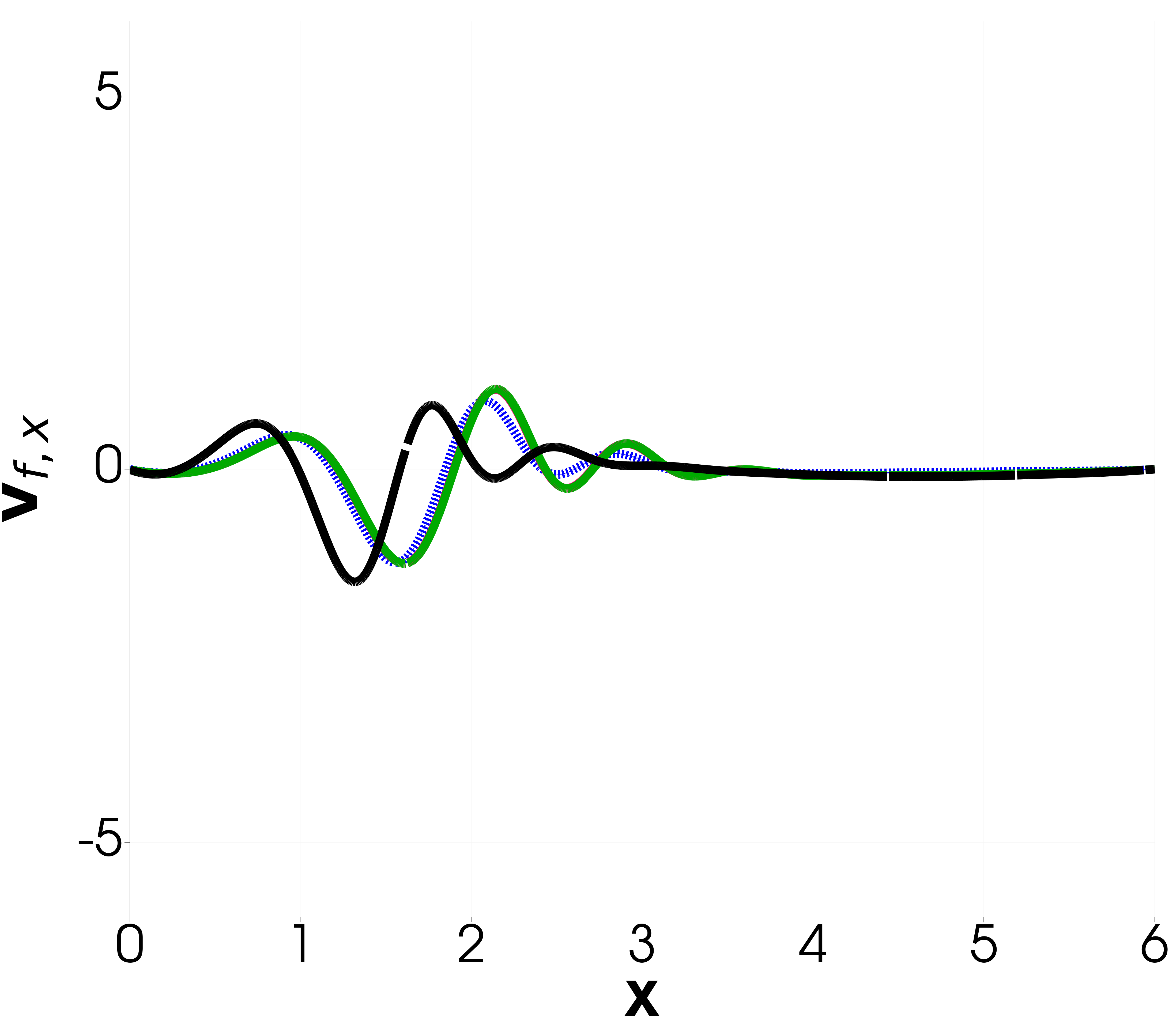}
        \caption{t = 0.0105s}
    \end{subfigure}
    \vspace{1ex}
    \caption{Numerical results for $\mathbf{u}_{p,y},~\mathbf{v}_{f,x}$ and $p_{f}$ on the fluid-poroelastic interface and $\mathbf{v}_{f,x}$ on the bottom fluid boundary obtained from Algorithm \ref{al1}. AL 1.1, AL 1.2, and AL 1.3 refer to the parameter configurations in Table \ref{tab8}, and the reference data obtained using a strongly-coupled method presented in \cite{cesmelioglu2017analysis}.}
    \label{fig:2000} 
\end{figure}
 2D snapshots of the pressure and fluid velocity are shown in Figure \ref{fig:2001}. Specifically, we compare the results obtained using Algorithm~\ref{al1} with $L_{1}=L_{2}=10^{3}$ and the reference data at four different time points. The results from the algorithm are displayed in the top row, and the reference data in the bottom row. Excellent agreement is observed.
\begin{figure}[htbp]
    \centering 
    \includegraphics[width=0.7\textwidth]{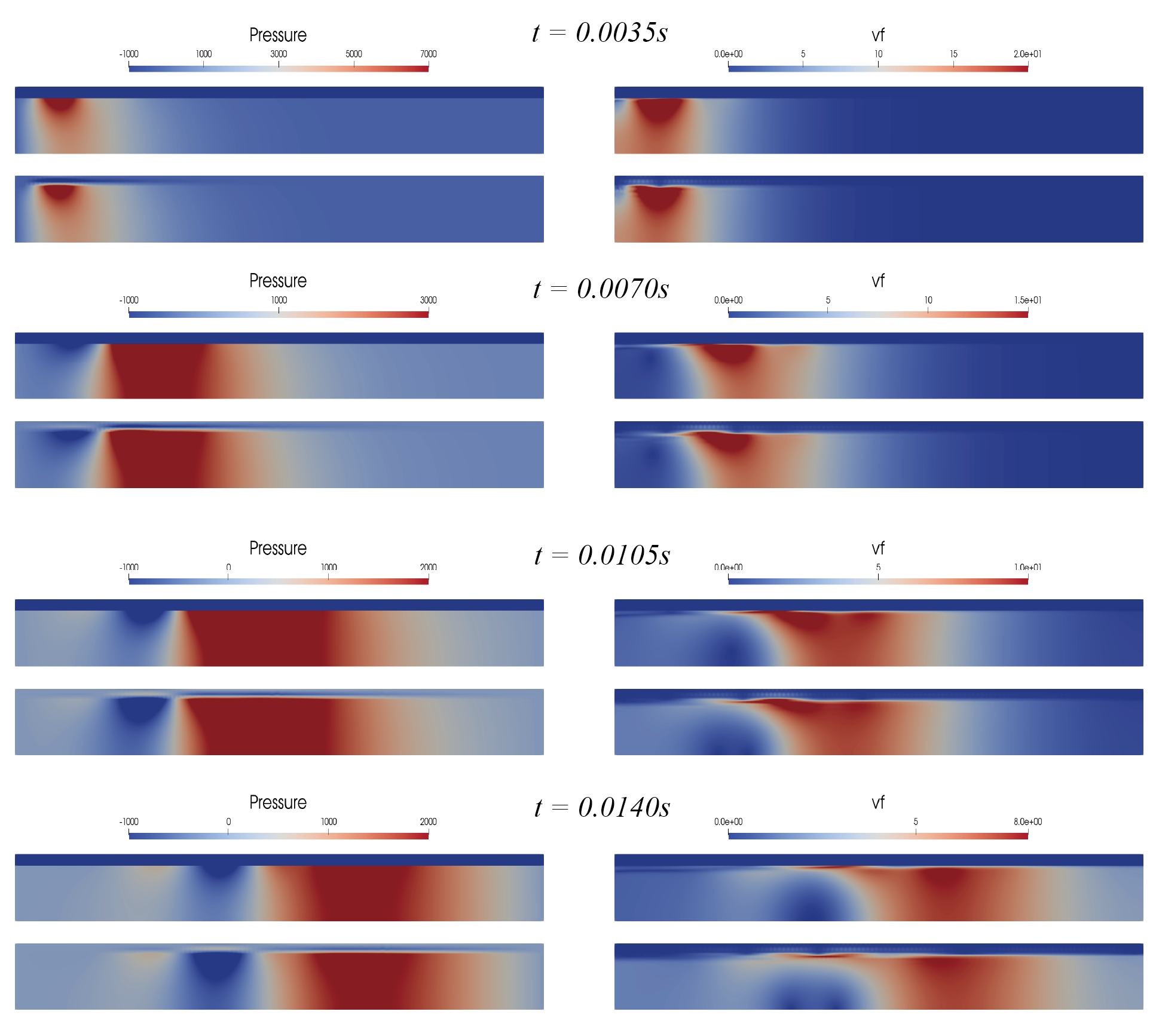} 
    \caption{Surface plots of pressure and fluid velocity for Algorithm~\ref{al1} with Case 1 (top) and the reference solution (bottom) at times $t = 0.0035$, $0.007$, $0.0105$, and $0.014$ s. The domain has been reflected in Paraview to recover the full channel.}
    \label{fig:2001}
\end{figure}

\section{Conclusion}

In this work, we developed a fully decoupled Robin--Robin scheme for the coupled Stokes--Biot fluid--poroelasticity interaction problem. The central ingredient is a four-variable reformulation of the fully dynamic Biot system, obtained through the introduction of two auxiliary variables. This reformulation is designed to improve robustness with respect to locking-related extreme parameters and, crucially, preserves the original FPSI interface conditions.

Based on this reformulation, we derived Robin--Robin transmission conditions that lead to a fully parallel time-stepping algorithm, in which the fluid and poroelastic subproblems can be solved independently without sub-iterations. Compared with existing partitioned Robin-type approaches, and in particular with our previous decoupling scheme for the standard Stokes--Biot system, the present work provides a substantial extension: it consistently incorporates a locking-aware reformulation into the fully coupled FPSI framework, preserves the original interface coupling structure, and yields a rigorous fully discrete analysis including equivalence of the reformulated and original systems, unconditional stability, and optimal-order error estimates.

The numerical experiments confirm the theoretical convergence results and demonstrate robust performance in parameter regimes associated with Poisson locking and early-time pressure oscillations. These results indicate that the proposed approach provides an effective and mathematically justified framework for parameter-robust partitioned simulation of FPSI problems. Future work will focus on extensions to more general fluid models, higher-order time discretizations, and interface treatments on nonmatching meshes.
\appendix
\section{ The proof of Theorem \ref{thm3.2}}
\label{app:proof}
Without loss of generality, we assume $\mathbf{f}_{f}=\mathbf{f}_{p}=\mathbf{g}=\mathbf{0}$ and $\phi_{f}=\phi_{p}=0$. Substituting $R_{1}$ and $R_{2} $ into \eqref{3.1} and setting $(\mathbf{w}_{f},q_{f})=(\mathbf{v}_{f},p_{f})$ in \eqref{3.1}-\eqref{3.2}, then integrating the resulting equations with regard to $t$ over $(0,s)$ for $s\in(0,T]$ and adding them, we obtain
\begin{align}
   &\frac{\rho_{f}}{2}\|\mathbf{v}_{f}(s)\|_{L^{2}(\Omega_{f})}^{2}+\int_{0}^{s} \Big[2\mu_{f}\|\varepsilon(\mathbf{v}_{f})\|_{L^{2}(\Omega_{f})}^{2}+c_{BJS}\|\mathcal{M}_{f}(\mathbf{v}_{f})\|_{L^{2}(\Gamma)}^{2}\Big]\,dt\label{A.1}\\
   &=\frac{\rho_{f}}{2}\|\mathbf{v}_{f}(0)\|_{L^{2}(\Omega_{f})}^{2}-\int_{0}^{s}\langle p_{p},\mathbf{v}_{f}\cdot\mathbf{n}_{f}\rangle_{\Gamma}\,dt\nonumber\\
   &\quad+\int_{0}^{s}c_{BJS}\langle \mathcal{M}_{p}(\partial_{t}\mathbf{u}_{p}),\mathcal{M}_{f}(\mathbf{v}_{f}\rangle_{\Gamma}\,dt.\nonumber
\end{align}
Similarly, substituting $R_{3}$--$R_{5}$ into \eqref{3.3}-\eqref{3.6} and differentiating \eqref{3.5} with respect to $t$ once, then setting $(\mathbf{z}_{p},\mathbf{w}_{p},\varphi_{p},\psi_{p})=(\rho_{p}\partial_{t}\mathbf{v}_{p},\partial_{t}\mathbf{u}_{p},\beta_{p},p_{p})$ in \eqref{3.3}-\eqref{3.6}, integrating the resulting equations with regard to $t$ over $(0,s)$ for $s\in(0,T]$ and adding them together, we arrive at
\begin{align}
&\frac{\rho_{p}}{2}\|\mathbf{v}_{p}(s)\|_{L^{2}(\Omega_{p})}^{2}+\mu_{p}\|\varepsilon(\mathbf{u}_{p}(s))\|_{L^{2}(\Omega_{p})}^{2}\label{A.2}\\
&\quad+\frac{1}{2\lambda_{p}}\|\alpha p_{p}(s)-\beta_{p}(s)\|_{L^{2}(\Omega_{p})}^{2}+\frac{c_{0}}{2}\|p_{p}(s)\|_{L^{2}(\Omega_{p})}^{2}\nonumber\\
&\quad+\int_{0}^{s}\Big[c_{BJS}\|\mathcal{M}_{p}(\partial_{t}\mathbf{u}_{p})\|_{L^{2}(\Gamma)}^{2}+\mu_{f}^{-1}\|K^{\frac{1}{2}}\nabla p_{p}\|_{L^{2}(\Omega_{p})}^{2}\Big]\,dt\nonumber\\
&=\frac{\rho_{p}}{2}\|\mathbf{v}_{p}(0)\|_{L^{2}(\Omega_{p})}^{2}+\mu_{p}\|\varepsilon(\mathbf{u}_{p}(0))\|_{L^{2}(\Omega_{p})}^{2}+\frac{1}{2\lambda_{p}}\|\alpha p_{p}(0)-\beta_{p}(0)\|_{L^{2}(\Omega_{p})}^{2}\nonumber\\
&\quad+\frac{c_{0}}{2}\|p_{p}(0)\|_{L^{2}(\Omega_{p})}^{2}-\int_{0}^{s}\langle p_{p},\partial_{t}\mathbf{u}_{p}\cdot\mathbf{n}_{p}\rangle_{\Gamma}\,dt\nonumber\\
&\quad+\int_{0}^{s}\Big[c_{BJS}\langle \mathcal{M}_{f}(\mathbf{v}_{f}),\mathcal{M}_{p}(\partial_{t}\mathbf{u}_{p})\rangle_{\Gamma}+\langle \mathbf{v}_{f}\cdot\mathbf{n}_{f}+\partial_{t}\mathbf{u}_{p}\cdot\mathbf{n}_{p},p_{p}\rangle_{\Gamma}\Big]\,dt.\nonumber
\end{align}
Adding \eqref{A.1} and \eqref{A.2}, then using the Cauchy-Schwarz and Young inequalities for the result, we further get
\begin{align}
&\frac{\rho_{f}}{2}\|\mathbf{v}_{f}(s)\|_{L^{2}(\Omega_{f})}^{2}+\frac{\rho_{p}}{2}\|\mathbf{v}_{p}(s)\|_{L^{2}(\Omega_{p})}^{2}+\mu_{p}\|\varepsilon(\mathbf{u}_{p}(s))\|_{L^{2}(\Omega_{p})}^{2}\label{A.3}\\
&\quad+\frac{c_{0}}{2}\|p_{p}(s)\|_{L^{2}(\Omega_{p})}^{2}+\frac{1}{2\lambda_{p}}\|\alpha p_{p}(s)-\beta_{p}(s)\|_{L^{2}(\Omega_{p})}^{2}\nonumber\\
&\quad+\int_{0}^{s} \Big[2\mu_{f}\|\varepsilon(\mathbf{v}_{f})\|_{L^{2}(\Omega_{f})}^{2}+\mu_{f}^{-1}\|K^{\frac{1}{2}}\nabla p_{p}\|_{L^{2}(\Omega_{p})}^{2}\Big]\,dt\nonumber\\
&\leq\frac{\rho_{f}}{2}\|\mathbf{v}_{f}(0)\|_{L^{2}(\Omega_{f})}^{2}+\frac{\rho_{p}}{2}\|\mathbf{v}_{p}(0)\|_{L^{2}(\Omega_{p})}^{2}+\mu_{p}\|\varepsilon(\mathbf{u}_{p}(0))\|_{L^{2}(\Omega_{p})}^{2}\nonumber\\
&\quad+\frac{c_{0}}{2}\|p_{p}(0)\|_{L^{2}(\Omega_{p})}^{2}+\frac{1}{2\lambda_{p}}\|\alpha p_{p}(s)-\beta_{p}(0)\|_{L^{2}(\Omega_{p})}^{2}.\nonumber
\end{align}
Consequently, inequality \eqref{A.3} yields uniform bounds for the Galerkin approximations, from which the existence of a solution follows by the standard Galerkin procedure and a compactness argument \cite{evans2022partial}.

Next, we prove the uniqueness of the solution to problem \eqref{3.1}-\eqref{3.6}. Assume that 
$(\mathbf{v}_{f1},p_{f1},\mathbf{v}_{p1},
\mathbf{u}_{p1},\beta_{p1},p_{p1})$
and
$(\mathbf{v}_{f2},p_{f2},\mathbf{v}_{p2},\mathbf{u}_{p2},\beta_{p2},p_{p2})$
are two distinct solutions of problem \eqref{3.1}-\eqref{3.6}. 
Substituting $R_{1}$--$R_{5}$ into \eqref{3.1}-\eqref{3.6} yields the following system:
\begin{align}
&\rho_{f}\big(\partial_{t}\mathbf{v}_{f1}-\partial_{t}\mathbf{v}_{f2},\mathbf{w}_{f}\big)_{\Omega_{f}}
+2\mu_{f}\big(\varepsilon(\mathbf{v}_{f1}-\mathbf{v}_{f2}),\varepsilon(\mathbf{w}_{f})\big)_{\Omega_{f}}
\label{A.4}\\
&\quad- \big(p_{f1}-p_{f2},\nabla\cdot\mathbf{w}_{f}\big)_{\Omega_{f}} +c_{BJS}\langle \mathcal{M}_{f}(\mathbf{v}_{f1}-\mathbf{v}_{f2}),\mathcal{M}_{f}(\mathbf{w}_{f})\rangle_{\Gamma}\nonumber\\
&=-\langle p_{p1}-p_{p2},\mathbf{w}_{f}\cdot\mathbf{n}_{f}\rangle_{\Gamma}
+c_{BJS}\langle \mathcal{M}_{p}(\partial_{t}\mathbf{u}_{p1}-\partial_{t}\mathbf{u}_{p2}),\mathcal{M}_{f}(\mathbf{w}_{f})\rangle_{\Gamma},\nonumber\\ &\big(\nabla\cdot(\mathbf{v}_{f1}-\mathbf{v}_{f2}),q_{f}\big)_{\Omega_{f}}=0.\label{A.5}\\
&\big(\mathbf{v}_{p1}-\mathbf{v}_{p2},\mathbf{z}_{p}\big)_{\Omega_{p}}-\big(\partial_{t}\mathbf{u}_{p1}-\partial_{t}\mathbf{u}_{p2},\mathbf{z}_{p}\big)_{\Omega_{p}}=0,\label{A.6}\\
&\rho_{p}\big(\partial_{t}\mathbf{v}_{p1}-\partial_{t}\mathbf{v}_{p2},\mathbf{w}_{p}\big)_{\Omega_{p}}+2\mu_{p}\big(\varepsilon(\mathbf{u}_{p1}-\mathbf{u}_{p2}),\varepsilon(\mathbf{w}_{p})\big)_{\Omega_{p}}\label{A.7}\\
&\quad-\big(\beta_{p1}-\beta_{p2},\nabla\cdot\mathbf{w}_{p}\big)_{\Omega_{p}}+c_{BJS}\langle\mathcal{M}_{p}(\partial_{t}\mathbf{u}_{p1}-\partial_{t}\mathbf{u}_{p2}),\mathcal{M}_{p}(\mathbf{w}_{p})\rangle_{\Gamma}\nonumber\\
&=-\langle p_{p1}-p_{p2},\mathbf{w}_{p}\cdot\mathbf{n}_{p}\rangle_{\Gamma}
+c_{BJS}\langle \mathcal{M}_{f}(\mathbf{v}_{f1}-\mathbf{v}_{f2}),\mathcal{M}_{p}(\mathbf{w}_{p})\rangle_{\Gamma},\nonumber\\
&\frac{1}{\lambda_{p}}\big(\beta_{p1}-\beta_{p2},\varphi_{p}\big)_{\Omega_{p}}+\big(\nabla\cdot(\mathbf{u}_{p1}-\mathbf{u}_{p2}),\varphi_{p}\big)_{\Omega_{p}}=\frac{\alpha}{\lambda_{p}}\big(p_{p1}-p_{p2},\varphi_{p}\big)_{\Omega_{p}},\label{A.8}
\end{align}
\begin{align}
&\big(c_{0}+\frac{\alpha^{2}}{\lambda_{p}}\big)\big(\partial_{t}p_{p1}-\partial_{t}p_{p2},\psi_{p}\big)_{\Omega_{p}}-\frac{\alpha}{\lambda_{p}}\big(\partial_{t}\beta_{p1}-\partial_{t}\beta_{p2},\psi_{p}\big)_{\Omega_{p}}\label{A.9}\\
&\quad+\big(\mu_{f}^{-1}K\nabla (p_{p1}-p_{p2}),\nabla\psi_{p}\big)_{\Omega_{p}}\nonumber\\
&=\langle (\mathbf{v}_{f1}-\mathbf{v}_{f2})\cdot\mathbf{n}_{f},\psi_{p}\rangle_{\Gamma}+\langle (\partial_{t}\mathbf{u}_{p1}-\partial_{t}\mathbf{u}_{p2})\cdot\mathbf{n}_{p},\psi_{p}\rangle_{\Gamma}\nonumber.
\end{align}

Similar to the inequality \eqref{A.3}, we have the following estimate for \eqref{A.4}-\eqref{A.9}:
\begin{align}
&\frac{\rho_{f}}{2}\|\mathbf{v}_{f1}(s)-\mathbf{v}_{f2}(s)\|_{L^{2}(\Omega_{f})}^{2}+\frac{\rho_{p}}{2}\|\mathbf{v}_{p1}(s)-\mathbf{v}_{p2}(s)\|_{L^{2}(\Omega_{p})}^{2}\label{A.11}\\
&+\mu_{p}\|\varepsilon(\mathbf{u}_{p1}(s)-\mathbf{u}_{p2}(s))\|_{L^{2}(\Omega_{p})}^{2}+\frac{c_{0}}{2}\|p_{p1}(s)-p_{p2}(s)\|_{L^{2}(\Omega_{p})}^{2}\nonumber\\
&\quad+\frac{1}{2\lambda_{p}}\|\alpha p_{p1}(s)-\beta_{p1}(s)-(\alpha p_{p2}(s)-\beta_{p2}(s))\|_{L^{2}(\Omega_{p})}^{2}\nonumber\\
&\quad+\int_{0}^{s} \Big[2\mu_{f}\|\varepsilon(\mathbf{v}_{f1}-\mathbf{v}_{f2})\|_{L^{2}(\Omega_{f})}^{2}+\mu_{f}^{-1}\|K^{\frac{1}{2}}\nabla (p_{p1}-p_{p2})\|_{L^{2}(\Omega_{p})}^{2}\Big]\,dt\leq0.\nonumber
\end{align}
By inequality \eqref{A.11}, we conclude that
$\mathbf{v}_{f1}=\mathbf{v}_{f2}$,
$\mathbf{v}_{p1}=\mathbf{v}_{p2}$,
$\mathbf{u}_{p1}=\mathbf{u}_{p2}$,
$\beta_{p1}=\beta_{p2}$,
and
$p_{p1}=p_{p2}$.
Moreover, invoking the inf-sup condition yields
$p_{f1}=p_{f2}$. The proof is complete.

\bibliographystyle{siamplain}
\bibliography{refs}

\end{document}